\documentclass[12pt]{amsart}

\usepackage{lineno,hyperref}

\usepackage{amsthm,amsmath,amssymb}
\usepackage{graphicx}
\usepackage[utf8]{inputenc}
\usepackage[T1]{fontenc}
\usepackage[toc,page]{appendix} 
\usepackage[english,french]{babel}
\usepackage{enumitem}
\usepackage{amsfonts}
\usepackage{mathtext}
\usepackage{array}
\usepackage{multirow}
\usepackage[font=footnotesize]{caption}
\usepackage{caption}
\usepackage{subcaption}
\usepackage{enumitem}
\usepackage{yhmath}
\usepackage{mathdots}
\usepackage{tikz}
\usepackage{placeins}
\usetikzlibrary{svg.path}
\usetikzlibrary{decorations.text}

\newtheorem{theo}{Theorem}
\newtheorem{conj}[theo]{Conjecture}
\newtheorem{prop}[theo]{Proposition}
\newtheorem{propdef}[theo]{Proposition-Definition}
\newtheorem{lem}[theo]{Lemma}
\newtheorem{cor}[theo]{Corollary}
{\theoremstyle{definition} \newtheorem{defi}[theo]{Definition}}
{\theoremstyle{definition} }
{\theoremstyle{definition} \newtheorem{algo}[theo]{Algorithm}}
{\theoremstyle{definition} }
{\theoremstyle{definition} }
{\theoremstyle{definition} }
{\theoremstyle{remark} \newtheorem{rem}[theo]{Remark}}
{\theoremstyle{definition} }

{\theoremstyle{definition} }
{\theoremstyle{definition} }
{\theoremstyle{definition} }
{\theoremstyle{definition} }
{\theoremstyle{definition} }
{\theoremstyle{remark} }
{\theoremstyle{definition} }

\DeclareMathOperator{\SP}{\mathcal{P}}

\DeclareMathOperator{\SpDC}{\text{SpDC}}
\DeclareMathOperator{\DC}{\text{DC}}

\DeclareMathOperator{\T} {\mathcal{T}}

\DeclareMathOperator{\TT} {\mathfrak{T}}

\DeclareMathOperator{\F} {\mathcal{F}}
\DeclareMathOperator{\Tab} {\mathcal{T}}

\DeclareMathOperator{\SpF}{\text{Sp}\mathcal{F}}
\DeclareMathOperator{\ndf}{\text{ndf}}
\DeclareMathOperator{\ngr}{\text{ng}}

\DeclareMathOperator{\fr}{\text{fr}}

\begin{document}

\title{Enumerating the symplectic Dellac configurations}

\author{Ange Bigeni}

\selectlanguage{english}

\begin{abstract}
Fang and Fourier defined the symplectic Dellac configurations in order to parametrize the torus fixed points of the symplectic degenerated flag varieties, and conjectured that their numbers are the elements of a sequence $(r_n)_{n \geq 0} = (1, 2, 10, 98, 1594, \hdots)$ which appears in the study by Randrianarivony and Zeng of the median Euler numbers. In this paper, we prove the conjecture by considering a combinatorial interpretation of the integers $r_n$ in terms of the surjective pistols (which form a well-known combinatorial model of the Genocchi numbers), and constructing an appropriate surjection from the symplectic Dellac configurations to the surjective pistols.
\end{abstract}

\maketitle

\section*{Notations}
For all pair of integers $(n,m)$ such that $n<m$, the set of integers $\left\{n,n+1,\hdots,m\right\}$ is denoted by $[n,m]$. If $n$ is a positive integer, we denote by $[n]$ the set $[1,n]$. The cardinality of a finite set $S$ is denoted by $\# S$. If a set of integers $\left\{i_1,i_2,\hdots,i_m\right\}$ has the property $i_k < i_{k+1}$ for all $k \in [m-1]$, we denote it by $\left\{i_1,i_2,\hdots,i_m\right\}_<$.

\section{Introduction}

Let $n$ be a positive integer. Recall that a Dellac configuration of size $n$~\cite{Dellac} is a tableau $D$, made of $n$ columns and $2n$ rows, that contains $2n$ dots such that~:
\begin{itemize}
\item every row contains exactly one dot;
\item every column contains exactly two dots;
\item if there is a dot in the box $(j,i)$ of $D$ (\textit{i.e.}, in the intersection of its $j$-th column from left to right and its $i$-th row from bottom to top), then $j \leq i \leq j+n$.
\end{itemize}

The set of the Dellac configurations of size $n$ is denoted by $\DC_n$. For example, in Figure~\ref{fig:DC3} are depicted the 7 elements of $\DC_3$.

\begin{figure}[!htbp]

\begin{center}

\begin{tikzpicture}[scale=0.35]
\draw (0,0) grid[step=1] (3,6);
\draw (0,0) -- (3,3);
\draw (0,3) -- (3,6);
\fill (0.5,0.5) circle (0.2);
\fill (0.5,1.5) circle (0.2);
\fill (1.5,2.5) circle (0.2);
\fill (1.5,3.5) circle (0.2);
\fill (2.5,4.5) circle (0.2);
\fill (2.5,5.5) circle (0.2);

\begin{scope}[xshift=3.5cm]
\draw (0,0) grid[step=1] (3,6);
\draw (0,0) -- (3,3);
\draw (0,3) -- (3,6);
\fill (0.5,0.5) circle (0.2);
\fill (0.5,1.5) circle (0.2);
\fill (1.5,2.5) circle (0.2);
\fill (2.5,3.5) circle (0.2);
\fill (1.5,4.5) circle (0.2);
\fill (2.5,5.5) circle (0.2);
\end{scope}

\begin{scope}[xshift=7cm]
\draw (0,0) grid[step=1] (3,6);
\draw (0,0) -- (3,3);
\draw (0,3) -- (3,6);
\fill (0.5,0.5) circle (0.2);
\fill (0.5,1.5) circle (0.2);
\fill (2.5,2.5) circle (0.2);
\fill (1.5,3.5) circle (0.2);
\fill (1.5,4.5) circle (0.2);
\fill (2.5,5.5) circle (0.2);
\end{scope}

\begin{scope}[xshift=10.5cm]
\draw (0,0) grid[step=1] (3,6);
\draw (0,0) -- (3,3);
\draw (0,3) -- (3,6);
\fill (0.5,0.5) circle (0.2);
\fill (1.5,1.5) circle (0.2);
\fill (0.5,2.5) circle (0.2);
\fill (1.5,3.5) circle (0.2);
\fill (2.5,4.5) circle (0.2);
\fill (2.5,5.5) circle (0.2);
\end{scope}

\begin{scope}[xshift=14cm]
\draw (0,0) grid[step=1] (3,6);
\draw (0,0) -- (3,3);
\draw (0,3) -- (3,6);
\fill (0.5,0.5) circle (0.2);
\fill (1.5,1.5) circle (0.2);
\fill (0.5,2.5) circle (0.2);
\fill (2.5,3.5) circle (0.2);
\fill (1.5,4.5) circle (0.2);
\fill (2.5,5.5) circle (0.2);
\end{scope}

\begin{scope}[xshift=17.5cm]
\draw (0,0) grid[step=1] (3,6);
\draw (0,0) -- (3,3);
\draw (0,3) -- (3,6);
\fill (0.5,0.5) circle (0.2);
\fill (1.5,1.5) circle (0.2);
\fill (1.5,2.5) circle (0.2);
\fill (0.5,3.5) circle (0.2);
\fill (2.5,4.5) circle (0.2);
\fill (2.5,5.5) circle (0.2);
\end{scope}

\begin{scope}[xshift=21cm]
\draw (0,0) grid[step=1] (3,6);
\draw (0,0) -- (3,3);
\draw (0,3) -- (3,6);
\fill (0.5,0.5) circle (0.2);
\fill (1.5,1.5) circle (0.2);
\fill (2.5,2.5) circle (0.2);
\fill (0.5,3.5) circle (0.2);
\fill (1.5,4.5) circle (0.2);
\fill (2.5,5.5) circle (0.2);
\end{scope}
\end{tikzpicture}

\end{center}

\caption{The $h_3 = 7$ elements of $\DC_3$.}
\label{fig:DC3}

\end{figure}
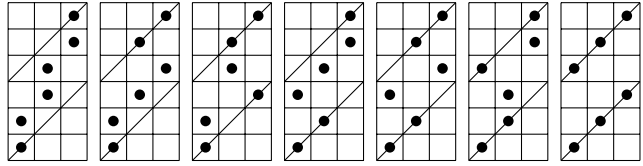

It is well-known~\cite{Feigin2} that the cardinality of $\DC_n$ is $h_n$ where $(h_n)_{n \geq 0} = (1, 1, 2, 7, 38, 295, \hdots)$ is the sequence of the normalized median Genocchi numbers~\cite{normalizedmediangenocchinumbers}. Feigin~\cite{Feigin2,Feigin3} proved that the Poincaré polynomial of the degenerate flag variety $\F^a_n$ has a combinatorial interpretation in terms of the Dellac configurations of size $n$, in particular its Euler characteristic equals $\# \DC_n = h_n$. Afterwards, following computer experiments, Cerulli Irelli and Feigin conjectured that in the case of the symplectic degenerate flag varieties $\SpF^a_{2n}$~\cite{feigin0}, the role of the sequence $(h_n)_{n \geq 0}$ is played by the sequence of positive integers $(r_n)_{n \geq 0} = (1,2,10,98,1594,\hdots)$~\cite{eulermedian} defined by Randrianarivony and Zeng~\cite{rand} following $r_n = D_n(1)/2^n$ where $D_0(x) =~1$ and
\begin{equation*}
D_{n+1}(x) = (x+1)(x+2) D_n(x+2) - x(x+1) D_n(x).
\end{equation*}

Now, Fang and Fourier~\cite{fang} have defined a combinatorial model of the Euler characteristic $\chi(\SpF^a_{2n})$ of the symplectic degenerate flag variety $\SpF^a_{2n}$, through the set $\SpDC_{2n}$ of the symplectic Dellac configurations of size $2n$.

\begin{defi}[Fang and Fourier~\cite{fang}]
A symplectic Dellac configuration of size $2n$ is an element $S$ of $\DC_{2n}$ such that, for all $i \in [4n]$ and $j \in [2n]$, there is a dot in the box $(j,i)$ of $S$ if and only if there is a dot in its box $(2n+1-j,4n+1-i)$ (in other words, there exists a central reflection of $S$ with respect to the center of $S$). The set of the symplectic Dellac configurations of size $2n$ is denoted by $\SpDC_{2n}$.
\end{defi}

For example, in Figure~\ref{fig:SpDC4} are depicted the 10 elements of $\SpDC_4$.

\begin{figure}[!htbp]

\begin{center}

\begin{tikzpicture}[scale=0.3]

\draw (0,0) grid[step=1] (2,8);
\draw (0,0) -- (2,2);
\draw (0,4) -- (2,6);
\fill (0.5,0.5) circle (0.2);
\fill (0.5,1.5) circle (0.2);
\fill (1.5,2.5) circle (0.2);
\fill (1.5,3.5) circle (0.2);
\begin{scope}[shift={(4,8)},rotate=180]
\draw (0,0) grid[step=1] (2,8);
\draw (0,0) -- (2,2);
\draw (0,4) -- (2,6);
\fill (0.5,0.5) circle (0.2);
\fill (0.5,1.5) circle (0.2);
\fill (1.5,2.5) circle (0.2);
\fill (1.5,3.5) circle (0.2);
\end{scope}

\begin{scope}[xshift=5cm]
\draw (0,0) grid[step=1] (2,8);
\draw (0,0) -- (2,2);
\draw (0,4) -- (2,6);
\fill (0.5,0.5) circle (0.2);
\fill (0.5,1.5) circle (0.2);
\fill (1.5,2.5) circle (0.2);
\fill (1.5,4.5) circle (0.2);
\begin{scope}[shift={(4,8)},rotate=180]
\draw (0,0) grid[step=1] (2,8);
\draw (0,0) -- (2,2);
\draw (0,4) -- (2,6);
\fill (0.5,0.5) circle (0.2);
\fill (0.5,1.5) circle (0.2);
\fill (1.5,2.5) circle (0.2);
\fill (1.5,4.5) circle (0.2);
\end{scope}
\end{scope}

\begin{scope}[xshift=10cm]
\draw (0,0) grid[step=1] (2,8);
\draw (0,0) -- (2,2);
\draw (0,4) -- (2,6);
\fill (0.5,0.5) circle (0.2);
\fill (0.5,1.5) circle (0.2);
\fill (1.5,3.5) circle (0.2);
\fill (1.5,5.5) circle (0.2);
\begin{scope}[shift={(4,8)},rotate=180]
\draw (0,0) grid[step=1] (2,8);
\draw (0,0) -- (2,2);
\draw (0,4) -- (2,6);
\fill (0.5,0.5) circle (0.2);
\fill (0.5,1.5) circle (0.2);
\fill (1.5,3.5) circle (0.2);
\fill (1.5,5.5) circle (0.2);
\end{scope}
\end{scope}

\begin{scope}[xshift=15cm]
\draw (0,0) grid[step=1] (2,8);
\draw (0,0) -- (2,2);
\draw (0,4) -- (2,6);
\fill (0.5,0.5) circle (0.2);
\fill (0.5,1.5) circle (0.2);
\fill (1.5,4.5) circle (0.2);
\fill (1.5,5.5) circle (0.2);
\begin{scope}[shift={(4,8)},rotate=180]
\draw (0,0) grid[step=1] (2,8);
\draw (0,0) -- (2,2);
\draw (0,4) -- (2,6);
\fill (0.5,0.5) circle (0.2);
\fill (0.5,1.5) circle (0.2);
\fill (1.5,4.5) circle (0.2);
\fill (1.5,5.5) circle (0.2);
\end{scope}
\end{scope}

\begin{scope}[xshift=20cm]
\draw (0,0) grid[step=1] (2,8);
\draw (0,0) -- (2,2);
\draw (0,4) -- (2,6);
\fill (0.5,0.5) circle (0.2);
\fill (1.5,1.5) circle (0.2);
\fill (0.5,2.5) circle (0.2);
\fill (1.5,3.5) circle (0.2);
\begin{scope}[shift={(4,8)},rotate=180]
\draw (0,0) grid[step=1] (2,8);
\draw (0,0) -- (2,2);
\draw (0,4) -- (2,6);
\fill (0.5,0.5) circle (0.2);
\fill (1.5,1.5) circle (0.2);
\fill (0.5,2.5) circle (0.2);
\fill (1.5,3.5) circle (0.2);
\end{scope}
\end{scope}

\begin{scope}[shift={(0,-9)}]
\draw (0,0) grid[step=1] (2,8);
\draw (0,0) -- (2,2);
\draw (0,4) -- (2,6);
\fill (0.5,0.5) circle (0.2);
\fill (1.5,1.5) circle (0.2);
\fill (0.5,2.5) circle (0.2);
\fill (1.5,4.5) circle (0.2);
\begin{scope}[shift={(4,8)},rotate=180]
\draw (0,0) grid[step=1] (2,8);
\draw (0,0) -- (2,2);
\draw (0,4) -- (2,6);
\fill (0.5,0.5) circle (0.2);
\fill (1.5,1.5) circle (0.2);
\fill (0.5,2.5) circle (0.2);
\fill (1.5,4.5) circle (0.2);
\end{scope}
\end{scope}

\begin{scope}[shift={(5,-9)}]
\draw (0,0) grid[step=1] (2,8);
\draw (0,0) -- (2,2);
\draw (0,4) -- (2,6);
\fill (0.5,0.5) circle (0.2);
\fill (1.5,1.5) circle (0.2);
\fill (1.5,2.5) circle (0.2);
\fill (0.5,3.5) circle (0.2);
\begin{scope}[shift={(4,8)},rotate=180]
\draw (0,0) grid[step=1] (2,8);
\draw (0,0) -- (2,2);
\draw (0,4) -- (2,6);
\fill (0.5,0.5) circle (0.2);
\fill (1.5,1.5) circle (0.2);
\fill (1.5,2.5) circle (0.2);
\fill (0.5,3.5) circle (0.2);
\end{scope}
\end{scope}

\begin{scope}[shift={(10,-9)}]
\draw (0,0) grid[step=1] (2,8);
\draw (0,0) -- (2,2);
\draw (0,4) -- (2,6);
\fill (0.5,0.5) circle (0.2);
\fill (1.5,1.5) circle (0.2);
\fill (2.5,2.5) circle (0.2);
\fill (0.5,3.5) circle (0.2);
\begin{scope}[shift={(4,8)},rotate=180]
\draw (0,0) grid[step=1] (2,8);
\draw (0,0) -- (2,2);
\draw (0,4) -- (2,6);
\fill (0.5,0.5) circle (0.2);
\fill (1.5,1.5) circle (0.2);
\fill (2.5,2.5) circle (0.2);
\fill (0.5,3.5) circle (0.2);
\end{scope}
\end{scope}

\begin{scope}[shift={(15,-9)}]
\draw (0,0) grid[step=1] (2,8);
\draw (0,0) -- (2,2);
\draw (0,4) -- (2,6);
\fill (0.5,0.5) circle (0.2);
\fill (1.5,1.5) circle (0.2);
\fill (1.5,2.5) circle (0.2);
\fill (0.5,4.5) circle (0.2);
\begin{scope}[shift={(4,8)},rotate=180]
\draw (0,0) grid[step=1] (2,8);
\draw (0,0) -- (2,2);
\draw (0,4) -- (2,6);
\fill (0.5,0.5) circle (0.2);
\fill (1.5,1.5) circle (0.2);
\fill (1.5,2.5) circle (0.2);
\fill (0.5,4.5) circle (0.2);
\end{scope}
\end{scope}

\begin{scope}[shift={(20,-9)}]
\draw (0,0) grid[step=1] (2,8);
\draw (0,0) -- (2,2);
\draw (0,4) -- (2,6);
\fill (0.5,0.5) circle (0.2);
\fill (1.5,1.5) circle (0.2);
\fill (2.5,2.5) circle (0.2);
\fill (0.5,4.5) circle (0.2);
\begin{scope}[shift={(4,8)},rotate=180]
\draw (0,0) grid[step=1] (2,8);
\draw (0,0) -- (2,2);
\draw (0,4) -- (2,6);
\fill (0.5,0.5) circle (0.2);
\fill (1.5,1.5) circle (0.2);
\fill (2.5,2.5) circle (0.2);
\fill (0.5,4.5) circle (0.2);
\end{scope}
\end{scope}

\end{tikzpicture}

\end{center}

\caption{The $10$ elements of $\SpDC_4$.}
\label{fig:SpDC4}

\end{figure}
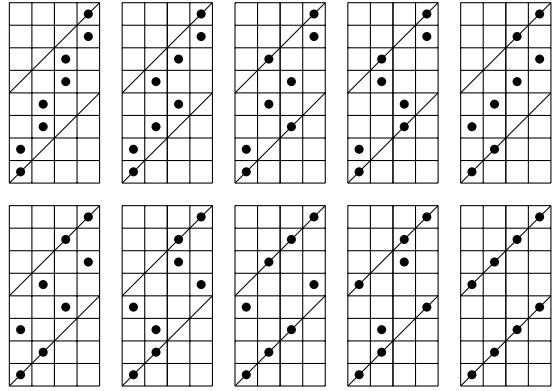

\begin{prop}[Fang and Fourier~\cite{fang}]
For all $n \geq 1$, the Euler characteristic of $\SpF^a_{2n}$ is the cardinality of $\SpDC_{2n}$.
\end{prop}

\begin{conj}[Cerulli Irelli and Feigin, Fang and Fourier~\cite{fang}]
\label{conj:1}
The cardinality of $\SpDC_{2n}$ equals $r_n$ for all $n \geq 1$.
\end{conj}

The aim of this paper is to prove the above conjecture. To do so, we use a combinatorial interpretation of the integers $r_n$ in terms of the surjective pistols. Recall that, for a given $n \geq 1$, a surjective pistol $f \in \SP_n$ is a surjective map $f : [2n] \twoheadrightarrow \left\{2,4,\hdots,2n\right\}$ such that $f(j) \geq j$ for all $j \in [2n]$. For a given element $f \in \SP_n$, an integer $j \in [2n-2]$ is said to be a doubled fixed point if there exists $j' < j$ such that $f(j') = f(j) = j$ (in particular $j$ is even). Let $\ndf(f)$ be the number of elements of $\left\{2,4,\hdots,2n\right\}$ that are not doubled fixed points of $f$ (by definition~$\ndf(f) \geq 1$ because $2n$ is never considered as a doubled fixed point, even though $f(2n-1) = f(2n) = 2n$ for all $f$). From now on, we assimilate every surjective pistol $f \in \SP_n$ into the sequence $(f(1),f(2),\hdots, f(2n))$, in which the images of the even integers that are doubled fixed points (respectively not doubled fixed points) are underlined (respectively written in bold characters). Also, we represent $f \in \SP_n$ by a tableau made of $n$ left-justified rows of length $2,4,6,\hdots,2n$ (from bottom to top) by plotting a dot inside the $\left( f(j)/2-j \right)$-th box (from bottom to top) of the $j$-th column of the tableau for all $j \in [2n]$ ; with precision, if $j$ is an even integer that is not a doubled fixed point of $f$, we plot a symbol $\times$ instead of a dot. For example, we represent in Figure \ref{fig:SP2} the $3$ elements of $\SP_2$, whose numbers of non doubled fixed points are respectively $2$, $1$ and $2$.

\begin{figure}[!htbp]

\begin{center}

\begin{tikzpicture}[scale=0.4]

\draw (0,0) grid[step=1] (2,2);
\draw (2,1) grid[step=1] (4,2);
\fill (0.5,1.5) circle (0.2);
\draw (1.5,0.5) node[scale=1]{$\times$};
\fill (2.5,1.5) circle (0.2);
\draw (3.5,1.5) node[scale=1]{$\times$};

\draw (0.5,2.5) node[scale=1]{$1$};
\draw (1.5,2.5) node[scale=1]{$2$};
\draw (2.5,2.5) node[scale=1]{$3$};
\draw (3.5,2.5) node[scale=1]{$4$};
\draw (-0.5,0.5) node[scale=1]{$2$};
\draw (-0.5,1.5) node[scale=1]{$4$};

\draw (1.75,-1) node[scale=0.8]{$f_1 = (4,\textbf{2},4,\textbf{4})$};

\begin{scope}[shift={(6,0)}]
\draw (0,0) grid[step=1] (2,2);
\draw (2,1) grid[step=1] (4,2);
\fill (0.5,0.5) circle (0.2);
\fill (1.5,0.5) circle (0.2);
\fill (2.5,1.5) circle (0.2);
\draw (3.5,1.5) node[scale=1]{$\times$};

\draw (0.5,2.5) node[scale=1]{$1$};
\draw (1.5,2.5) node[scale=1]{$2$};
\draw (2.5,2.5) node[scale=1]{$3$};
\draw (3.5,2.5) node[scale=1]{$4$};
\draw (-0.5,0.5) node[scale=1]{$2$};
\draw (-0.5,1.5) node[scale=1]{$4$};

\draw (1.75,-1) node[scale=0.8]{$f_2 = (2,\underline{2},4,\textbf{4})$};
\end{scope}

\begin{scope}[shift={(12,0)}]
\draw (0,0) grid[step=1] (2,2);
\draw (2,1) grid[step=1] (4,2);
\fill (0.5,0.5) circle (0.2);
\draw (1.5,1.5) node[scale=1]{$\times$};
\fill (2.5,1.5) circle (0.2);
\draw (3.5,1.5) node[scale=1]{$\times$};

\draw (0.5,2.5) node[scale=1]{$1$};
\draw (1.5,2.5) node[scale=1]{$2$};
\draw (2.5,2.5) node[scale=1]{$3$};
\draw (3.5,2.5) node[scale=1]{$4$};
\draw (-0.5,0.5) node[scale=1]{$2$};
\draw (-0.5,1.5) node[scale=1]{$4$};

\draw (1.75,-1) node[scale=0.8]{$f_3 = (2,\textbf{4},4,\textbf{4})$};
\end{scope}

\end{tikzpicture}

\end{center}
\caption{The $G_6 = 3$ elements of $\SP_2$.}
\label{fig:SP2}

\end{figure}

Randrianarivony and Zeng~\cite{rand} proved the following Formula~for all $n \geq 1$~:

\begin{equation}
\label{eq:formuleRNpistols}
r_n = \sum_{f \in \SP_n} 2^{\ndf(f)}.
\end{equation}
For example, in the case $n=2$, we do obtain $r_2 = 2^2 + 2 + 2^2$ as seen in Figure \ref{fig:SP2}. We know from Dumont~\cite{Dumont2} that the surjective pistols form a combinatorial interpretation of the sequence of the Genocchi numbers $(G_{2k})_{k \geq 1} = (1,1,3,17,155,2073,\hdots)$~\cite{genocchi1}~: for all $n \geq 1$, the cardinality of $\SP_n$ equals $G_{2n+2}$.

Now, we are going to obtain (in Proposition~\ref{prop:cardinalSpDCselonT}) an analogous formula for the cardinality $\SpDC_{2n}$, in terms of the combinatorial objects defined as follows.

\begin{defi}
\label{defi:tableaux}
Let $\Tab_n$ be the set of tableaux $T$ made of $n$ columns and $2n$ rows, that contain $2n$ dots such that~:
\begin{itemize}
\item every row contains exactly one dot;
\item every column contains exactly two dots;
\item if there is a dot in the box $(j,i)$ of $T$, then $j \leq i$.
\end{itemize}

(This is in fact the Definition~of $\DC_n$, minus the condition that each box $(j,i)$ that contains a dot implies $i \leq j+n$.)

If a dot of $T$ is located in a box $(j,i)$ such that $i \geq 2n+1 - j$, we say that it is \textit{free} and we represent it by a star instead of a dot. Let $\fr(T)$ be the number of free dots of $T$.

\end{defi}

For example, in Figure~\ref{fig:T2} are depicted the 3 elements of $\Tab_2$, and their numbers of free dots are respectively 2, 1 and 2.

\begin{figure}[!htbp]

\begin{center}

\begin{tikzpicture}[scale=0.4]

\draw (0,0) grid[step=1] (2,4);
\draw (0,0) -- (2,2);
\draw [dashed] (0,4) -- (2,2);
\fill (0.5,0.5) circle (0.2);
\fill (0.5,1.5) circle (0.2);
\draw (1.5,2.5) node[scale=0.7]{$\bigstar$};
\draw (1.5,3.5) node[scale=0.7]{$\bigstar$};

\begin{scope}[shift={(3,0)}]
\draw (0,0) grid[step=1] (2,4);
\draw (0,0) -- (2,2);
\draw [dashed] (0,4) -- (2,2);
\fill (0.5,0.5) circle (0.2);
\fill (1.5,1.5) circle (0.2);
\fill (0.5,2.5) circle (0.2);
\draw (1.5,3.5) node[scale=0.7]{$\bigstar$};
\end{scope}

\begin{scope}[shift={(6,0)}]
\draw (0,0) grid[step=1] (2,4);
\draw (0,0) -- (2,2);
\draw [dashed] (0,4) -- (2,2);
\fill (0.5,0.5) circle (0.2);
\fill (1.5,1.5) circle (0.2);
\draw (1.5,2.5) node[scale=0.7]{$\bigstar$};
\draw (0.5,3.5) node[scale=0.7]{$\bigstar$};
\end{scope}

\end{tikzpicture}

\end{center}
\caption{The 3 elements of $\Tab_2$.}
\label{fig:T2}

\end{figure}

By considering the $\binom{n+1}{2} = (n+1)n/2$ possible locations of the two dots of the last column of any $T \in \Tab_n$ (from left to right), it is easy to obtain the induction formula $\# \T_n = \binom{n+1}{2} \# \Tab_{n-1}$ for all $n \geq 2$, and finally to compute
\begin{equation*}
\label{eq:cardinalT}
\# \Tab_n = (n+1)!n!/2^n.
\end{equation*}

\begin{prop}
\label{prop:cardinalSpDCselonT}
For all $n \geq 1$, we have
\begin{equation*}
\label{eq:cardinalSpDCselonT}
\# \SpDC_{2n} = \sum_{T \in \Tab_n} 2^{\fr(T)}.
\end{equation*}
\end{prop}

\begin{proof}
Any $S \in \SpDC_{2n}$ can be partitioned as follows,
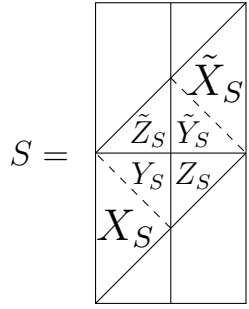
\begin{center}

\begin{tikzpicture}[scale=1]

\draw (-0.75,2) node[scale=1.1] {$S = $};

\draw (0,0) rectangle (2,4);
\draw (0,0) -- (2,2);
\draw (0,2) -- (2,4);
\draw (1,0) -- (1,4);
\draw (0,2) -- (2,2);
\draw [dashed] (1,1) -- (0,2);
\draw [dashed] (1,3) -- (2,2);

\draw (0.4,1) node[scale=1.4] {$X_S$};
\draw (1.6,3) node[scale=1.4] {$\tilde{X}_S$};
\draw (0.7,1.7) node[scale=1]{$Y_S$};
\draw (1.3,2.3) node[scale=1]{$\tilde{Y}_S$};
\draw (0.7,2.3) node[scale=1]{$\tilde{Z}_S$};
\draw (1.3,1.7) node[scale=1]{$Z_S$};


\end{tikzpicture}

\end{center}
where there are no dots in the blank areas, and where the area $\tilde{X}_S$ (respectively $\tilde{Y}_S$ and $\tilde{Z}_S$) is symmetrical to the area $X_S$ (respectively $Y_S$ and $Z_S$) following the center of $S$. Now, for any dot of $Y_S$ (respectively $Z_S$), say, located in the box $(j,i)$ with $2n \geq i \geq 2n+1-j$ (respectively $2n \geq i \geq j$), we can define a new configuration $s_i(S) \in \SpDC_{2n}$ by relocating this dot in the box $(2n+1-j,i)$ of $Z_S$ (respectively $Y_S$), and relocating the dot located in the box $(2n+1-j,4n+1-j)$ of $\tilde{Y}_S$ (respectively $\tilde{Z}_S$), in the box $(j,4n+1-i)$ of $\tilde{Z}_S$ (respectively $\tilde{Y}_S$). Thus, if $E_i$ is defined as the set of the configurations $S \in \SpDC_{2n}$ whose $i$-th row (from bottom to top) doesn't contain its dot in $X_S$, then it is clear that $s_i$ is an involution of $E_i$. Also, if $S \in E_{i_1} \cap E_{i_2}$ for some $(i_1,i_2) \in [n+1,2n]^2$, obviously $s_{i_1} \circ s_{i_2}(S) = s_{i_2} \circ s_{i_1}(S)$. Consequently, for all $S \in \SpDC_{2n}$, there exists one unique $T \in \Tab_{n}$ such that $S$ is obtained by applying a finite number of these involutions on the configuration $S_T \in \SpDC_{2n}$ defined by $Z_{S_T}$ and $\tilde{Z}_{S_T}$ being empty, and

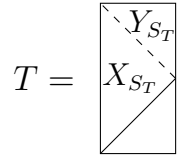
\begin{center}

\begin{tikzpicture}[scale=1]

\draw (-0.75,1) node[scale=1.1] {$T = $};

\draw (0,0) rectangle (1,2);
\draw (1,1) [dashed] -- (0,2);
\draw (0,0) -- (1,1);

\draw (0.4,1) node[scale=1]{$X_{S_T}$};
\draw (0.7,1.7) node[scale=1]{$Y_{S_T}$};

\end{tikzpicture}

\end{center}
so that $S_T$ generates a total amount of $2^{\# Y_{S_T}}$ elements of $\SpDC_{2n}$, where $\# Y_{S_T}$ is the number of dots located in $Y_{S_T}$ (in other words $\# Y_{S_T} = \fr(T)$).
\end{proof}

For example, we depict in Figure~\ref{fig:generationSpDC4} how the $3$ elements of $\Tab_2$ generate the $10 = 2^2 + 2 + 2^2$ elements of $\SpDC_4$.

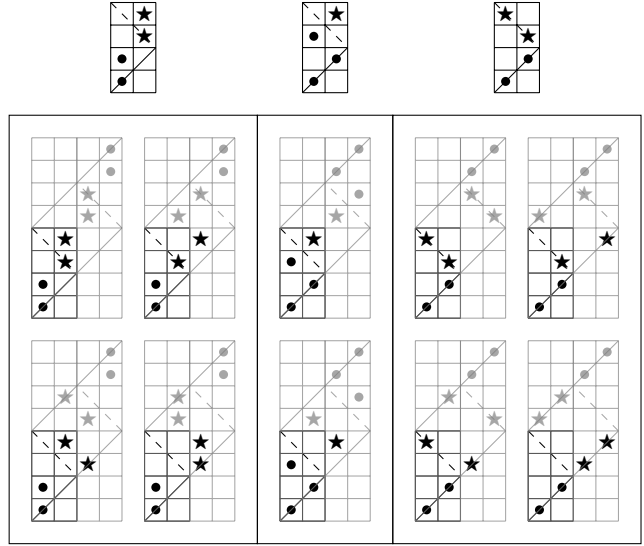
\begin{figure}[!htbp]

\begin{center}

\begin{tikzpicture}[scale=0.3]

\draw (0,0) grid[step=1] (2,4);
\draw (0,0) -- (2,2);
\draw [dashed] (0,4) -- (2,2);
\fill (0.5,0.5) circle (0.2);
\fill (0.5,1.5) circle (0.2);
\draw (1.5,2.5) node[scale=0.6]{$\bigstar$};
\draw (1.5,3.5) node[scale=0.6]{$\bigstar$};

\draw [gray,opacity = 0.7] (0,0) grid[step=1] (2,8);
\draw [gray,opacity = 0.7] (0,0) -- (2,2);
\draw [gray,opacity = 0.7] (0,4) -- (2,6);
\begin{scope}[shift={(4,8)},rotate=180]
\draw [gray,opacity = 0.7] (0,0) grid[step=1] (2,8);
\draw [gray,opacity = 0.7] (0,0) -- (2,2);
\draw [gray,opacity = 0.7] (0,4) -- (2,6);
\draw [gray,opacity = 0.7] [dashed] (0,4) -- (2,2);
\fill [gray,opacity = 0.7] (0.5,0.5) circle (0.2);
\fill [gray,opacity = 0.7] (0.5,1.5) circle (0.2);
\draw [gray,opacity = 0.7] (1.5,2.5) node[scale=0.6]{$\bigstar$};
\draw [gray,opacity = 0.7] (1.5,3.5) node[scale=0.6]{$\bigstar$};
\end{scope}

\begin{scope}[shift={(5,0)}]

\draw (0,0) grid[step=1] (2,4);
\draw (0,0) -- (2,2);
\draw [dashed] (0,4) -- (2,2);
\fill (0.5,0.5) circle (0.2);
\fill (0.5,1.5) circle (0.2);
\draw (1.5,2.5) node[scale=0.6]{$\bigstar$};
\draw (2.5,3.5) node[scale=0.6]{$\bigstar$};

\draw [gray,opacity = 0.7] (0,0) grid[step=1] (2,8);
\draw [gray,opacity = 0.7] (0,0) -- (2,2);
\draw [gray,opacity = 0.7] (0,4) -- (2,6);
\begin{scope}[shift={(4,8)},rotate=180]
\draw [gray,opacity = 0.7] (0,0) grid[step=1] (2,8);
\draw [gray,opacity = 0.7] (0,0) -- (2,2);
\draw [gray,opacity = 0.7] (0,4) -- (2,6);
\draw [gray,opacity = 0.7] [dashed] (0,4) -- (2,2);
\fill [gray,opacity = 0.7] (0.5,0.5) circle (0.2);
\fill [gray,opacity = 0.7] (0.5,1.5) circle (0.2);
\draw [gray,opacity = 0.7] (1.5,2.5) node[scale=0.6]{$\bigstar$};
\draw [gray,opacity = 0.7] (2.5,3.5) node[scale=0.6]{$\bigstar$};
\end{scope}

\end{scope}

\begin{scope}[shift={(0,-9)}]

\draw (0,0) grid[step=1] (2,4);
\draw (0,0) -- (2,2);
\draw [dashed] (0,4) -- (2,2);
\fill (0.5,0.5) circle (0.2);
\fill (0.5,1.5) circle (0.2);
\draw (2.5,2.5) node[scale=0.6]{$\bigstar$};
\draw (1.5,3.5) node[scale=0.6]{$\bigstar$};

\draw [gray,opacity = 0.7] (0,0) grid[step=1] (2,8);
\draw [gray,opacity = 0.7] (0,0) -- (2,2);
\draw [gray,opacity = 0.7] (0,4) -- (2,6);
\begin{scope}[shift={(4,8)},rotate=180]
\draw [gray,opacity = 0.7] (0,0) grid[step=1] (2,8);
\draw [gray,opacity = 0.7] (0,0) -- (2,2);
\draw [gray,opacity = 0.7] (0,4) -- (2,6);
\draw [gray,opacity = 0.7] [dashed] (0,4) -- (2,2);
\fill [gray,opacity = 0.7] (0.5,0.5) circle (0.2);
\fill [gray,opacity = 0.7] (0.5,1.5) circle (0.2);
\draw [gray,opacity = 0.7] (2.5,2.5) node[scale=0.6]{$\bigstar$};
\draw [gray,opacity = 0.7] (1.5,3.5) node[scale=0.6]{$\bigstar$};
\end{scope}

\end{scope}

\begin{scope}[shift={(5,-9)}]

\draw (0,0) grid[step=1] (2,4);
\draw (0,0) -- (2,2);
\draw [dashed] (0,4) -- (2,2);
\fill (0.5,0.5) circle (0.2);
\fill (0.5,1.5) circle (0.2);
\draw (2.5,2.5) node[scale=0.6]{$\bigstar$};
\draw (2.5,3.5) node[scale=0.6]{$\bigstar$};

\draw [gray,opacity = 0.7] (0,0) grid[step=1] (2,8);
\draw [gray,opacity = 0.7] (0,0) -- (2,2);
\draw [gray,opacity = 0.7] (0,4) -- (2,6);
\begin{scope}[shift={(4,8)},rotate=180]
\draw [gray,opacity = 0.7] (0,0) grid[step=1] (2,8);
\draw [gray,opacity = 0.7] (0,0) -- (2,2);
\draw [gray,opacity = 0.7] (0,4) -- (2,6);
\draw [gray,opacity = 0.7] [dashed] (0,4) -- (2,2);
\fill [gray,opacity = 0.7] (0.5,0.5) circle (0.2);
\fill [gray,opacity = 0.7] (0.5,1.5) circle (0.2);
\draw [gray,opacity = 0.7] (2.5,2.5) node[scale=0.6]{$\bigstar$};
\draw [gray,opacity = 0.7] (2.5,3.5) node[scale=0.6]{$\bigstar$};
\end{scope}

\end{scope}

\draw (-1,-10) rectangle (10,9);

\begin{scope}[shift={(3.5,10)}]

\draw (0,0) grid[step=1] (2,4);
\draw (0,0) -- (2,2);
\draw [dashed] (0,4) -- (2,2);
\fill (0.5,0.5) circle (0.2);
\fill (0.5,1.5) circle (0.2);
\draw (1.5,2.5) node[scale=0.6]{$\bigstar$};
\draw (1.5,3.5) node[scale=0.6]{$\bigstar$};

\end{scope}


\begin{scope}[shift={(11,0)}]

\draw (0,0) grid[step=1] (2,4);
\draw (0,0) -- (2,2);
\draw [dashed] (0,4) -- (2,2);
\fill (0.5,0.5) circle (0.2);
\fill (1.5,1.5) circle (0.2);
\fill (0.5,2.5) circle (0.2);
\draw (1.5,3.5) node[scale=0.6]{$\bigstar$};

\draw [gray,opacity = 0.7] (0,0) grid[step=1] (2,8);
\draw [gray,opacity = 0.7] (0,0) -- (2,2);
\draw [gray,opacity = 0.7] (0,4) -- (2,6);
\begin{scope}[shift={(4,8)},rotate=180]
\draw [gray,opacity = 0.7] (0,0) grid[step=1] (2,8);
\draw [gray,opacity = 0.7] (0,0) -- (2,2);
\draw [gray,opacity = 0.7] (0,4) -- (2,6);
\draw [gray,opacity = 0.7] [dashed] (0,4) -- (2,2);
\fill [gray,opacity = 0.7] (0.5,0.5) circle (0.2);
\fill [gray,opacity = 0.7] (1.5,1.5) circle (0.2);
\fill [gray,opacity = 0.7] (0.5,2.5) circle (0.2);
\draw [gray,opacity = 0.7] (1.5,3.5) node[scale=0.6]{$\bigstar$};
\end{scope}

\begin{scope}[shift={(0,-9)}]

\draw (0,0) grid[step=1] (2,4);
\draw (0,0) -- (2,2);
\draw [dashed] (0,4) -- (2,2);
\fill (0.5,0.5) circle (0.2);
\fill (1.5,1.5) circle (0.2);
\fill (0.5,2.5) circle (0.2);
\draw (2.5,3.5) node[scale=0.6]{$\bigstar$};

\draw [gray,opacity = 0.7] (0,0) grid[step=1] (2,8);
\draw [gray,opacity = 0.7] (0,0) -- (2,2);
\draw [gray,opacity = 0.7] (0,4) -- (2,6);
\begin{scope}[shift={(4,8)},rotate=180]
\draw [gray,opacity = 0.7] (0,0) grid[step=1] (2,8);
\draw [gray,opacity = 0.7] (0,0) -- (2,2);
\draw [gray,opacity = 0.7] (0,4) -- (2,6);
\draw [gray,opacity = 0.7] [dashed] (0,4) -- (2,2);
\fill [gray,opacity = 0.7] (0.5,0.5) circle (0.2);
\fill [gray,opacity = 0.7] (1.5,1.5) circle (0.2);
\fill [gray,opacity = 0.7] (0.5,2.5) circle (0.2);
\draw [gray,opacity = 0.7] (2.5,3.5) node[scale=0.6]{$\bigstar$};
\end{scope}

\end{scope}

\draw (-1,-10) rectangle (5,9);

\begin{scope}[shift={(1,10)}]

\draw (0,0) grid[step=1] (2,4);
\draw (0,0) -- (2,2);
\draw [dashed] (0,4) -- (2,2);
\fill (0.5,0.5) circle (0.2);
\fill (1.5,1.5) circle (0.2);
\fill (0.5,2.5) circle (0.2);
\draw (1.5,3.5) node[scale=0.6]{$\bigstar$};

\end{scope}

\end{scope}


\begin{scope}[shift={(17,0)}]

\draw (0,0) grid[step=1] (2,4);
\draw (0,0) -- (2,2);
\draw [dashed] (0,4) -- (2,2);
\fill (0.5,0.5) circle (0.2);
\fill (1.5,1.5) circle (0.2);
\draw (1.5,2.5) node[scale=0.6]{$\bigstar$};
\draw (0.5,3.5) node[scale=0.6]{$\bigstar$};

\draw [gray,opacity = 0.7] (0,0) grid[step=1] (2,8);
\draw [gray,opacity = 0.7] (0,0) -- (2,2);
\draw [gray,opacity = 0.7] (0,4) -- (2,6);
\begin{scope}[shift={(4,8)},rotate=180]
\draw [gray,opacity = 0.7] (0,0) grid[step=1] (2,8);
\draw [gray,opacity = 0.7] (0,0) -- (2,2);
\draw [gray,opacity = 0.7] (0,4) -- (2,6);
\draw [gray,opacity = 0.7] [dashed] (0,4) -- (2,2);
\fill [gray,opacity = 0.7] (0.5,0.5) circle (0.2);
\fill [gray,opacity = 0.7] (1.5,1.5) circle (0.2);
\draw [gray,opacity = 0.7] (1.5,2.5) node[scale=0.6]{$\bigstar$};
\draw [gray,opacity = 0.7] (0.5,3.5) node[scale=0.6]{$\bigstar$};
\end{scope}

\begin{scope}[shift={(0,-9)}]

\draw (0,0) grid[step=1] (2,4);
\draw (0,0) -- (2,2);
\draw [dashed] (0,4) -- (2,2);
\fill (0.5,0.5) circle (0.2);
\fill (1.5,1.5) circle (0.2);
\draw (2.5,2.5) node[scale=0.6]{$\bigstar$};
\draw (0.5,3.5) node[scale=0.6]{$\bigstar$};

\draw [gray,opacity = 0.7] (0,0) grid[step=1] (2,8);
\draw [gray,opacity = 0.7] (0,0) -- (2,2);
\draw [gray,opacity = 0.7] (0,4) -- (2,6);
\begin{scope}[shift={(4,8)},rotate=180]
\draw [gray,opacity = 0.7] (0,0) grid[step=1] (2,8);
\draw [gray,opacity = 0.7] (0,0) -- (2,2);
\draw [gray,opacity = 0.7] (0,4) -- (2,6);
\draw [gray,opacity = 0.7] [dashed] (0,4) -- (2,2);
\fill [gray,opacity = 0.7] (0.5,0.5) circle (0.2);
\fill [gray,opacity = 0.7] (1.5,1.5) circle (0.2);
\draw [gray,opacity = 0.7] (2.5,2.5) node[scale=0.6]{$\bigstar$};
\draw [gray,opacity = 0.7] (0.5,3.5) node[scale=0.6]{$\bigstar$};
\end{scope}

\end{scope}

\begin{scope}[shift={(5,0)}]

\draw (0,0) grid[step=1] (2,4);
\draw (0,0) -- (2,2);
\draw [dashed] (0,4) -- (2,2);
\fill (0.5,0.5) circle (0.2);
\fill (1.5,1.5) circle (0.2);
\draw (1.5,2.5) node[scale=0.6]{$\bigstar$};
\draw (3.5,3.5) node[scale=0.6]{$\bigstar$};

\draw [gray,opacity = 0.7] (0,0) grid[step=1] (2,8);
\draw [gray,opacity = 0.7] (0,0) -- (2,2);
\draw [gray,opacity = 0.7] (0,4) -- (2,6);
\begin{scope}[shift={(4,8)},rotate=180]
\draw [gray,opacity = 0.7] (0,0) grid[step=1] (2,8);
\draw [gray,opacity = 0.7] (0,0) -- (2,2);
\draw [gray,opacity = 0.7] (0,4) -- (2,6);
\draw [gray,opacity = 0.7] [dashed] (0,4) -- (2,2);
\fill [gray,opacity = 0.7] (0.5,0.5) circle (0.2);
\fill [gray,opacity = 0.7] (1.5,1.5) circle (0.2);
\draw [gray,opacity = 0.7] (1.5,2.5) node[scale=0.6]{$\bigstar$};
\draw [gray,opacity = 0.7] (3.5,3.5) node[scale=0.6]{$\bigstar$};
\end{scope}

\end{scope}

\begin{scope}[shift={(5,-9)}]

\draw (0,0) grid[step=1] (2,4);
\draw (0,0) -- (2,2);
\draw [dashed] (0,4) -- (2,2);
\fill (0.5,0.5) circle (0.2);
\fill (1.5,1.5) circle (0.2);
\draw (2.5,2.5) node[scale=0.6]{$\bigstar$};
\draw (3.5,3.5) node[scale=0.6]{$\bigstar$};

\draw [gray,opacity = 0.7] (0,0) grid[step=1] (2,8);
\draw [gray,opacity = 0.7] (0,0) -- (2,2);
\draw [gray,opacity = 0.7] (0,4) -- (2,6);
\begin{scope}[shift={(4,8)},rotate=180]
\draw [gray,opacity = 0.7] (0,0) grid[step=1] (2,8);
\draw [gray,opacity = 0.7] (0,0) -- (2,2);
\draw [gray,opacity = 0.7] (0,4) -- (2,6);
\draw [gray,opacity = 0.7] [dashed] (0,4) -- (2,2);
\fill [gray,opacity = 0.7] (0.5,0.5) circle (0.2);
\fill [gray,opacity = 0.7] (1.5,1.5) circle (0.2);
\draw [gray,opacity = 0.7] (2.5,2.5) node[scale=0.6]{$\bigstar$};
\draw [gray,opacity = 0.7] (3.5,3.5) node[scale=0.6]{$\bigstar$};
\end{scope}

\end{scope}

\draw (-1,-10) rectangle (10,9);

\begin{scope}[shift={(3.5,10)}]

\draw (0,0) grid[step=1] (2,4);
\draw (0,0) -- (2,2);
\draw [dashed] (0,4) -- (2,2);
\fill (0.5,0.5) circle (0.2);
\fill (1.5,1.5) circle (0.2);
\draw (1.5,2.5) node[scale=0.6]{$\bigstar$};
\draw (0.5,3.5) node[scale=0.6]{$\bigstar$};

\end{scope}

\end{scope}

\end{tikzpicture}

\end{center}
\caption{Generation of the $2^2 + 2 + 2^2$ elements of $\SpDC_4$ from the $3$ elements of $\Tab_2$.}
\label{fig:generationSpDC4}

\end{figure}

Now, Conjecture~\ref{conj:1} is a corollary of the following Theorem~in view of Formula~(\ref{eq:formuleRNpistols}) and Proposition~\ref{prop:cardinalSpDCselonT}.

\begin{theo}
\label{theo:bigeni}
There exists a surjective map $\varphi : \Tab_n \twoheadrightarrow \SP_n$ such that
\begin{equation}
\label{eq:lasommesurlesTdunmemeF}
\sum_{T \in \varphi^{-1}(f)} 2^{\fr(T)} = 2^{\ndf(f)}
\end{equation}
for all $f \in \SP_n$.
\end{theo}

The rest of this paper aims at proving Theorem~\ref{theo:bigeni}, and is organized as follows. In Section \ref{sec:Tpaths}, we introduce the $j$-tableaux (a generalization of the tableaux $T \in \T_n$), on which we define a family of paths, namely, the $T$-paths. In Section \ref{sec:fromtableauxtopistols}, we use these paths to define the pistol labeling of a tableau (in Algorithm \ref{algo:pistollabeling}), which produces (in Definition~\ref{defi:varphi}) the Definition~of $\varphi$. In Section \ref{sec:injection}, we define the notion of $(f,j)$-insertion of a dot into a $j$-tableau, which allows to formulate Algorithm \ref{algo:insertionlabeling} and produces the Definition~of a map $\phi : \SP_n \rightarrow \T_n$. In Section \ref{sec:connexion}, we first prove that $\varphi \circ \phi$ is the identity map of $\SP_n$ (hence $\phi : \SP_n \rightarrow \T_n$ is injective and $\varphi : \T_n \rightarrow \SP_n$ is surjective), then we make the image $\phi(\SP_n) \subset \T_n$ of $\phi$ explicit, and we prove that $\phi \circ \varphi_{|\phi(\SP_n)}$ is the identity map of this set. Finally, in Section \ref{sec:proof}, we finish the proof of Theorem~\ref{theo:bigeni}, \textit{i.e.}, we show that Formula~(\ref{eq:lasommesurlesTdunmemeF}) is true for all $f \in \SP_n$. To do so, we make $\varphi^{-1}(f)$ explicit by defining Algorithm \ref{algo:switch} and Algorithm \ref{algo:mute}, which allow to construct every element of $\varphi^{-1}(f)$ from one given element of it (like $\phi(f)$).

\section{$j$-tableaux and $T$-paths}
\label{sec:Tpaths}

\begin{defi}
\label{defi:partialtableaux}
Let $j \in [n]$, a \textit{$j$-tableau} $T \in \Tab_n^j$ is a tableau made of $n$ columns (denoted by $C_1^T,C_2^T,\hdots, C_n^T$ from left to right) and $2n$ rows (denoted by $R_1^T,R_2^T,R_n^T,R_{2n-1}^T, R_{2n-2}^T, \hdots, R_{n+1}^T, R_{2n}^T$ from bottom to top), that contain between $2j$ and $2n$ dots above the line $y = x$ (for all $i \in [2n]$, if the row $R_i^T$ contains a dot, it is denoted by $d_i^T$) such that :
\begin{itemize}
\item each column $C_1^T,C_2^T,\hdots,C_{j-1}^T$ contains exactly two dots and the other columns contain at most two dots;
\item each row $R_1^T,R_2^T,\hdots,R_{j-1}^T$ contains exactly one dot and the other rows contain at most one dot.
\end{itemize} 
\end{defi}

In particular, a tableau $T \in \Tab_n$ is also a $j$-tableau for all $j \in [n]$.

\begin{defi}
\label{defi:TPath}
Let $j \in [n]$, $T \in \Tab_n^j$ and $i \in [j,2n]$ such that the intersection of the row $R_i^T$ with the columns $C_1^T,C_2^T,\hdots,C_{j-1}^T$ is empty. The \textit{$T$-path from the box $C_j^T \cap R_i^T$} is the sequence $(i_0,i_1,\hdots) \in [2n]^{\mathbb{N}}$ defined by $i_0 = i$, and, for all $k \in \mathbb{N}$, by the following rules.

\begin{enumerate}
\item If $i_k \in [j,n] \sqcup [n+j,2n]$, then $i_{k+1} = i_k$.
\item If $i_k$ is of the kind $n+j_k$ with $j_k \in [j-1]$, then $d_{i_{k+1}}^T$ is defined as the upper dot of the column $C_{j_k}^T$.
\item Otherwise $i_k \in [j-1]$, and $d_{i_{k+1}}^T$ is defined as the lower dot of the column $C_{i_k}^T$.
\end{enumerate}

\end{defi}

\begin{rem}
In the context of Definition~\ref{defi:TPath}, the $T$-path from the box $C_j^T \cap R_i^T$ becomes stationary if and only if $i_k \in [j,n] \sqcup [n+j,2n]$ for $k$ big enough.
\end{rem}

\begin{propdef}
\label{prop:TPathstationnaire}
With the notations of Definition~\ref{defi:TPath}, the $T$-path $(i_0,i_1,\hdots)$ from the box $C_j^T \cap R_i^T$ becomes stationary, \textit{i.e.}, the integer $i_k$ belongs to the set $[j,n] \sqcup [n+j,2n]$ for $k$ big enough. This integer is said to be the \underline{arrival} of this $T$-path. Also, let $\pi_j^T$ be the map that maps every integer $i \in [j,2n]$ such that the first $j-1$ boxes of $R_i^T$ are empty, to the arrival of the $T$-path from the box $C_j^T \cap R_i^T$. Then $\pi_j^T$ is bijective.
\end{propdef}

\begin{proof}
Suppose that $i_k \not\in [j,n] \sqcup [n+j,2n]$ for all $k \geq 0$. Since $[2n]$ is a finite set, and because $(i_k)_{k \geq 0}$ is defined by induction, there exists $0 \leq k_1 < k_2$ such that $i_{k_1} = i_{k_2}$. Now, Rule (1) of Definition~\ref{defi:TPath} is never applied, so the sequence $(i_k)_{k \geq 0}$ is reversible : for all $k > 0$, let $j_{k-1} \in [n]$ such that $d_{i_k}^T \in C_{j_{k-1}}^T$ ; if $d_{i_k}^T$ is the upper dot of $C_{j_k}^T$, then $i_{k-1} = n+j_{k-1}$, otherwise $i_{k-1} = j_{k-1}$. Consequently, the equality $i_{k_1} = i_{k_2}$ implies $i = i_0 = i_{k_2-k_1}$. Since $k_2 - k_1 > 0$ and, for all $k > 0$, the dot $d_{i_k}^T$ belongs to a column $C_{j_{k-1}}^T$ for some $j_{k-1} \in [j-1]$, then $d_i^T$ cannot belong to $C_j^T$, which is absurd. So $i_k \in [j,n] \sqcup [n+j,2n]$ for $k$ big enough. 

Let $k_{\min}$ be the smallest integer $k \geq 0$ such that $i_k$ is the arrival of the $T$-path. As stated before, the sequence $(i_0,i_1,\hdots,i_{k_{\min}})$ is reversible because it never involves Rule (1) of Definition~\ref{defi:TPath}, so the application $\pi_j^T$ is injective. Finally, the number of integers $i \in [j,2n]$ such that the first $j-1$ boxes of $R_i^T$ are empty, is exactly $2(n-j+1) = \#([j,n] \sqcup [n+j,2n])$~: by definition~of $\Tab_j^n$, the first $j-1$ rows of $T$ contain exactly $j-1$ dots, and the first $j-1$ columns of $T$ contain exactly $2(j-1)$ dots, so, among the $2n-j+1$ rows $R_j^T,R_{j+1}^T,\hdots,R_{2n}^T$, exactly $j-1$ of them contain their dot in one of their $j-1$ first boxes. So $\pi_j^T$ is bijective.
\end{proof}

\begin{rem}
The fixed points of $\pi_j^T$ are the integers $i \in [j,n] \sqcup [j+n,2n]$ such that the first $j-1$ boxes of $R_i^T$ are empty.
\end{rem}

For example (in this case $n = 7$ and $j = 4$), consider the $4$-tableau $T_0 \in \Tab_7^4$ which appears in Figure~\ref{fig:T0path}. In this example the columns $C_4^{T_0},C_5^{T_0},C_6^{T_0},C_7^{T_0}$ are empty. The set of integers $i \in [j,n] \sqcup [n+j,2n] = [4,7] \sqcup [11,14]$ such that the $j-1 = 3$ first boxes of $R_i^{T_0}$ are empty, is $\left\{4,5,7,8,9,11,12,13\right\}$. For all $i \in \left\{4,5,7,11,12,13\right\}$, we obtain $\pi_4^{T_0}(i) = i$ because $i \in [j,n] \sqcup [n+j,2n]$. In Figure~\ref{fig:T0path}, we show that the $T_0$-path from the box $C_4^{T_0} \cap R_8^{T_0}$ (respectively $C_4^{T_0} \cap R_9^{T_0}$) is the sequence $(i_k)_{k \geq 0} = (8,2,10,6,6,6,\hdots)$, which becomes stationary at $i_3 = 6$, element of $[j,n] \sqcup [n+j,2n]$, hence $\pi_4^{T_0}(8) = 6$ (respectively the sequence $(i'_k)_{k \geq 0} = (9,14,14,14,\hdots)$, which becomes stationary at $i'_1 = 14 \in [j,n] \sqcup [n+j,2n]$, hence $\pi_4^{T_0}(9) = 14$). As a summary, we obtain 
$$\pi_4^{T_0} = \begin{pmatrix}
4 & 5 & 7 & 8 & 9 & 11 & 12 & 13\\
4 & 5 & 7 & 6 & 14 & 11 & 12 & 13
\end{pmatrix}.$$ 

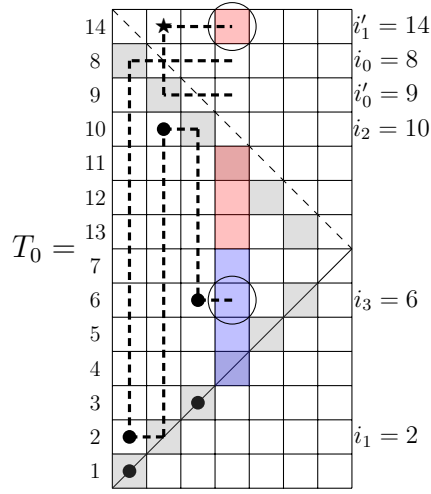
\begin{figure}[!htbp]

\begin{center}

\begin{tikzpicture}[scale=0.455]

\draw (-2,7) node[scale=1]{$T_0 = $};

\draw (0,0) grid[step=1] (7,14);
\draw (0,0) -- (7,7);
\draw (0,14) [dashed] -- (7,7);
\fill (0.5,0.5) circle (0.2);
\fill (0.5,1.5) circle (0.2);
\fill (1.5,10.5) circle (0.2);
\draw (1.5,13.5) node[scale=0.6]{$\bigstar$};
\fill (2.5,2.5) circle (0.2);
\fill (2.5,5.5) circle (0.2);

\draw (-0.5,0.5) node[scale=0.7]{$1$};
\draw (-0.5,1.5) node[scale=0.7]{$2$};
\draw (-0.5,2.5) node[scale=0.7]{$3$};
\draw (-0.5,3.5) node[scale=0.7]{$4$};
\draw (-0.5,4.5) node[scale=0.7]{$5$};
\draw (-0.5,5.5) node[scale=0.7]{$6$};
\draw (-0.5,6.5) node[scale=0.7]{$7$};
\draw (-0.5,12.5) node[scale=0.7]{$8$};
\draw (-0.5,11.5) node[scale=0.7]{$9$};
\draw (-0.5,10.5) node[scale=0.7]{$10$};
\draw (-0.5,9.5) node[scale=0.7]{$11$};
\draw (-0.5,8.5) node[scale=0.7]{$12$};
\draw (-0.5,7.5) node[scale=0.7]{$13$};
\draw (-0.5,13.5) node[scale=0.7]{$14$};

\draw [fill=gray,opacity=0.25] (0,12) rectangle (1,13);
\draw [fill=gray,opacity=0.25] (1,11) rectangle (2,12);
\draw [fill=gray,opacity=0.25] (2,10) rectangle (3,11);
\draw [fill=gray,opacity=0.25] (3,9) rectangle (4,10);
\draw [fill=gray,opacity=0.25] (4,8) rectangle (5,9);
\draw [fill=gray,opacity=0.25] (5,7) rectangle (6,8);

\draw [fill=gray,opacity=0.25] (0,0) rectangle (1,1);
\draw [fill=gray,opacity=0.25] (1,1) rectangle (2,2);
\draw [fill=gray,opacity=0.25] (2,2) rectangle (3,3);
\draw [fill=gray,opacity=0.25] (3,3) rectangle (4,4);
\draw [fill=gray,opacity=0.25] (4,4) rectangle (5,5);
\draw [fill=gray,opacity=0.25] (5,5) rectangle (6,6);

\draw [fill=blue,opacity=0.25] (3,3) rectangle (4,7);
\draw [fill=red,opacity=0.25] (3,7) rectangle (4,10);
\draw [fill=red,opacity=0.25] (3,13) rectangle (4,14);

\draw [very thick] (3.5,12.5) [densely dashed] -- (0.5,12.5);
\draw [very thick] (0.5,12.5) [densely dashed] -- (0.5,1.5);
\draw [very thick] (0.5,1.5) [densely dashed] -- (1.5,1.5);
\draw [very thick] (1.5,1.5) [densely dashed] -- (1.5,10.5);
\draw [very thick] (1.5,10.5) [densely dashed] -- (2.5,10.5);
\draw [very thick] (2.5,10.5) [densely dashed] -- (2.5,5.5);
\draw [very thick] (2.5,5.5) [densely dashed] -- (3.5,5.5);
\draw (3.5,5.5) circle (0.7);

\draw [very thick] (3.5,11.5) [densely dashed] -- (1.5,11.5);
\draw [very thick] (1.5,11.5) [densely dashed] -- (1.5,13.5);
\draw [very thick] (1.5,13.5) [densely dashed] -- (3.5,13.5);
\draw (3.5,13.5) circle (0.7);

\draw (8,12.5) node[scale=0.8]{$i_0 = 8$};
\draw (8,1.5) node[scale=0.8]{$i_1 = 2$};
\draw (8.15,10.5) node[scale=0.8]{$i_2 = 10$};
\draw (8,5.5) node[scale=0.8]{$i_3 = 6$};
\draw (8,11.5) node[scale=0.8]{$i'_0 = 9$};
\draw (8.15,13.5) node[scale=0.8]{$i'_1 = 14$};

\end{tikzpicture}

\end{center}

\caption{$T_0$-paths $(8,2,10,6,6,6,\hdots)$ and $(9,14,14,14,\hdots)$ from the respective boxes $C_4^{T_0} \cap R_8^{T_0}$ and $C_4^{T_0} \cap R_9^{T_0}$.}
\label{fig:T0path}

\end{figure}

\section{From the tableaux to the surjective pistols}

\label{sec:fromtableauxtopistols}

\subsection{Pistol labeling of a tableau}
\label{sec:pistollabeling}

Let $T \in \Tab_n$. We consider a vectorial version of the statistic of free dots $\fr : \Tab_n \rightarrow [n]$, through $\overrightarrow{\fr} : \Tab_n \rightarrow \left\{0,1\right\}^n$ defined by $$\overrightarrow{\fr}(T) = [\fr_1(T),\fr_2(T),\hdots,\fr_n(T)]$$ where $\fr_i(T) = 1$ if and only if the dot $d_{n+i}^T$ is free.

We are going to give (in Algorithm \ref{algo:pistollabeling}) three labels to every dot of~$T$~:
\begin{itemize}
\item a \textit{digital} label, \textit{i.e.}, an element of $[0,n-1]$;
\item a \textit{type} label, \textit{i.e.}, either the letter $\alpha$ or $\beta$;
\item a \textit{parity} label, \textit{i.e.}, either the letter $o$ (for \textit{odd}) or $e$ (for \textit{even}).
\end{itemize}

If a dot $d$ is labeled with the type label $t \in \left\{\alpha,\beta\right\}$, the digital label $h \in [0,n-1]$ and the parity label $p \in \left\{o,e\right\}$, we denote the data of these three labels by $t_h^p$, and we name it the \textit{pistol label} of $d$. Sometimes, we will also write that $d$ is labeled with $t_h$ if we know its digital label $h$ and its type label $t$ but not its parity label.

\begin{defi}
\label{defi:twindots}
Let $T \in \Tab_n$ and $i \in [n]$. The dots $d_i^T$ and $d_{n+i}^T$ are said to be \textit{twin dots}. Let $j_1$ and $j_2$ such that $d_i^T \in C_{j_1}^T$ and $d_{n+i}^T \in C_{j_2}^T$. The dot $d_{i,\min}^T$ is defined as $d_i^T$ if $j_1 \leq j_2$, as $d_{n+i}^T$ otherwise.
\end{defi}

\begin{algo}[pistol labeling of a tableau]
\label{algo:pistollabeling}
For $j$ from $n$ down to $1$, assume that each of the $2(n-j)$ dots of the columns $C_{j+1}^T,\hdots,C_n^T$ have already received its pistol label. At this step, in the parts I., II. and III., we give every dot of $C_j^T$ its digital, type and parity label respectively.

\begin{enumerate}[label=\Roman*.]

\item \underline{The digital labels.} For all $i \in [j,2n]$, if the dot $d_i^T$ belongs to $C_j^T$, let $i' = \pi_j^T(i) \in [j,n] \sqcup [n+j,2n]$. We define the digital label of $d_i^T$ as $i'-j$ if $i' \in [j,n]$, as $i'-n-j$ if $i' \in [n+j,2n]$.

\item \underline{The type labels.} For all $i \in [j,2n]$, if the dot $d_i^T$ belongs to $C_j^T$, let $h \in [0,n-j]$ be its digital label. We consider $j' = j+h \in [j,n]$, and $i' = \pi_j^T(i) \in \left\{j',n+j'\right\}$.

\begin{enumerate}[label = \arabic* -]

\item Assume first that $j'>j$. By hypothesis, the two dots of $C_{j'}^T$ have already received their pistol labels. If they have different type labels, we define $(\gamma,\bar{\gamma})$ as $(\alpha,\beta)$, otherwise we define $(\gamma,\bar{\gamma})$ as $(\beta,\alpha)$. 
\begin{enumerate}[label=\alph*)]
\item If one of the dots of $C_{j'}^T$ is labeled with $\beta_0^e$, then we define the type label of $d_i^T$ as $\alpha$ if $i'=j'$, as $\beta$ if $i' =n+j'$.
\item Otherwise, we define the type label of $d_i^T$ as $\gamma$ if $d_{i'}^T = d_{j',\min}^T$, as $\bar{\gamma}$ otherwise.
\end{enumerate}

\item If $j'=j$, let $d \neq d_i^T$ be the other dot of $C_j^T$.

\begin{enumerate}[label=\alph*)]
\item If the digital label of $d$ is $0$, then we define the type label of $d_i^T$ as $\alpha$ if $i'=j$, as $\beta$ if $i'=n+j$.

\item Otherwise, the type label $t$ of $d$ has already been defined by Rule II.1- of this algorithm.

\begin{enumerate}[leftmargin=0.3cm,label=\roman*.]
\item If $t = \alpha$, we define the type label of $d_i^T$ as $\alpha$ if $i \neq i'$ and $i'=j$, as $\beta$ otherwise.

\item If $t = \beta$, we define the type label of $d_i^T$ as $\alpha$ if $d_{i'}^T = d_{j,\min}^T$, as $\beta$ otherwise.
\end{enumerate}

\end{enumerate}
\end{enumerate}

\item \underline{The parity labels.} Let $h_1 \leq h_2$ be the digital labels of the dots of $C_j^T$.

\begin{enumerate}[label=\arabic* -]
\item If the type labels of the dots of $C_j^T$ are different, we label with $o$ the dot whose type label is $\alpha$, and with $e$ the dot whose type label is $\beta$.
\item Otherwise, it is necessary that $h_1 \neq h_2$ (if $h_1 = h_2$, then the type labels of the dots of $C_j^T$ are defined by Rule II.1- or Rule II.2-a) of this algorithm, and in both case the type labels are different).
\begin{enumerate}[label=\alph*)]
\item If the type label of the dots of $C_j^T$ is $\alpha$, we label with $e$ the dot whose digital label is $h_1$, and with $o$ the other dot.
\item If they both have the type label $\beta$, we label with $o$ the dot whose digital label is $h_1$, and with $e$ the other dot.
\end{enumerate}
\end{enumerate}
\end{enumerate}

\end{algo}

For example, we depict in Figure \ref{fig:T0} the pistol labeling of a tableau $T_1 \in \Tab_7$. On the left of this figure appears the tableau $T_1$ per se (and we specified on the left the indices of its rows, and on the right its vector statistic $\overrightarrow{\fr}(T_1) = [1,1,0,0,1,1,1]$) ; on the right appears its pistol-labeled version. The details of this pistol labeling are given in Appendix \ref{annex:pistollabelingT1}.

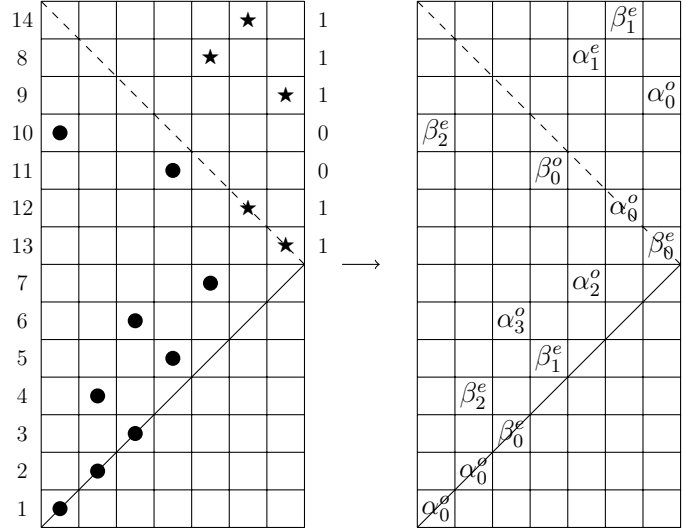
\begin{figure}[!htbp]

\begin{center}

\begin{tikzpicture}[scale=0.5]

\draw (0,0) grid[step=1] (7,14);
\draw (0,0) -- (7,7);
\draw (0,14) [dashed] -- (7,7);
\fill (0.5,0.5) circle (0.2);
\fill (0.5,10.5) circle (0.2);
\fill (1.5,1.5) circle (0.2);
\fill (1.5,3.5) circle (0.2);
\fill (2.5,2.5) circle (0.2);
\fill (2.5,5.5) circle (0.2);
\fill (3.5,4.5) circle (0.2);
\fill (3.5,9.5) circle (0.2);
\fill (4.5,6.5) circle (0.2);
\draw (4.5,12.5) node[scale=0.6]{$\bigstar$};
\draw (5.5,8.5) node[scale=0.6]{$\bigstar$};
\draw (5.5,13.5) node[scale=0.6]{$\bigstar$};
\draw (6.5,7.5) node[scale=0.6]{$\bigstar$};
\draw (6.5,11.5) node[scale=0.6]{$\bigstar$};

\draw (-0.5,0.5) node[scale=0.7]{$1$};
\draw (-0.5,1.5) node[scale=0.7]{$2$};
\draw (-0.5,2.5) node[scale=0.7]{$3$};
\draw (-0.5,3.5) node[scale=0.7]{$4$};
\draw (-0.5,4.5) node[scale=0.7]{$5$};
\draw (-0.5,5.5) node[scale=0.7]{$6$};
\draw (-0.5,6.5) node[scale=0.7]{$7$};
\draw (-0.5,12.5) node[scale=0.7]{$8$};
\draw (-0.5,11.5) node[scale=0.7]{$9$};
\draw (-0.5,10.5) node[scale=0.7]{$10$};
\draw (-0.5,9.5) node[scale=0.7]{$11$};
\draw (-0.5,8.5) node[scale=0.7]{$12$};
\draw (-0.5,7.5) node[scale=0.7]{$13$};
\draw (-0.5,13.5) node[scale=0.7]{$14$};

\draw (7.5,13.5) node[scale=0.7]{$1$};
\draw (7.5,12.5) node[scale=0.7]{$1$};
\draw (7.5,11.5) node[scale=0.7]{$1$};
\draw (7.5,10.5) node[scale=0.7]{$0$};
\draw (7.5,9.5) node[scale=0.7]{$0$};
\draw (7.5,8.5) node[scale=0.7]{$1$};
\draw (7.5,7.5) node[scale=0.7]{$1$};

\draw [->] (8,7) -- (9,7);

\begin{scope}[xshift=10cm]
\draw (0,0) grid[step=1] (7,14);
\draw (0,0) -- (7,7);
\draw (0,14) [dashed] -- (7,7);
\draw (0.5,0.5) node[scale=0.9]{$\alpha_0^o$};
\draw (0.5,10.5) node[scale=0.9]{$\beta_2^e$};
\draw (1.5,1.5) node[scale=0.9]{$\alpha_0^o$};
\draw (1.5,3.5) node[scale=0.9]{$\beta_2^e$};
\draw (2.5,2.5) node[scale=0.9]{$\beta_0^e$};
\draw (2.5,5.5) node[scale=0.9]{$\alpha_3^o$};
\draw (3.5,4.5) node[scale=0.9]{$\beta_1^e$};
\draw (3.5,9.5) node[scale=0.9]{$\beta_0^o$};
\draw (4.5,6.5) node[scale=0.9]{$\alpha_2^o$};
\draw (4.5,12.5) node[scale=0.9]{$\alpha_1^e$};
\draw (5.5,8.5) node[scale=0.9]{$\alpha_0^o$};
\draw (5.5,13.5) node[scale=0.9]{$\beta_1^e$};
\draw (6.5,7.5) node[scale=0.9]{$\beta_0^e$};
\draw (6.5,11.5) node[scale=0.9]{$\alpha_0^o$};
\end{scope}

\end{tikzpicture}

\end{center}

\caption{Tableau $T_1 \in \Tab_7$ (on the left) mapped to its pistol labeling (on the right).}
\label{fig:T0}

\end{figure}

\begin{rem}
\label{rem:aproposdupistollabeling}
We enumerate here a few facts about the pistol labeling of~$T \in \T_n$.
\begin{enumerate}[label=(\alph*)]
\item For all $j \in [n]$ and $i \in [j,2n]$, if the dot $d_i^T \in C_j^T$, then its digital label belongs to the set $[0,n-j]$.
\item If a dot $d_i^T$ in a column $C_j^T$ is labeled with $\alpha_0^e$, then by Rule III.2-a) of Algorithm \ref{algo:pistollabeling} the other dot of $C_j^T$ has the pistol label $\alpha_h^o$ for some $h \in [n-j]$. Also, the type label $\alpha$ of $d_i^T$ has necessarily been defined by Rule II.2-b)i., and in particular $i \in [n+1,n+j-1]$.
\item Every column of $T$ contains exactly one dot whose parity label is $o$ (respectively $e$).
\item By Rule II.2-a) and Rule III.1- of Algorithm \ref{algo:pistollabeling}, the pistol labels of the two dots of $C_n^T$ are $\alpha_0^o$ and $\beta_0^e$. Consequently, the type label of $d_n^T$ (respectively $d_{2n}^T$) is defined either by Rule II.1-a) or Rule II.2-a), and in either case it is $\alpha$ (respectively $\beta$). Also, whether its parity label is defined by Rule III.1- or Rule III.2-a) (respectively Rule III.1- or Rule III.2-b)), it equals $o$ (respectively $e$), and its pistol label is $\alpha_{n-j}^o$ (respectively $\beta_{n-j}^e$) where $C_j^T$ is the column that contains $d_n^T$ (respectively $d_{2n}^T$).
\item For all $i \in [n]$, if one of the dots of $C_i^T$ has the pistol label $\beta_0^e$, then the other dot of $C_i^T$ has the type label $\alpha$ (otherwise the parity labels of the two dots of $C_i^T$ would have been defined by Rule III.2-b) of Algorithm \ref{algo:pistollabeling}, following which the digital label of the dot labeled with $e$ cannot be $0$).
\item For all $j \in [n]$ and $i \in [j,2n]$, if $d_i^T \in C_j^T$, let $h \in [0,n-j]$ be its digital label and $j'= j+h$, then $$i \in \left\{ \left(\pi_j^T\right)^{-1}(j'), \left(\pi_j^T\right)^{-1}(n+j') \right\}.$$ Consequently, if the dots of $C_j^T$ are $d_{\left(\pi_j^T\right)^{-1}(j')}^T$ and $d_{\left(\pi_j^T\right)^{-1}(n+j')}^T$, then the twin dots $d_{j'}^T$ and $d_{n+j'}^T$ are located in the columns $C_1^T,C_2^T,\hdots,C_j^T$.
\item For all $i \in [n]$, if no dot of $C_i^T$ is labeled with $\beta_0^e$, then the type label of $d_{i,\min}^T$ is $\alpha$ if the dots of $C_i^T$ have different type labels, otherwise it is $\beta$.
\end{enumerate}
\end{rem}

\begin{defi}
Following Remark~13.(c), for all $T \in \T_n$ and $j \in [n]$, we define the \textit{odd} dot (respectively \textit{even} dot) of $C_j^T$ as the dot whose parity label is $o$ (respectively $e$).
\end{defi}

\subsection{A map from the tableaux to the surjective pistols}

\begin{defi}[Map $\varphi : \T_n \rightarrow \left\{2,4,\hdots,2n\right\}^{[2n]}$]
\label{defi:varphi}
Let $T \in \T_n$, we define a map $\varphi(T) : [2n] \rightarrow \left\{2,4,\hdots,2n\right\}$ as follows : for all $j \in [n]$, following Remark~\ref{rem:aproposdupistollabeling}.(a) and 13.(c), let $t_o \in \left\{\alpha,\beta\right\}$ and $h_o \in [0,n-j]$ (respectively $t_e \in \left\{\alpha,\beta\right\}$ and $h_e \in [0,n-j]$) be the type and digital labels of the odd dot (respectively even dot) of $C_j^T$. We first define $\varphi(T)(2j-1)$ as $2(j+h_o)$. Afterwards,
\begin{itemize}
\item if $t_e = \alpha$ and $h_e = 0$, we also define $\varphi(T)(2j)$ as $2(j+h_0)$;
\item otherwise, we define $\varphi(T)(2j)$ as $2(j+h_e)$.
\end{itemize}
\end{defi}

\begin{lem}
\label{lem:ouestdmin}
Let $T \in \Tab_n$, $f = \varphi(T)$ and $i \in [n]$. If the dot $d_{i,\min}^T$ is located in the column $C_{j}^T$, then there exists $k \in \left\{2j-1,2j\right\}$ such that $f(k) = 2i$, and $j$ is the integer $j_{\min} = \lceil k_{min} / 2 \rceil$ where
$$k_{\min} = \min \left\{k \in [2i] : f(k) = 2i\right\}.$$
\end{lem}

\begin{proof}
Since $d_{i,\min}^T = d_i^T$ or $d_{n+i}^T$, by Part I. of Algorithm \ref{algo:pistollabeling} its digital label is $i-j$. Consequently, either $d_{i,\min}^T$ is the odd dot of $C_{j}^T$, in which case $\varphi(T)(2j-1) = 2i$, or it is the even dot and $\varphi(T)(2j) = 2i$ because $d_{i,\min}^T$ cannot be labeled with $\alpha_0^e$ in view of Remark~\ref{rem:aproposdupistollabeling}.(b). In either case there exists  $k \in \left\{2j-1,2j\right\}$ such that $f(k) = 2i$, so $j \geq j_{\min}$. Reciprocally, since $f(k_{\min}) = 2i$, by Definition~\ref{defi:varphi} one dot of $C_{j_{\min}}^T$ has the digital label $i-j_{\min}$. By Definition~\ref{defi:TPath}, this implies that there exists $j' \leq j_{\min}$ such that $d_i^T \in C_{j'}^T$ or $d_{n+i}^T \in C_{j'}^T$, hence $j \leq j' \leq j_{\min}$, and $j = j_{\min}$.
\end{proof}

\begin{cor}
In particular, for all $T \in \Tab_n$, the map $\varphi(T)$ is surjective, thus belongs to $\SP_n$.
\end{cor}

For example, the tableau $T_1 \in \T_7$ depicted in Figure \ref{fig:T0} provides the surjective pistol $f_1 = (2,\textbf{6},4,\textbf{8},12,\underline{6},8,\textbf{10},14,\textbf{12},12,\textbf{14},14,\textbf{14}) \in \SP_7$ (whose vector statistic is $\overrightarrow{\ndf}(f_1) = [1,1,0,1,1,1,1]$) depicted in Figure~\ref{fig:f0}.

\begin{figure}[!htbp]

\begin{center}

\begin{tikzpicture}[scale=0.45]

\draw [->] (-3,3.5) -- (-2,3.5);

\begin{scope}[xshift=-11cm,yshift=-3.5cm]
\draw (0,0) grid[step=1] (7,14);
\draw (0,0) -- (7,7);
\draw (0,14) [dashed] -- (7,7);
\draw (0.5,0.5) node[scale=0.9]{$\alpha_0^o$};
\draw (0.5,10.5) node[scale=0.9]{$\beta_2^e$};
\draw (1.5,1.5) node[scale=0.9]{$\alpha_0^o$};
\draw (1.5,3.5) node[scale=0.9]{$\beta_2^e$};
\draw (2.5,2.5) node[scale=0.9]{$\beta_0^e$};
\draw (2.5,5.5) node[scale=0.9]{$\alpha_3^o$};
\draw (3.5,4.5) node[scale=0.9]{$\beta_1^e$};
\draw (3.5,9.5) node[scale=0.9]{$\beta_0^o$};
\draw (4.5,6.5) node[scale=0.9]{$\alpha_2^o$};
\draw (4.5,12.5) node[scale=0.9]{$\alpha_1^e$};
\draw (5.5,8.5) node[scale=0.9]{$\alpha_0^o$};
\draw (5.5,13.5) node[scale=0.9]{$\beta_1^e$};
\draw (6.5,7.5) node[scale=0.9]{$\beta_0^e$};
\draw (6.5,11.5) node[scale=0.9]{$\alpha_0^o$};
\end{scope}

\draw (0,0) grid[step=1] (2,7);
\draw (2,1) grid[step=1] (4,7);
\draw (4,2) grid[step=1] (6,7);
\draw (6,3) grid[step=1] (8,7);
\draw (8,4) grid[step=1] (10,7);
\draw (10,5) grid[step=1] (12,7);
\draw (12,6) grid[step=1] (14,7);

\fill (0.5,0.5) circle (0.2);
\draw (1.5,2.5) node[scale=1]{$\times$};
\fill (2.5,1.5) circle (0.2);
\draw (3.5,3.5) node[scale=1]{$\times$};
\fill (4.5,5.5) circle (0.2);
\fill (5.5,2.5) circle (0.2);
\fill (6.5,3.5) circle (0.2);
\draw (7.5,4.5) node[scale=1]{$\times$};
\fill (8.5,6.5) circle (0.2);
\draw (9.5,5.5) node[scale=1]{$\times$};
\fill (10.5,5.5) circle (0.2);
\draw (11.5,6.5) node[scale=1]{$\times$};
\fill (12.5,6.5) circle (0.2);
\draw (13.5,6.5) node[scale=1]{$\times$};

\draw (-0.5,0.5) node[scale=0.9]{$2$};
\draw (-0.5,1.5) node[scale=0.9]{$4$};
\draw (-0.5,2.5) node[scale=0.9]{$6$};
\draw (-0.5,3.5) node[scale=0.9]{$8$};
\draw (-0.5,4.5) node[scale=0.9]{$10$};
\draw (-0.5,5.5) node[scale=0.9]{$12$};
\draw (-0.5,6.5) node[scale=0.9]{$14$};

\draw (0.5,7.5) node[scale=0.9]{$1$};
\draw (1.5,7.5) node[scale=0.9]{$2$};
\draw (2.5,7.5) node[scale=0.9]{$3$};
\draw (3.5,7.5) node[scale=0.9]{$4$};
\draw (4.5,7.5) node[scale=0.9]{$5$};
\draw (5.5,7.5) node[scale=0.9]{$6$};
\draw (6.5,7.5) node[scale=0.9]{$7$};
\draw (7.5,7.5) node[scale=0.9]{$8$};
\draw (8.5,7.5) node[scale=0.9]{$9$};
\draw (9.5,7.5) node[scale=0.9]{$10$};
\draw (10.5,7.5) node[scale=0.9]{$11$};
\draw (11.5,7.5) node[scale=0.9]{$12$};
\draw (12.5,7.5) node[scale=0.9]{$13$};
\draw (13.5,7.5) node[scale=0.9]{$14$};

\end{tikzpicture}

\end{center}
\caption{The pistol-labeled version of $T_1 \in \Tab_7$ (on the left) is mapped by $\varphi$ to the surjective pistol $f_1 = (2,\textbf{6},4,\textbf{8},12,\underline{6},8,\textbf{10},14,\textbf{12},12,\textbf{14},14,\textbf{14}) \in \SP_7$ (on the right).}
\label{fig:f0}

\end{figure}

We now introduce a vectorial version of the statistic of non-doubled fixed points $\ndf : \DC_n \rightarrow [n]$ through $\overrightarrow{\ndf} : \DC_n \rightarrow \left\{0,1\right\}^n$ defined by $$\overrightarrow{\ndf}(f) = [\ndf_1(f),\ndf_2(f),\hdots,\ndf_n(f)]$$ where $\ndf_i(f) = 1$ if and only if $2i$ is not a doubled fixed point of $f \in \SP_n$.

In the example of $T_1 \in \Tab_7$ and $f_1 = \varphi(T_1) \in \SP_7$, note that $$\overrightarrow{\ndf}(f_1) = [1,1,0,1,1,1,1] \neq [1,1,0,0,1,1,1] = \overrightarrow{\fr}(T_1).$$ In order to define a statistic on tableaux that would be preserved by $\varphi$, we introduce the notion of grounded dots hereafter.

\begin{defi}
Let $T \in \Tab_n$ and $i \in [n]$. We say that the dot $d_{n+i}^T$ is \textit{grounded} if it is not free and if one of the dots of the column $C_i^T$ has the pistol label $\beta_0^e$. Let $\overrightarrow{\ngr}(T) = [\ngr_1(T),\ngr_2(T),\hdots,\ngr_n(T)]$ where $\ngr_i(T) \in \left\{0,1\right\}$ equals $0$ if and only if $d_{n+i}^T$ is grounded for all $i \in [n]$.
\end{defi}

For example, consider the tableau $T_1 \in \Tab_7$ depicted in Figure \ref{fig:T0}, we depict in Figure \ref{fig:T0grounded} the pistol labeling of $T_1$ in which every non-grounded dot has been encircled, which gives $$\overrightarrow{\ngr}(T_1) = [1,1,0,\textbf{1},1,1,1] = \overrightarrow{\ndf}(f_1).$$
Note that in general the dot $d_{2n}^T$, always being free, is never grounded, even though the column $C_n^T$ always has a dot labeled with $\beta_0^e$ (which is similar to $2n$ never being considered as a doubled fixed point of $f \in \SP_n$ even though $f(2n-1) = f(2n) = 2n$).

\begin{figure}[!htbp]

\begin{center}

\begin{tikzpicture}[scale=0.45]

\draw (0,0) grid[step=1] (7,14);
\draw (0,0) -- (7,7);
\draw (0,14) [dashed] -- (7,7);
\draw (0.5,0.5) node[scale=0.8]{$\alpha_0^o$};
\draw (0.5,10.5) node[scale=0.8]{$\beta_2^e$};
\draw (1.5,1.5) node[scale=0.8]{$\alpha_0^o$};
\draw (1.5,3.5) node[scale=0.8]{$\beta_2^e$};
\draw (2.5,2.5) node[scale=0.8]{$\beta_0^e$};
\draw (2.5,5.5) node[scale=0.8]{$\alpha_3^o$};
\draw (3.5,4.5) node[scale=0.8]{$\beta_1^e$};
\draw (3.5,9.5) node[scale=0.8]{$\beta_0^o$};
\draw (4.5,6.5) node[scale=0.8]{$\alpha_2^o$};
\draw (4.5,12.5) node[scale=0.8]{$\alpha_1^e$};
\draw (5.5,8.5) node[scale=0.8]{$\alpha_0^o$};
\draw (5.5,13.5) node[scale=0.8]{$\beta_1^e$};
\draw (6.5,7.5) node[scale=0.8]{$\beta_0^e$};
\draw (6.5,11.5) node[scale=0.8]{$\alpha_0^o$};

\draw (5.5,13.5) circle (0.45);
\draw (4.5,12.5) circle (0.45);
\draw (6.5,11.5) circle (0.45);
\draw (3.5,9.5) circle (0.45);
\draw (5.5,8.5) circle (0.45);
\draw (6.5,7.5) circle (0.45);

\draw (-0.5,0.5) node[scale=0.7]{$1$};
\draw (-0.5,1.5) node[scale=0.7]{$2$};
\draw (-0.5,2.5) node[scale=0.7]{$3$};
\draw (-0.5,3.5) node[scale=0.7]{$4$};
\draw (-0.5,4.5) node[scale=0.7]{$5$};
\draw (-0.5,5.5) node[scale=0.7]{$6$};
\draw (-0.5,6.5) node[scale=0.7]{$7$};
\draw (-0.5,12.5) node[scale=0.7]{$8$};
\draw (-0.5,11.5) node[scale=0.7]{$9$};
\draw (-0.5,10.5) node[scale=0.7]{$10$};
\draw (-0.5,9.5) node[scale=0.7]{$11$};
\draw (-0.5,8.5) node[scale=0.7]{$12$};
\draw (-0.5,7.5) node[scale=0.7]{$13$};
\draw (-0.5,13.5) node[scale=0.7]{$14$};

\draw (7.5,13.5) node[scale=0.7]{$1$};
\draw (7.5,12.5) node[scale=0.7]{$1$};
\draw (7.5,11.5) node[scale=0.7]{$1$};
\draw (7.5,10.5) node[scale=0.7]{$0$};
\draw (7.5,9.5) node[scale=0.7]{$\textbf{1}$};
\draw (7.5,8.5) node[scale=0.7]{$1$};
\draw (7.5,7.5) node[scale=0.7]{$1$};

\end{tikzpicture}

\end{center}

\caption{Pistol labeling of the tableau $T_1 \in \Tab_7$.}
\label{fig:T0grounded}

\end{figure}
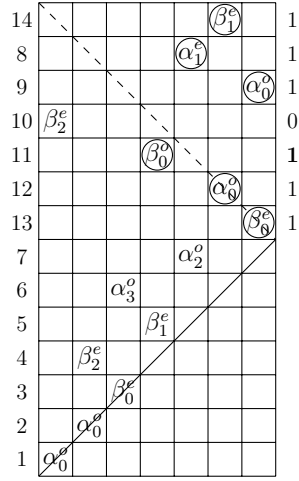

\begin{lem}
\label{lem:pointfixesietseulementsibeta0e}
Let $T \in \T_n$ and $f = \varphi(T) \in \SP_n$. For all $i \in [n]$, the integer $2i$ is a fixed point of $f$ if and only if $C_i^T$ has a dot labeled with~$\beta_0^e$.
\end{lem}

\begin{proof}
If $C_i^T$ has a dot labeled with $\beta_0^e$, by Definition~\ref{defi:varphi} we have $f(2i) = 2i$. Reciprocally, suppose that no dot of $C_i^T$ has the pistol label $\beta_0^e$, and that $f(2i) = 2i$. If the digital label of the even dot of $C_i^T$ was $h_e > 0$, we would have $f(2i) = 2(i+h_e) > 2i$, so its digital label is necessarily $h_e = 0$, and since its pistol label is not $\beta_0^e$ by hypothesis, then it must be $\alpha_0^e$. In view of Remark~\ref{rem:aproposdupistollabeling}(b), this implies that the other dot of $C_i^T$ has the pistol label $\alpha_{h_o}^o$ for some $h_o >0$. We then have $f(2i) = 2(i+h_o) > 2i$, which is absurd.
\end{proof}

\begin{prop}
\label{prop:egalitedesstats}
Let $f \in \SP_n$ and $T \in \varphi^{-1}(f) \subset \T_n$. We have $\overrightarrow{\ngr}(T) = \overrightarrow{\ndf}(f).$
\end{prop}

\begin{proof}
Let $i \in [n]$. If $d_{n+i}^T$ is a grounded dot, then in particular the even dot of $C_i^T$ has the pistol label $\beta_0^e$, so $f(2i) = 2i$ by Lemma~\ref{lem:pointfixesietseulementsibeta0e}. Also, the dot $d_{n+i}^T$ is not free, \textit{i.e.}, it is located in a column $C_j^T$ with $j \leq i$. Let also $j' \leq j$ such that $d_{i,\min}^T \in C_{j'}^T$. By Lemma~\ref{lem:ouestdmin}, there exists $k \in \left\{2j'-1,2j'\right\}$ such that $f(2k) = 2i$. Consequently, if $j' < i$, then $2i$ is a doubled fixed point of $f$. Otherwise, we have $j'= j = i$, so the two dots of $C_i^T$ are $d_i^T$ and $d_{n+i}^T$, in which case it is straightforward from Algorithm \ref{algo:pistollabeling} that their pistol labels are respectively $\alpha_0^o$ and $\beta_0^e$, thus $f(2i-1) = f(2i) = 2i$, and $2i$ is still a doubled fixed point of $f$.

Reciprocally, if $d_{n+i}^T$ is not a grounded dot, then either it is free, or the even dot of $C_i^T$ is not labeled with $\beta_0^e$. If the even dot of $C_i^T$ is not labeled with $\beta_0^e$, then $2i$ is not a fixed point of $f$ by Lemma~\ref{lem:pointfixesietseulementsibeta0e}, in particular it is not a doubled fixed point. Assume now that the even dot of $C_i^T$ is labeled with $\beta_0^e$ but that $d_{n+i}^T$ is free (which implies that $d_{i,\min}^T = d_i^T$). We have $f(2i) = 2i$ by Lemma~\ref{lem:pointfixesietseulementsibeta0e}. By Remark~\ref{rem:aproposdupistollabeling}(f), since $d_{n+i}^T$ is free, the dot of $C_i^T$ labeled with $\beta_0^e$ is $d_{\left(\pi_i^T\right)^{-1}(i)}^T$. This forces its label $\beta$ to have been defined by Rule II.2-b) of Algorithm \ref{algo:pistollabeling}, and in view of Remark~\ref{rem:aproposdupistollabeling}(e), it has been defined with precision by Rule II.2-b)i., which implies in this situation that $d_i^T \in C_i^T$. As a summary, the two dots of $C_i^T$ are $d_{i,\min}^T = d_i^T$ (whose pistol label is $\beta_0^e$), and another dot whose pistol label is $\alpha_k^o$ with $k \neq 0$. By Definition~\ref{defi:varphi}, we then have $f(2i) = 2i$ and $f(2i-1) = 2(i+k) > 2i$. Also, since $d_{i,\min}^T = d_i^T$, by Lemma~\ref{lem:ouestdmin} we know that $\min \left\{k \in [2i] : f(k) = 2i\right\}$ belongs to $\left\{2i-1,2i\right\}$, so it is $2i$, which is consequently a fixed point of $f$ but not a doubled fixed point.
\end{proof}

\section{From the surjective pistols to the tableaux}
\label{sec:injection}

\subsection{Insertion labels and $(f,j)$-insertions}

\begin{defi}[Insertion of a dot into a $j$-tableau]
\label{defi:insertiondansuntableau}
Let $f \in \SP_n$, $j \in [n]$ and $T \in \T_n^j$. We consider $i \in [j,n] \sqcup [n+j,2n]$, and $i' = \left(\pi_j^T\right)^{-1}(i) \in [j,2n]$. Following Proposition-Definition~\ref{prop:TPathstationnaire}, the $j-1$ first boxes of $R_{i'}^T$ are empty. Now, if the box $C_j^T \cap R_{i'}^T$ is also empty, we define a new $j$-tableau by plotting a dot in this box. This operation is called the \textit{insertion of a dot into the box $C_j^T \cap R_{i}^T$}.
\end{defi}

For example (in this case $n = 7$ and $j = 4$), the insertion of a dot into the box $C_4^{T_0} \cap R_6^{T_0}$ (respectively the box $C_4^{T_0} \cap R_{14}^{T_0}$) of the $4$-tableau $T_0 \in \T_7^4$ depicted in Figure \ref{fig:T0}, leads to plotting a dot in the box $C_4^{T_0} \cap R_{8}^{T_0}$ (respectively the box $C_4^{T_0} \cap R_{9}^{T_0}$).

\begin{defi}[labeled $j$-tableaux]
\label{defi:labeledjtableaux}
Let $j \in [n]$, we denote by $\TT_n^j$ the set of tableaux $T \in \T_n^j$ whose dots are labeled with the letter $a$ or $b$, whose columns $C_{j+1}^T,\hdots,C_n^T$ are empty, and whose column $C_j^T$ contains at most one dot.
\end{defi}

\begin{defi}[$(f,j)$-insertion of labeled dots into a labeled $j$-tableau]
\label{defi:deltainsertions}
Let $j \in [n]$ and $T \in \TT_n^j$. We consider $l \in \left\{a,b\right\}$ and $h \in [0,n-j]$. The \textit{$(f,j)$-insertion in $T$ of a dot labeled with $l$ at the height $h$} consists of the following. Let $i = j+ h \in [j,n]$.
\begin{enumerate}[label=\arabic*.]
\item Suppose that $i = j$.
\begin{enumerate}
\item If $R_j^T$ is empty, we insert a dot labeled with $l$ in the box $C_j^T \cap R_j^T$.
\item Otherwise, we insert a dot labeled with $l$ in the box $C_j^T \cap R_j^T$ (respectively $C_j^T \cap R_{n+j}^T$) if $l = a$ (respectively $l = b$).
\end{enumerate}
\item Suppose that $i > j$.
\begin{enumerate}
\item If $R_i^T$ is empty,
\begin{enumerate}[label=\roman*.]
\item if $l = b$ and $f(2i) = 2i$, then we insert a dot labeled with $l$ in the box $C_j^T \cap R_{n+i}^T$;
\item otherwise, we insert a dot labeled with $l$ in the box $C_j^T \cap R_i^T$.
\end{enumerate}
\item Otherwise, let $l' \in \left\{a,b\right\}$ be the label of the dot of $R_i^T$. If $l = l'$ (respectively $l \neq l'$), then we insert a dot labeled with $l$ in the box $C_j^T \cap R_i^T$ (respectively $C_j^T \cap R_{n+i}^T$).
\end{enumerate}
\end{enumerate}

\end{defi}

\subsection{A map from the surjective pistols to the tableaux}

Let $f \in \SP_n$, and $T^1 \in \TT_n^1$ be the empty labeled $1$-tableau. For $j$ from $1$ to~$n$, we are going to define (in Algorithm \ref{algo:insertionlabeling}) a labeled $(j+1)$-tableau $T^{j+1} \in \TT_n^{j+1}$ by filling $C_j^{T^j}$ with two dots located above the line $y = x$, and labeled with the letter $a$ or $b$.

\begin{algo}
\label{algo:insertionlabeling}

For $j$ from $1$ to $n$, we consider the induction hypothesis $H(j)$ defined as follows.
\begin{enumerate}[label=(\Alph*)]
\item $T^j \in \TT_n^j$.
\item If the row $R_j^{T^j}$ is empty and $f(2j) > 2j$, then $f(2k) \neq 2j$ for all $k \in [2j-2]$ (hence $f(2j-1) = 2j$ because $f$ is surjective).
\end{enumerate}

Hypothesis $H(1)$ is obviously true and we initiate the following algorithm for $j = 1$. Let $(\delta_o,\delta_e) = (f(2j-1)/2-j,f(2j)/2-j) \in [0,n-j]^2$.

\begin{enumerate}[label=\Roman*.]

\item We define first two labels $l_o$ and $l_e$ as follows.

\begin{enumerate}[label=\arabic* -]

\item If the row $R_j^{T^j}$ is empty, let $(l_o,l_e) = (a,b)$.

\item Otherwise, let $d$ be the dot of $R_j^{T^j}$.
\begin{enumerate}[label=\alph*)]
\item If $d$ is labeled with $a$, let $(l_o,l_e) = (a,b)$.
\item Else,
\begin{enumerate}[label=\roman*.]
\item if $\delta_o < \delta_e$, let $(l_o,l_e) = (b,b)$;
\item if $\delta_o \geq \delta_e$, let $(l_o,l_e) = (a,a)$.
\end{enumerate}
\end{enumerate}

\end{enumerate}

\item Then, we define two heights $(h_o,h_e) \in [0,n-j]^2$ as follows. The height $h_o$ is defined as $\delta_o$. Afterwards,

\begin{enumerate}[label=\arabic* -]

\item if $l_e = a$ and $\delta_o = \delta_e$, we define $h_e$ as $0$;

\item otherwise, we define $h_e$ as $\delta_e$.

\end{enumerate}

\end{enumerate}

We finally define $T^{j+1}$ as the tableau obtained first by $(f,j)$-inserting in $T^j$ a dot labeled with $l_o$ at the height $h_o$, then by $(f,j)$-inserting in the resulted tableau a dot labeled with $l_e$ at the level $h_e$. We prove now that Hypothesis $H(j+1)$ is true.

\begin{enumerate}[label=(\Alph*)]
\item Following the condition (A) of Hypothesis $H(j)$, since $T^{j+1}$ is obtained by plotting two dots in $C_j^{T^j}$ and in two empty rows of $T^j$, we only need to prove that $R_j^{T^{j+1}}$ contains a dot. Either $R_j^{T^j}$ contains a dot, in which case $R_j^{T^{j+1}}$ too, or, following the condition (B) of Hypothesis $H(j)$, we have $\delta_o = 0$ or $\delta_e = 0$, hence $h_o = 0$ or $h_e = 0$, which implies that the box $C_j^{T^{j+1}} \cap R_j^{T^{j+1}}$ contains a dot by Rule 1.(a) of Definition~\ref{defi:deltainsertions}. So $T^{j+1} \in \TT_n^{j+1}$.
\item If $2j+2$ is not a fixed point of $f$ and if there exists $k \in [2j]$ that is mapped to $2j+2$ by $f$, suppose that $k$ is the smallest integer to have that property and let $j' = \lceil k/2 \rceil \leq j$ ; at the $j'$-th step of the algorithm, a dot is $(f,j')$-inserted in $T^{j'}$ at the level $h = j+1-j'$. Since $k$ is minimal, the row $R_{j+1}^{T^{j'}}$ is empty, so the box $C_{j'}^{T^{j'+1}} \cap R_{j+1}^{T^{j'+1}}$ contains a dot by Rule 2.(a)ii. of Definition~\ref{defi:deltainsertions}.
\end{enumerate}

So the above algorithm is well-defined and, following Hypothesis $H(n+1)$, produces a tableau $T^{n+1} \in \TT_n^{n+1}$, in other words, a tableau $T \in \Tab_n$ whose dots are labeled with the letter $a$ or $b$. We define $\Phi(f)$ as this tableau $T \in \Tab_n$. 

\end{algo}

For example, consider the surjective pistol $$f_1 = (2,\textbf{6},4,\textbf{8},12,\underline{6},8,\textbf{10},14,\textbf{12},12,\textbf{14},14,\textbf{14}) \in \SP_7,$$ whose graphical representation is depicted in Figure \ref{fig:f0}. We depict in Figure \ref{fig:Phif0} the insertion-labeled version of the tableau $\Phi(f_1) \in \Tab_7$, which is in fact the tableau $T_1 \in \Tab_7$ depicted in Figure \ref{fig:T0}, mapped to $f_1$ by $\varphi$ (see Figure \ref{fig:f0}). The details of this computation are given in Appendix~\ref{annex:insertionlabelingf1}.

\begin{figure}[!htbp]

\begin{center}

\begin{tikzpicture}[scale=0.4]

\draw (0,0) grid[step=1] (7,14);
\draw (0,0) -- (7,7);
\draw (0,14) [dashed] -- (7,7);
\draw (0.5,0.5) node[scale=1]{$a$};
\draw (0.5,10.5) node[scale=1]{$b$};
\draw (1.5,1.5) node[scale=1]{$a$};
\draw (1.5,3.5) node[scale=1]{$b$};
\draw (2.5,2.5) node[scale=1]{$b$};
\draw (2.5,5.5) node[scale=1]{$a$};
\draw (3.5,4.5) node[scale=1]{$b$};
\draw (3.5,9.5) node[scale=1]{$b$};
\draw (4.5,6.5) node[scale=1]{$a$};
\draw (4.5,12.5) node[scale=1]{$a$};
\draw (5.5,8.5) node[scale=1]{$a$};
\draw (5.5,13.5) node[scale=1]{$b$};
\draw (6.5,7.5) node[scale=1]{$b$};
\draw (6.5,11.5) node[scale=1]{$a$};
\end{tikzpicture}

\end{center}
\caption{The insertion labeling of the tableau $\Phi(f_1) \in \Tab_7$.}
\label{fig:Phif0}

\end{figure}

\section{Connection between $\varphi$ and $\phi$}

\label{sec:connexion}

\begin{lem}
\label{lem:sij1pluspetitquej2alorsaetb}
Let $f \in \SP_n$, $T = \Phi(f) \in \T_n$ and $i \in [n]$. If $d_{i,\min}^T = d_{n+i}^T$, then $f(2i) = 2i$ and the two dots of $C_i^T$ have different insertion labels.
\end{lem}

\begin{proof}
In general, the dot $d_{i,\min}^T$ is, by its definition, always plotted by Rule 1.(a) or Rule 2.(a) of Definition~\ref{defi:deltainsertions}. Now, if $d_{i,\min}^T = d_{n+i}^T$, then with precision it must be plotted by Rule 2.(a)i., following which $f(2i) = 2i$.

Afterwards, if the insertion labels of the dots of $C_i^T$ are defined by Rule I.1- of Algorithm \ref{algo:insertionlabeling}, then they are different by definition. If they are defined by Rule I.2-, then $d_i^T$ belongs to a column $C_j^T$ with $j < i$, which implies that it was plotted by Rule 2.(a)ii. of Definition~\ref{defi:deltainsertions}. Since $f(2i) = 2i$, this implies that its insertion label is $a$, hence the insertion labels of the dots of $C_i^T$ are different by Rule I.2-a) of Algorithm \ref{algo:insertionlabeling}.
\end{proof}

\begin{lem}
\label{lem:insertionlabelsequalpistollabels}
Let $f \in \SP_n$ and $T = \Phi(f) \in \Tab_n$. The type label of a dot of $T$ is $\alpha$ if and only if its insertion label is $a$.
\end{lem}

\begin{proof} Assume that the Lemma~is true for the dots of the columns $C_{j+1}^T,C_{j+2}^T, \hdots,C_n^T$ for some $j \in [n]$. First of all, we prove that for all $k \in [j+1,n]$, if $d_{k,\min}^T = d_{n+k}^T$, then the even dot of the column $C_k^T$ is labeled with $\beta_0^e$~: if so, then by Lemma~\ref{lem:sij1pluspetitquej2alorsaetb}, the two dots of $C_k^T$ have different insertion labels, hence different type labels by hypothesis (because $k > j$). With precision, the dot whose type label is $\beta$ has been $(f,j)$-inserted in $T^{k}$ with the label $l_e = b$ at the height $h_e = f(2k)/2-k= 0$, so the pistol label of this dot is $\beta_0^e$ in view of Rule III.1- of Algorithm \ref{algo:pistollabeling}.

Now, let $d_i^T \in C_j^T$, $j'= j + h$ where $h \in [0,n-j]$ is the digital label of $d_i^T$, and $i' = \pi_j^{T}(i) \in \left\{j',n+j'\right\}$.

\begin{itemize}
\item If the type label of $d_i^T$ is defined by Rule II.1-a) of Algorithm~\ref{algo:pistollabeling}, then $f(2j') = 2j'$ : indeed, by Remark~\ref{rem:aproposdupistollabeling}(e), the type label of the odd dot of $C_{j'}^T$ is $\alpha$, so by hypothesis the insertion labels of the dots of $C_{j'}^T$ are $a$ and $b$, which implies that $(l_o,l_e) = (a,b)$ at the $j'$-th step of Algorithm \ref{algo:insertionlabeling}, hence $h_e = \delta_e = f(2j)/2-j$ is the digital label $0$ of the even dot of $C_{j'}^T$. Since $j'>j$, the dot $d_{i'}^T$ has been plotted by Rule 2.(a) of Definition~\ref{defi:deltainsertions}, and since $f(2j') = 2j'$, its insertion label equals $a$ (respectively $b$) if and only if $i' = j'$ (respectively $i' = n+j'$) following Rule 2.(a)ii. (respectively Rule 2.(a)i.) of Definition~\ref{defi:deltainsertions}, hence if and only if its type label is $\alpha$ (respectively $\beta$) in this context.

\item If the type label of $d_i^T$ is defined by Rule II.1-b) of Algorithm \ref{algo:pistollabeling}, let $(c,\bar{c})$ be defined as $(a,b)$ if the two dots of $C_{j'}^T$ have different type labels, as $(b,a)$ otherwise. The aim of this part is to prove that the insertion label of $d_i^T$ is $c$ (respectively $\bar{c}$) if its type label is $\gamma$ (respectively $\bar{\gamma}$), \textit{i.e.}, if $d_{i'}^T = d_{j',\min}^T$ (respectively if $d_{i'}^T \neq d_{j',\min}^T$). Now, if $d_{i'}^T = d_{j',\min}^T$ (respectively $d_{i'}^T \neq d_{j',\min}^T$), then in this context we know that $i'=j'$ (respectively $i'=n+j'$), because we showed at the beginning of the proof that if $d_{j',\min}^T = d_{n+j'}^T$ then one of the dots of $C_{j'}^T$ is labeled with $\beta_0^e$, which is not true by hypothesis. So the insertion label of $d_{i'}^T$ is $c$ (respectively $\bar{c}$) following Rule I.2- of Algorithm \ref{algo:insertionlabeling} (respectively following that very same rule and the fact that $d_{n+j'}^T$ has been plotted by Rule 2.(b) of Definition~\ref{defi:deltainsertions}).

\item If the type label of $d_i^T$ is defined by Rule II.2-a) of Algorithm \ref{algo:pistollabeling}, it is straightforward that the two dots of $C_j^T$ are $d_{i_o}^T$ (labeled with $\alpha_0^o$) and $d_{i_e}^T$ (labeled with $\beta_0^e$) where $i_o = \left(\pi_j^T\right)^{-1}(j)$ and $i_e = \left(\pi_j^T\right)^{-1}(n+j)$. Now, at the $j$-th step of Algorithm \ref{algo:insertionlabeling}, we prove that $(l_o,l_e) = (a,b)$ and that the insertion labels of $d_{i_o}^T$ and $d_{i_e}^T$ are respectively $a$ and $b$.
\begin{itemize}[label=*]
\item If $(l_o,l_e)$ has been defined by Rule I.1- of Algorithm \ref{algo:insertionlabeling}, it is straightforward that it is $(a,b)$, and by Rule 1.(a) (respectively Rule 1.(b)) of Definition~\ref{defi:deltainsertions}, the dot $d_{i_o}^T = d_j^T$ belongs to $C_j^T$ and is labeled with $a$ (respectively the dot $d_{i_e}^T$ was plotted by inserting a dot labeled with $b$ in the box $C_{j}^T \cap R_{n+j}^T$).
\item Else, let $d$ be the dot of $R_j^T$. It has been plotted by Rule 2.(a)ii. of Definition~\ref{defi:deltainsertions}, so, since $f(2j-1) = f(2j) = 2j$ by hypothesis, it implies that its insertion label is $a$, hence $(l_o,l_e) = (a,b)$ by Rule I.2-a) of Algorithm \ref{algo:insertionlabeling}. Consequently, by Rule 1.(b) of Definition~\ref{defi:deltainsertions}, the dot $d_{i_o}^T$ (respectively $d_{i_e}^T$) was plotted by inserting a dot labeled with $a$ (respectively $b$) in the box $C_j^T \cap R_{j}^T$ (respectively $C_j^T \cap R_{n+j}^T$).
\end{itemize}

\item If the type label of $d_i^T$ is defined by Rule II.2-b) of Algorithm \ref{algo:pistollabeling}, the insertion label of the other dot $d$ of $C_j^T$ being $a$ if and only if its type label is $t = \alpha$ has already been proved above.
\begin{itemize}[label=*]
\item If $t = \alpha$, suppose first that the type label of $d_i^T$ is $\alpha$, \textit{i.e.}, that $i \neq i' = j$. The equality $i \neq i'$ implies that $d_i^T$ has been plotting by Rule 1.(b) of Definition~\ref{defi:deltainsertions}, and the equality $i'=j$ then implies that its insertion label is $a$. Afterwards, if the type label of $d_i^T$ is $\beta$, then either $i = i'$ or $i' = n+j$. If $i'=n+j$, then $d_i^T$ has been plotted by Rule 1.(b) of Definition~\ref{defi:deltainsertions} with the insertion label $b$. Assume finally that $i=i'=j$. Then, at the $j$-th step of Algorithm \ref{algo:insertionlabeling}, the pair of labels $(l_o,l_e)$ has been defined by Rule I.1-, \textit{i.e.}, it equals $(a,b)$. Since the insertion label of $d$ is $a$ by hypothesis, then the insertion label of $d_i^T$ is $b$.
\item If $t = \beta$, since the digital label of $d$ is not $0$ by hypothesis, by Rule III. of Algorithm \ref{algo:pistollabeling} no dot of $C_j^T$ is labeled with $\beta_0^e$. Following the beginning of this proof, this implies that $d_{j,\min}^T = d_j^T$. In view of this, the type label of $d_i^T$ is $\alpha$ if and only if $i'=j$. If $i'=j$, either the pair of labels $(l_o,l_e)$ is defined by Rule I.1- of Algorithm \ref{algo:insertionlabeling} hence is $(a,b)$, and the insertion label of $d_i^T$ is $a$ because that of $d$ is $b$ by hypothesis, or $d_i^T$ has been plotted by Rule 1.(b) of Definition~\ref{defi:deltainsertions}, so its insertion label is $a$ because $i'=j$. Assume finally that $i' \neq j$, \textit{i.e.}, that $i'= n+j$ and the type label of $d_i^T$ is $\beta$. It has necessarily been plotted by Rule 1.(b) of Definition~\ref{defi:deltainsertions}, and its insertion label is $b$ because $i'=n+j$.
\end{itemize}
\end{itemize}
So the Lemma~is true by induction.
\end{proof}

\begin{prop}
\label{prop:varphicircphi}
The composition $\varphi \circ \Phi$ is the identity map of $\SP_n$.
\end{prop}

\begin{proof}
Let $f \in \SP_n$, $T = \Phi(f) \in \Tab_n$ and $g = \varphi(T) \in \SP_n$. We want to prove that $g = f$. Let $j \in [n]$. By Part II. of Algorithm \ref{algo:insertionlabeling} and Definition~\ref{defi:deltainsertions}, we know that one of the dots $d_o$ of $C_j^T$ has the digital label $h_o = \delta_o = f(2j-1)/2-j$, and that the other dot $d_e$ of $C_j^T$ has the digital label $h_e$ that has the following property :

\begin{itemize}
\item if $\delta_o = \delta_e$ and $l_e = a$, then $h_e=0$;
\item otherwise $h_e = \delta_e = f(2j)/2-j$.
\end{itemize}

Also, by Lemma~\ref{lem:insertionlabelsequalpistollabels}, the type label of $d_o$ (respectively $d_e$) is $\alpha$ if and only if $l_o = a$ (respectively $l_e = a$).

We prove now that the parity labels of $d_o$ and $d_e$ are $o$ and $e$ respectively, and, at the same time, that $g_{|\left\{2j-1,2j\right\}} = f_{|\left\{2j-1,2j\right\}}$.
\begin{enumerate}[label = \arabic* -]
\item If their type labels are different, we know that their insertion labels are different, and by Part I. of Algorithm \ref{algo:insertionlabeling} this implies that $l_o = a$ and $l_e = b$, hence the type labels of $d_o$ and $d_e$ are $\alpha$ and $\beta$ respectively. As a result, by Part III.1- of Algorithm \ref{algo:pistollabeling}, the parity labels of $d_o$ and $d_e$ are $o$ and $e$. Also, since $l_e \neq a$, the digital labels of $d_o$ and $d_e$ are $\delta_o$ and $\delta_e$ respectively, so, by Definition~\ref{defi:varphi}, we have $g(2j-1) = 2(j+\delta_o) = f(2j-1)$ and $g(2j) = 2(j+\delta_e) = f(2j)$.

\item Otherwise, the insertion labels of $d_o$ and $d_e$ are the same, so they have been defined by Rule I.2-(b) of Algorithm \ref{algo:insertionlabeling}.

\begin{enumerate}[label = \arabic*)]

\item If the type label of $d_o$ and $d_e$ is $\beta$, their insertion label is $b$, so it has been defined by Rule I.2-(b)i. of Algorithm \ref{algo:insertionlabeling}. In particular $\delta_o < \delta_e$, and since in that case $\delta_o$ and $\delta_e$ are the digital labels of $d_o$ and $d_e$ respectively, by Rule III.2-b) of Algorithm \ref{algo:pistollabeling} the parity labels of $d_o$ and $d_e$ are $o$ and $e$ respectively, and by Definition~\ref{defi:varphi}, we have $g(2j-1) = 2(j+\delta_o) = f(2j-1)$ and $g(2j) = 2(j+\delta_e) = f(2j)$.

\item If their type labels are $\alpha$, their insertion label is $a$, which has been defined by Rule I.2-(b)ii. of Algorithm \ref{algo:insertionlabeling}. In particular $\delta_o \geq \delta_e$. Let $h_o$ and $h_e$ be the digital labels of $d_o$ and $d_e$ respectively. Since we are in the context III.2- of Algorithm \ref{algo:pistollabeling}, we know that $h_o \neq h_e$. If $\delta_o = \delta_e$ then $(h_o,h_e) = (\delta_o,0)$, otherwise $(h_o,h_e) = (\delta_o,\delta_e)$, so in any case $h_o > h_e$ and by Rule III.2-a) of Algorithm \ref{algo:pistollabeling} the parity labels of $d_o$ and $d_e$ are $o$ and $e$ respectively. Afterwards,

\begin{itemize}

\item If $\delta_o = \delta_e$, hence $(h_o,h_e) = (\delta_o,0)$, then by Definition~\ref{defi:varphi} we have $g(2j-1) = 2(j+\delta_o) = f(2j-1)$ and $g(2j) = 2(j+\delta_o) = 2(j+\delta_e) = f(2j)$;

\item otherwise $(h_o,h_e) = (\delta_o,\delta_e)$ and, by Definition~\ref{defi:varphi}, we have $g(2j-1) = 2(j+\delta_o) = f(2j-1)$ and $g(2j) = 2(j+\delta_e) = f(2j)$.

\end{itemize}

\end{enumerate}

\end{enumerate}
So $g_{|\left\{2j-1,2j\right\}} = f_{|\left\{2j-1,2j\right\}}$ for all $j \in [n]$.
\end{proof}

Proposition~\ref{prop:varphicircphi} implies that the maps $\phi : \SP_n \rightarrow \T_n$ and $\varphi~:~\T_n \rightarrow \SP_n$ are respectively injective and surjective. We intend now to make the image of $\Phi$ explicit.

\begin{defi}
Let $T \in \T_n$, we define $\mathcal{S}(T) \subset [n]$ as the set of integers $i \in [n]$ such that :
\begin{itemize}
\item the dot $d_{n+i}^T$ is not free;
\item the twin dots $d_i^T$ and $d_{n+i}^T$ are not in the same column;
\item no dot of $C_i^T$ has the pistol label $\beta_0^e$.
\end{itemize}
For all such $i$, we define $\mu_T(i)$ as $1$ if $d_{i,\min}^T = d_i^T$, as $-1$ otherwise.

Afterwards, we define $\mathcal{C}(T) \subset [n-1]$ as the set of integers $j \in [n]$ such that $C_j^T$ contains twin dots, say, the dots $d_i^T$ and $d_{n+i}^T$ for some $i \in [n]$, such that no dot of $C_i^T$ is labeled with $\beta_0^e$. For all such $j$, we define $t_T(j)$ as the type label of $d_i^T$.
\end{defi}

\begin{rem}
\label{rem:sommedefreeplusSplusCegalegrounded}
For all $T \in \T_n$, by Proposition~\ref{prop:egalitedesstats}, we have the formula~$$\fr(T) + \# \mathcal{S}(T) + \# \mathcal{C}(T) = \ngr(T) = \ndf(f)$$
where $f = \varphi(T)$.
\end{rem}

\begin{rem}
\label{rem:detailssurlimagedevarphi}
In the proof of Lemma~\ref{lem:insertionlabelsequalpistollabels}, we showed that for all $T$ of the kind $\phi(f)$ for some $f \in \SP_n$, and for all $i \in [n]$, if $d_{i,\min}^T = d_{n+i}^T$, then one of the dots of $C_i^T$ is labeled with $\beta_0^e$, hence $i \not\in \mathcal{S}(T)$.
\end{rem}

\begin{defi}
\label{defi:soustableaux}
Let $\tilde{\T}_n$ be the subset of $\T_n$ made of the tableaux $T$ that have the following properties : for all $i \in [n]$,
\begin{enumerate}[label=(\alph*)]
\item if $i \in \mathcal{S}(T)$, then $d_{i,\min}^T = d_i^T$;
\item if $j \in \mathcal{C}(T)$, then $t_T(j) = \alpha$.
\end{enumerate}
\end{defi}

\begin{lem}
\label{lem:imagePhiinclus}
The image of $\Phi : \SP_n \rightarrow \Tab_n$ is a subset of $\tilde{\Tab}_n$.
\end{lem}

\begin{proof}
Let $f \in \SP_n$ and $T = \Phi(f) \in \Tab_n$. The tableau $T$ having the property (a) of Definition~\ref{defi:soustableaux} comes from Remark~\ref{rem:detailssurlimagedevarphi}. Now, let $j \in \mathcal{C}(T)$ and $i \geq j$ such that $C_j^T$ contains the twin dots $d_i^T$ and $d_{n+i}^T$. They both have the same digital label $i-j$ (which implies that $h_o = h_e = i-j$ at the $j$-th step of Algorithm \ref{algo:insertionlabeling}), so by Part III. of Algorithm~\ref{algo:pistollabeling} their type labels are different. In view of Lemma~\ref{lem:insertionlabelsequalpistollabels}, this implies that their insertion labels are $l_o = a$ and $l_e = b$. Now, at the beginning of the $j$-th step of Algorithm \ref{algo:insertionlabeling}, the row $R_i^{T^j}$ is empty, so the $(f,j)$-insertion in $T^j$ of a dot labeled with $l_o = a$ at the height $h_o = i-j$ leads to plotting a dot labeled with $a$ in the box $C_j^{T^j} \cap R_i^{T^j}$ following Rule 1.(a) or Rule 2.(a)ii. of Definition~\ref{defi:deltainsertions}. So $d_i^T$ is the dot of $C_j^T$ whose insertion label is $a$, and its type label is $\alpha$ by Lemma~\ref{lem:insertionlabelsequalpistollabels}, hence $T$ has the property (b) of Definition~\ref{defi:soustableaux}.
\end{proof}

\begin{defi}
\label{defi:epsilon}
For all $T \in \T_n$ and $j \in [n]$, we define $\epsilon_j^T$ as the set of the pistol labels of the dots of $C_j^T$.
\end{defi}

\begin{lem}
\label{lem:uniquefacondechangerdetypelabel}
Let $f \in \SP_n$ and $(T,T') \in \varphi^{-1}(f)^2$. Let $j \in [n]$ such that $\epsilon_j^T \neq \epsilon_j^{T'}$. There exists $k \in \mathcal{C}(T) \cap \mathcal{C}(T') \cap [j-1]$ such that $t_T(k) \neq t_{T'}(k)$.
\end{lem}

\begin{proof}
Suppose first that the dots of $C_j^T$ (respectively $C_j^{T'}$) have the same type label. If, with precision, the four dots of $C_j^T$ and $C_j^{T'}$ have the same type label, since $\varphi(T) = \varphi(T') = f$, in view of Definition~\ref{defi:varphi} the set of the digital labels of the dots of $C_j^T$ equals the set of the digital labels of the dots of $C_j^{T'}$, and by Part III.2- of Algorithm \ref{algo:pistollabeling} we have in fact $\epsilon_j^T = \epsilon_j^{T'}$, which is false by hypothesis. So, should $T$ and $T'$ be transposed, we can suppose that the dots of $C_j^T$ (respectively $C_j^{T'}$) have the type label $\alpha$ (respectively $\beta$). By Part III.2- of Algorithm \ref{algo:pistollabeling}, we have $\epsilon_j^T = \left\{\alpha_{h_o}^o, \alpha_{h_e}^e\right\}$ with $h_o > h_e$, and $\epsilon_j^{T'} = \left\{\beta_{h'_o}^o,\beta_{h'_e}^e\right\}$ with $h'_o < h'_e$. Following Definition~\ref{defi:varphi}, this implies that $f(2j-1) \geq f(2j)$ and $f(2j-1) < f(2j)$, which is absurd.

So, should $T$ and $T'$ be transposed, we can suppose that the dots of $C_j^T$ have different type labels, and by Part III.1- of Algorithm \ref{algo:pistollabeling}, we have $\epsilon_j^T = \left\{\alpha_{h_o}^o, \beta_{h_e}^e\right\}$ for some $(h_o,h_e)\in [0,n-j]^2$, which, following Definition~\ref{defi:varphi}, implies that $f(2j-1) = 2(j+h_o)$ and $f(2j) = 2(j+h_e)$. Now, if the dots of $C_j^{T'}$ had different type labels, then by Part III.1- of Algorithm \ref{algo:pistollabeling} and Definition~\ref{defi:varphi} we would have $\epsilon_j^T = \epsilon_j^{T'}$, so it is necessary that the dots of $C_j^{T'}$ have the same type label. This implies several things, enumerated hereafter.

\begin{enumerate}[label=(\Alph*)]

\item Since no dot of $C_j^{T'}$ is labeled with $\beta_0^e$ in view of Remark~\ref{rem:aproposdupistollabeling}(e), then by Lemma~\ref{lem:pointfixesietseulementsibeta0e} the integer $2j$ is not a fixed point of $f$, hence no dot of $C_j^T$ is labeled with $\beta_0^e$ (\textit{i.e.}, we have $h_e > 0$).

\item Suppose that $d_{j,\min}^{T'} \in C_j^{T'}$ (hence $d_{j,\min}^{T'} = d_j^{T'}$). Then its type label is defined by Rule II.2- of Algorithm \ref{algo:pistollabeling}. Since no dot of $C_j^{T'}$ has the pistol label $\beta_0^e$, then it is with precision defined by Rule II.2-b). Since it is $d_{j,\min}^{T'}$, whether it is defined by Rule II.2-b)i. or Rule II.2-b)ii., the type labels of the two dots of $C_j^{T'}$ are different, which is absurd. So $d_{j,\min}^{T'}$ belongs to a column $C_{j'}^{T'}$ with $j'<j$.

\item Consequently, by Lemma~\ref{lem:ouestdmin}, we know that $d_{j,\min}^T \in C_{j'}^T$.

\item Following Remark~\ref{rem:aproposdupistollabeling}(e), the type label of $d_{j,\min}^T$ is $\alpha$, whereas the type label of $d_{j,\min}^{T'}$ is $\beta$.

\end{enumerate}

Now, if $j' \in \mathcal{C}(T)$, then with precision $d_{j,\min}^T = d_j^T$ because the other dot of $C_{j'}^T$ is $d_{n+j}^T$ (in particular $t_{j'}(T) = \alpha$ following (D)), also $\epsilon_{j'}^T = \left\{\alpha_{j-j'}^o, \beta_{j-j'}^e\right\}$ in view of Rule II.1-b) and Rule III.1- of Algorithm \ref{algo:pistollabeling}, and by Definition~\ref{defi:varphi} we have $f(2j'-1) = f(2j') = 2j$. Still by Definition~\ref{defi:varphi}, since one of the dots of $C_{j'}^{T'}$ (the dot $d_{j,\min}^{T'}$ by (D)) doesn't have the type label $\alpha$, then the two dots of $C_{j'}^{T'}$ have the same digital label $j-j'$ (incidentally, by Part III.2- of Algorithm \ref{algo:pistollabeling} this implies that they have different type labels, \textit{i.e.}, that the dot of $C_{j'}^{T'}$ that is not $d_{j,\min}^{T'}$ has the type label $\alpha$). Let $d$ be the twin dot of $d_{j,\min}^{T'}$, \textit{i.e.}, the dot defined as $d_j^{T'}$ if $d_{j,\min}^{T'} = d_{n+j}^{T'}$, as $d_{n+j}^{T'}$ if $d_{j,\min}^{T'} = d_{j}^{T'}$. Since $d_{j,\min}^{T'} \in C_{j'}^{T'}$ and the other dot of $C_{j'}^{T'}$ has the digital label $j-j'$, by Definition~\ref{defi:TPath} this implies that $d$ is located in a column $C_{j''}^{T'}$ with $j'' \leq j'$. By Definition~of $d_{j,\min}^{T'}$, we have $j'' \geq j'$ hence $j'' = j'$. In other words, the integer $j'$ belongs to $\mathcal{C}(T')$ and $d_{j,\min}^{T'} = d_j^{T'}$ has the type label $\beta$ in view of (D), hence $t_{j'}(T') = \beta \neq t_{j'}(T)$, which is exactly the statement of the lemma.

Otherwise (if $j' \not\in \mathcal{C}(T)$), suppose that $\epsilon_{j'}^T = \epsilon_{j'}^{T'}$. Since the type labels of $d_{j,\min}^T \in C_{j'}^T$ and $d_{j,\min}^{T'} \in C_{j'}^{T'}$ are respectively $\alpha$ and $\beta$ and their digital label $j-j'$, in view of Part III. of Algorithm \ref{algo:pistollabeling} this implies that $\epsilon_{j'}^T = \epsilon_{j'}^{T'} = \left\{\alpha_{j-j'}^o, \beta_{j-j'}^e\right\}$. In particular, the two dots of $C_j^T$ have the same digital label $j-j'$. By Definition~\ref{defi:TPath}, this means that both the twin dots $d_j^T$ and $d_{n+j}^T$ are located in columns $C_{j''}^T$ with $j'' \leq j'$. By Definition~of $d_{j,\min}^T \in C_{j'}^T$, this forces those two dots to be the dots of $C_{j'}^T$, which contradicts $j' \not\in \mathcal{C}(T)$ since no dot of $C_j^T$ is labeled with $\beta_0^e$ in view of (A). So, necessarily $\epsilon_{j'}^T \neq \epsilon_{j'}^{T'}$ and we are in the situation of the beginning of the proof with $j$ being replaced by $j'$. This produces some integer $j^{(2)} \in [j'-1]$ such that $d_{j',\min}^T \in C_{j^{(2)}}^T$ and $d_{j',\min}^{T'} \in C_{j^{(2)}}^{T'}$ do not have the same type label. If the statement of the Lemma~is false, then it in fact produces a strictly decreasing sequence of integers $(j^{(2)},j^{(3)},\hdots) \in [n]^{\mathbb{N}}$, which is absurd. So the Lemma~is true.
\end{proof}

\begin{prop}
\label{prop:auplusunelementdeTtilde}
The map $\varphi_{|\tilde{T}_n}$ is injective.
\end{prop}

\begin{proof}
Let $(T_1,T_2) \in \left( \tilde{T}_n \right)^2$ such that $\varphi(T_1) = \varphi(T_2) =: f \in \SP_n$. By Lemma~\ref{lem:uniquefacondechangerdetypelabel} and the property (b) of Definition~\ref{defi:soustableaux}, we know that $\epsilon_j^{T_1} = \epsilon_j^{T_2}$ for all $j \in [n]$. Let $\mathcal{H}(j)$ be the induction hypothesis that for all $k \in [j-1]$ and $i \in [k,2n]$, the dot $d_i^{T_1}$ belongs to $C_k^{T_1}$ only if $d_i^{T_2}$ belongs to $C_k^{T_2}$. Hypothesis $\mathcal{H}(1)$ is obviously true. Suppose that Hypothesis $\mathcal{H}(j)$ is true for some $j \in [n-1]$. The fact that the dots of $T_1$ and $T_2$ are located at the same levels in their $j-1$ first columns implies that $\pi_j^{T_1} = \pi_j^{T_2} =: \pi_j$. Let $i_1 \in [j,n]$ such that $d^1 = d_{i_1}^{T_1} \in C_j^{T_1}$, we consider the digital label $h \in [0,n-j]$ of $d^1$, $j'=j+h$, and $i'_1 = \pi_j(i_1) \in \left\{j',n+j'\right\}$. We denote by $d^2 = d_{i_2}^{T_2}$ the dot of $C_j^{T_2}$ that has the same pistol label as $d^1$. We intend to prove that $d^2 = d_i^{T_2}$, \textit{i.e.}, that $i'_1 = i'_2$ because $\pi_j$ is injective. Since $\epsilon_j^{T_1} = \epsilon_j^{T_2}$ and the column $C_j^{T_1}$ has a dot labeled with $\beta_0^e$ if and only if the column $C_j^{T_2}$ has a dot labeled with $\beta_0^e$ (in view of Lemma~\ref{lem:pointfixesietseulementsibeta0e}), then the type label of $d^1$ and $d^2$ is defined by the same rule among Rules II.1-a),II.1-b),II.2-a),II.2-b)i. and II.2-b)ii. of Algorithm \ref{algo:pistollabeling}. The equality $i'_1 = i'_2$ is then straightforward in each case in view of $d_{j',\min}^{T_1} = d_{j'}^{T_1}$ and $d_{j',\min}^{T_2} = d_{j'}^{T_2}$ following the condition (a) of Definition~\ref{defi:soustableaux}. So Hypothesis $\mathcal{H}(j)$ is true. By induction, Hypothesis $\mathcal{H}(n)$ is true, thence $T_1 = T_2$.
\end{proof}

\begin{cor}
\label{cor:Phicircvarphi}
The map $\Phi \circ \varphi_{|\tilde{T}_n}$ is the identity map of $\tilde{T}_n$. (In view of Lemma~\ref{lem:imagePhiinclus}, it implies that the image of $\Phi : \SP_n \rightarrow \Tab_n$ is exactly $\tilde{\Tab}_n$.)
\end{cor}

\begin{proof}
Let $T \in \tilde{T}_n$, $f = \varphi(T) \in \SP_n$ and $T'=\phi(f) \in \tilde{T}_n$. By Proposition~\ref{prop:varphicircphi}, we know that $\varphi(T') = f$, so $T = T'$ in view of Proposition~\ref{prop:auplusunelementdeTtilde}. 
\end{proof}

\section{Proof of Theorem~\ref{theo:bigeni}}

\label{sec:proof}

We now know that the injection $\phi : \SP_n \hookrightarrow \T_n$ induces a bijection from $\SP_n$ to $\tilde{\T}_n \subset \T_n$, whose inverse map is $\varphi_{|\tilde{\T}_n}$, and which maps the statistic $\overrightarrow{\ndf}$ to the statistic $\overrightarrow{\ngr}$ in view of Proposition~\ref{prop:egalitedesstats}. To finish the proof of Theorem~\ref{theo:bigeni}, it remains to show Formula (\ref{eq:lasommesurlesTdunmemeF}) for all $f \in \SP_n$, which we do in this section with the help of Algorithm \ref{algo:switch} and Algorithm \ref{algo:mute}, which compute $\varphi^{-1}(f)$.
\begin{defi}
\label{defi:ensemblesjumeaux}
Let $f \in \SP_n$ and $j \in [n]$. We define $\T_f(j)$ as the set of the tableaux $T \in \varphi^{-1}(f)$ such that $j \in \mathcal{C}(T)$. Let $T_0 \in \T_f(j)$. We define $\T(T_0,j)$ as the set of tableaux $T \in \varphi^{-1}(f)$ such that $\mathcal{C}(T) \cap [j~-~1] = \mathcal{C}(T_0) \cap [j-1]$ and $t_T(k) = t_{T_0}(k)$ for all $k \in \mathcal{C}(T) \cap [j-1]$ (this set is not empty because it contains $T_0$).
\end{defi}

\begin{lem}
\label{lem:motivationdefplusprecise}
Using the notations of Definition~\ref{defi:ensemblesjumeaux}, if $T \in \T(T_0,j)$, then the dots of $C_j^T$ have the same levels as the dots of $C_j^{T_0}$, and $j \in \mathcal{C}(T)$.
\end{lem}

\begin{proof}
Let $i \in [j,n]$ such that the dots of $C_j^{T_0}$ are the twin dots $d_i^{T_0}$ and $d_{n+i}^{T_0}$.
Since $j \in \mathcal{C}(T_0)$, we have $i > j$ (otherwise the dot $d_{n+i}^{T_0}$ would have the pistol label $\beta_0^e$). Consequently, the type labels of the twin dots of $C_j^{T_0}$ are both defined by Rule II.1-b) of Algorithm \ref{algo:pistollabeling}, so they are different. Let $j' \leq j$ such that $d_{j,\min}^{T_0} \in C_{j'}^{T_0}$. With precision, we have $j'<j$ : otherwise, we would have $d_{j,\min}^{T_0} = d_j^{T_0} \in C_j^{T_0}$ and $i$ would equal $j$, which is false. The type label of $d_{j,\min}^{T_0}$ is $\alpha$ following Rule II.1-b) of Algorithm \ref{algo:pistollabeling}. Now, if $T \in \T(T_0,j)$, suppose that the type label of $d_{j,\min}^T$ is $\beta$. By Lemma~\ref{lem:ouestdmin}, we know that $d_{j,\min}^T \in C_{j'}^T$. Also, by hypothesis and Lemma~\ref{lem:uniquefacondechangerdetypelabel}, it is necessary that $\epsilon_j^{T_0} = \epsilon_j^T$. Since $d_{j,\min}^{T_0}$ is labeled with $\alpha_{j-j'}$ and $d_{j,\min}^T$ with $\beta_{j-j'}$, it is necessary that $\epsilon_j^{T_0} = \epsilon_j^T = \left\{\alpha_{j-j'}^o,\beta_{j-j'}^e\right\}$. In particular, by Remark~\ref{rem:aproposdupistollabeling}(f), the dots of $C_{j'}^{T_0}$ are $d_{\left(\pi_{j'}^{T_0}\right)^{-1}(j)}^{T_0}$ and $d_{\left(\pi_{j'}^{T_0}\right)^{-1}(n+j)}^{T_0}$, and the dots of $C_{j'}^{T}$ are $d_{\left(\pi_{j'}^{T}\right)^{-1}(j)}^{T}$ and $d_{\left(\pi_{j'}^{T}\right)^{-1}(n+j)}^{T}$. Now, by Definition~\ref{defi:twindots} of $d_{j,\min}^{T_0} \in C_{j'}^{T_0}$ and $d_{j,\min}^{T} \in C_{j'}^{T}$, these two dots are $d_j^{T_0}$ and $d_{n+j}^{T_0}$ (respectively $d_j^{T}$ and $d_{n+j}^{T}$). In other words, since no dot of $C_j^{T_0}$ is labeled with $\beta_0^e$ (hence no dot of $C_j^{T}$ is labeled with $\beta_0^e$ in view of Lemma~\ref{lem:pointfixesietseulementsibeta0e}), we have $j' \in \mathcal{C}(T_0) \cap \mathcal{C}(T)$ and $t_{T_0}(j') = \alpha \neq \beta = t_T(j')$, which is absurd by hypothesis. So the type label of $d_{j,\min}^T$ is $\alpha$, hence the two dots of $C_j^T$ have different type labels. Afterwards, since $\varphi(T_0) = \varphi(T) = f$ and the type labels of the dots of $C_j^{T_0}$ and $C_j^T$ are not both $\alpha$, by Definition~\ref{defi:varphi} the digital label of the odd dot (respectively even dot) of $C_j^T$ is the same as the digital label of the odd dot (respectively even dot) of $C_j^{T_0}$. Since the dots of $C_j^{T_0}$ are twins, they have the same digital label $i-j$, and so do the dots of $C_j^T$. Consequently, by Remark~\ref{rem:aproposdupistollabeling}(f), the dots of $C_j^T$ are $d_{\left(\pi_j^T\right)^{-1}(i)}^T = d_i^T$ and $d_{\left(\pi_j^T\right)^{-1}(n+i)}^T = d_{n+i}^T$.

Finally, let $f = \varphi(T_0)$. By hypothesis $j \in \mathcal{C}(T_0)$, so $f(2j) > 2j$ in view of Lemma~\ref{lem:pointfixesietseulementsibeta0e}. Since $f = \varphi(T)$, then Lemma~\ref{lem:pointfixesietseulementsibeta0e} also implies that $j \in \mathcal{C}(T)$.
\end{proof}

\begin{defi}
\label{defi:sousensemblesjumeaux}
Using the notations of Definition~\ref{defi:ensemblesjumeaux}, by Lemma~\ref{lem:motivationdefplusprecise} we can decompose $\T(T_0,j)$ into the disjoint union $\T(T_0,j,\alpha) \sqcup \T(T_0,j,\beta)$ where, for all $\gamma \in \left\{\alpha,\beta\right\}$, the subset $\T(T_0,j,\gamma)$ is the set of the tableaux $T \in \T(T_0,j)$ such that $t_T(j) = \gamma$.  
\end{defi}

\subsection{An operation on $\mathcal{S}(T)$}

\begin{defi}
\label{defi:deuxpuissancem}
Let $T \in \T_n$ and $\left\{i_1,i_2,\hdots,i_m\right\}_{<} = \mathcal{S}(T)$. For all $k \in [m]$, we define $\mu_k^T$ as $1$ if $d_{i_k,\min}^T = d_{i_k}^T$, as $-1$ otherwise.
\end{defi}

\begin{algo}
\label{algo:switch}
Let $T \in \T_n$ and $\left\{i_1,i_2,\hdots,i_m\right\}_{<} = \mathcal{S}(T)$. We consider $\mu = (\mu_1,\mu_2,\hdots,\mu_m) \in \left\{-1,1\right\}^m$, and we define a tableau $S_{\mu}(T)$ as follows. Let $T'_1$ be the empty 1-tableau. For $j$ from $1$ to $n$, suppose that $T'_j$ is a $j$-tableau (which is true for $j=1$). In particular the map $\pi_j^{T'_j}$ is defined. Let $(r_1,r_2) \in [j,2n]^2$ such that $d_{r_1}^T$ and $d_{r_2}^T$ are the two dots of $C_j^T$. For all $p \in \left\{1,2\right\}$, we consider the integer $r'_p=\pi_j^T(r_p) \in [j,n] \sqcup [n+j,2n]$.

\begin{enumerate}[label =\arabic* -]

\item If $r'_p \in  \left\{i_k,n+i_k\right\}$ for some $k \in [m]$, let $(r_{\gamma},r_{\bar{\gamma}})$ be the pair $(i_k,n+i_k)$ if $\mu_k = 1$, or the pair $(n+i_k,i_k)$ if $\mu_k = -1$. We define the integer $r''_p$ as $\left(\pi_j^{T'_j}\right)^{-1}(r_{\gamma})$ if the type label of $d_{r_p}^T$ is the same as that of $d_{i_k,\min}^T$, as $\left(\pi_j^{T'_j}\right)^{-1}(r_{\bar{\gamma}})$ otherwise.

\item Otherwise, we define the integer $r''_p$ as $\left(\pi_j^{T'_j}\right)^{-1}(r'_p)$.

\end{enumerate}

Since $\pi_j^{T'_j}$ is bijective, the integers $r''_1$ and $r''_2$ are different, and by definition~the rows $R_{r''_1}^{T'_j}$ and $R_{r''_2}^{T'_j}$ are empty. We then define the $(j+1)$-tableau $T'_{j+1}$ by plotting two dots in the boxes $C_{j}^{T'_j} \cap R_{r''_1}^{T'_j}$ and $C_{j}^{T'_j} \cap R_{r''_2}^{T'_j}$. This algorithm produces a $(n+1)$-tableau $T'_n$ which we denote by $S_{\mu}(T)$, and which belongs to $\T_n$ as a $(n+1)$-tableau.

\end{algo}

For example, in Figure \ref{fig:undes4}, we consider a tableau $T \in \T_7$ such that $\mathcal{S}(T) = \left\{3,5\right\}$ and $(\mu_1^T,\mu_2^T) = (-1,1)$. In this figure, the tableau $T$ is depicted with its pistol labeling.

\begin{figure}[!htbp]
    \centering

\begin{tikzpicture}[scale=0.5]

\begin{scope}[xshift=1cm,yshift=-31cm]
\draw (0,0) grid[step=1] (7,14);
\draw (0,0) -- (7,7);
\draw (0,14) [dashed] -- (7,7);
\draw [black] (0.5,0.5) node[scale=1]{$\beta_0^e$};
\draw [black] (0.5,10.5) node[scale=1]{$\alpha_2^o$};
\draw [black] (1.5,1.5) node[scale=1]{$\alpha_0^o$};
\draw [black] (1.5,2.5) node[scale=1]{$\beta_1^e$};
\draw [black] (2.5,3.5) node[scale=1]{$\beta_1^e$};
\draw [black] (2.5,9.5) node[scale=1]{$\alpha_1^o$};
\draw [black] (3.5,5.5) node[scale=1]{$\alpha_2^e$};
\draw [black] (3.5,6.5) node[scale=1]{$\alpha_3^o$};
\draw [black] (4.5,4.5) node[scale=1]{$\alpha_0^o$};
\draw [black] (4.5,7.5) node[scale=1]{$\beta_1^e$};
\draw [black] (5.5,12.5) node[scale=1]{$\alpha_1^o$};
\draw [black] (5.5,13.5) node[scale=1]{$\beta_1^e$};
\draw [black] (6.5,11.5) node[scale=1]{$\beta_0^e$};
\draw [black] (6.5,8.5) node[scale=1]{$\alpha_0^o$};

\draw [fill=gray,opacity=0.25] (0,12) rectangle (1,13);
\draw [fill=gray,opacity=0.25] (1,11) rectangle (2,12);
\draw [fill=gray,opacity=0.25] (2,10) rectangle (3,11);
\draw [fill=gray,opacity=0.25] (3,9) rectangle (4,10);
\draw [fill=gray,opacity=0.25] (4,8) rectangle (5,9);
\draw [fill=gray,opacity=0.25] (5,7) rectangle (6,8);

\draw [fill=gray,opacity=0.25] (0,0) rectangle (1,1);
\draw [fill=gray,opacity=0.25] (1,1) rectangle (2,2);
\draw [fill=gray,opacity=0.25] (2,2) rectangle (3,3);
\draw [fill=gray,opacity=0.25] (3,3) rectangle (4,4);
\draw [fill=gray,opacity=0.25] (4,4) rectangle (5,5);
\draw [fill=gray,opacity=0.25] (5,5) rectangle (6,6);

\draw (-0.5,2.5) node[scale=0.8]{$3$};
\draw (-0.5,10.5) node[scale=0.8]{$10$};
\draw (7.5,5.5) node[scale=0.8]{$5$};
\draw (7.5,7.5) node[scale=0.8]{$12$};
\end{scope}

\end{tikzpicture}

 \caption{Tableau $T \in \T_7$ such that $\mathcal{S}(T) = \left\{3,5\right\}$.}
 \label{fig:undes4}

\end{figure}

In Figure \ref{fig:les4}, we represent the pistol-labeled versions of the tableaux $S_{\mu}(T)$ for all $\mu \in \left\{-1,1\right\}^2$.

\begin{figure}[!htbp]
    \centering

\begin{tikzpicture}[scale=0.5]

\draw (0,-15.5) -- (18,-15.5);
\draw (9,1) -- (9,-32);

\begin{scope}[xshift=1cm,yshift=-14cm]
\draw (0,0) grid[step=1] (7,14);
\draw (0,0) -- (7,7);
\draw (0,14) [dashed] -- (7,7);
\draw [black] (0.5,0.5) node[scale=1]{$\beta_0^e$};
\draw [black] (0.5,2.5) node[scale=1]{$\alpha_2^o$};
\draw [black] (1.5,1.5) node[scale=1]{$\alpha_0^o$};
\draw [black] (1.5,10.5) node[scale=1]{$\beta_1^e$};
\draw [black] (2.5,3.5) node[scale=1]{$\beta_1^e$};
\draw [black] (2.5,9.5) node[scale=1]{$\alpha_1^o$};
\draw [black] (3.5,5.5) node[scale=1]{$\alpha_2^e$};
\draw [black] (3.5,6.5) node[scale=1]{$\alpha_3^o$};
\draw [black] (4.5,4.5) node[scale=1]{$\alpha_0^o$};
\draw [black] (4.5,7.5) node[scale=1]{$\beta_1^e$};
\draw [black] (5.5,11.5) node[scale=1]{$\alpha_1^o$};
\draw [black] (5.5,13.5) node[scale=1]{$\beta_1^e$};
\draw [black] (6.5,8.5) node[scale=1]{$\beta_0^e$};
\draw [black] (6.5,12.5) node[scale=1]{$\alpha_0^o$};

\draw [fill=gray,opacity=0.25] (0,12) rectangle (1,13);
\draw [fill=gray,opacity=0.25] (1,11) rectangle (2,12);
\draw [fill=gray,opacity=0.25] (2,10) rectangle (3,11);
\draw [fill=gray,opacity=0.25] (3,9) rectangle (4,10);
\draw [fill=gray,opacity=0.25] (4,8) rectangle (5,9);
\draw [fill=gray,opacity=0.25] (5,7) rectangle (6,8);

\draw [fill=gray,opacity=0.25] (0,0) rectangle (1,1);
\draw [fill=gray,opacity=0.25] (1,1) rectangle (2,2);
\draw [fill=gray,opacity=0.25] (2,2) rectangle (3,3);
\draw [fill=gray,opacity=0.25] (3,3) rectangle (4,4);
\draw [fill=gray,opacity=0.25] (4,4) rectangle (5,5);
\draw [fill=gray,opacity=0.25] (5,5) rectangle (6,6);

\draw (3.5,-1) node[scale=1]{$S_{(1,1)}(T)$};

\draw (-0.5,2.5) node[scale=0.8]{$3$};
\draw (-0.5,10.5) node[scale=0.8]{$10$};
\draw (7.5,5.5) node[scale=0.8]{$5$};
\draw (7.5,7.5) node[scale=0.8]{$12$};
\end{scope}

\begin{scope}[xshift=10cm,yshift=-14cm]
\draw (0,0) grid[step=1] (7,14);
\draw (0,0) -- (7,7);
\draw (0,14) [dashed] -- (7,7);
\draw [black] (0.5,0.5) node[scale=1]{$\beta_0^e$};
\draw [black] (0.5,2.5) node[scale=1]{$\alpha_2^o$};
\draw [black] (1.5,1.5) node[scale=1]{$\alpha_0^o$};
\draw [black] (1.5,10.5) node[scale=1]{$\beta_1^e$};
\draw [black] (2.5,3.5) node[scale=1]{$\beta_1^e$};
\draw [black] (2.5,9.5) node[scale=1]{$\alpha_1^o$};
\draw [black] (3.5,7.5) node[scale=1]{$\alpha_2^e$};
\draw [black] (3.5,6.5) node[scale=1]{$\alpha_3^o$};
\draw [black] (4.5,4.5) node[scale=1]{$\alpha_0^o$};
\draw [black] (4.5,5.5) node[scale=1]{$\beta_1^e$};
\draw [black] (5.5,12.5) node[scale=1]{$\alpha_1^o$};
\draw [black] (5.5,13.5) node[scale=1]{$\beta_1^e$};
\draw [black] (6.5,11.5) node[scale=1]{$\beta_0^e$};
\draw [black] (6.5,8.5) node[scale=1]{$\alpha_0^o$};

\draw [fill=gray,opacity=0.25] (0,12) rectangle (1,13);
\draw [fill=gray,opacity=0.25] (1,11) rectangle (2,12);
\draw [fill=gray,opacity=0.25] (2,10) rectangle (3,11);
\draw [fill=gray,opacity=0.25] (3,9) rectangle (4,10);
\draw [fill=gray,opacity=0.25] (4,8) rectangle (5,9);
\draw [fill=gray,opacity=0.25] (5,7) rectangle (6,8);

\draw [fill=gray,opacity=0.25] (0,0) rectangle (1,1);
\draw [fill=gray,opacity=0.25] (1,1) rectangle (2,2);
\draw [fill=gray,opacity=0.25] (2,2) rectangle (3,3);
\draw [fill=gray,opacity=0.25] (3,3) rectangle (4,4);
\draw [fill=gray,opacity=0.25] (4,4) rectangle (5,5);
\draw [fill=gray,opacity=0.25] (5,5) rectangle (6,6);

\draw (3.5,-1) node[scale=1]{$S_{(1,-1)}(T)$};

\draw (-0.5,2.5) node[scale=0.8]{$3$};
\draw (-0.5,10.5) node[scale=0.8]{$10$};
\draw (7.5,5.5) node[scale=0.8]{$5$};
\draw (7.5,7.5) node[scale=0.8]{$12$};
\end{scope}

\begin{scope}[xshift=1cm,yshift=-31cm]
\draw (0,0) grid[step=1] (7,14);
\draw (0,0) -- (7,7);
\draw (0,14) [dashed] -- (7,7);
\draw [black] (0.5,0.5) node[scale=1]{$\beta_0^e$};
\draw [black] (0.5,10.5) node[scale=1]{$\alpha_2^o$};
\draw [black] (1.5,1.5) node[scale=1]{$\alpha_0^o$};
\draw [black] (1.5,2.5) node[scale=1]{$\beta_1^e$};
\draw [black] (2.5,3.5) node[scale=1]{$\beta_1^e$};
\draw [black] (2.5,9.5) node[scale=1]{$\alpha_1^o$};
\draw [black] (3.5,5.5) node[scale=1]{$\alpha_2^e$};
\draw [black] (3.5,6.5) node[scale=1]{$\alpha_3^o$};
\draw [black] (4.5,4.5) node[scale=1]{$\alpha_0^o$};
\draw [black] (4.5,7.5) node[scale=1]{$\beta_1^e$};
\draw [black] (5.5,12.5) node[scale=1]{$\alpha_1^o$};
\draw [black] (5.5,13.5) node[scale=1]{$\beta_1^e$};
\draw [black] (6.5,11.5) node[scale=1]{$\beta_0^e$};
\draw [black] (6.5,8.5) node[scale=1]{$\alpha_0^o$};

\draw [fill=gray,opacity=0.25] (0,12) rectangle (1,13);
\draw [fill=gray,opacity=0.25] (1,11) rectangle (2,12);
\draw [fill=gray,opacity=0.25] (2,10) rectangle (3,11);
\draw [fill=gray,opacity=0.25] (3,9) rectangle (4,10);
\draw [fill=gray,opacity=0.25] (4,8) rectangle (5,9);
\draw [fill=gray,opacity=0.25] (5,7) rectangle (6,8);

\draw [fill=gray,opacity=0.25] (0,0) rectangle (1,1);
\draw [fill=gray,opacity=0.25] (1,1) rectangle (2,2);
\draw [fill=gray,opacity=0.25] (2,2) rectangle (3,3);
\draw [fill=gray,opacity=0.25] (3,3) rectangle (4,4);
\draw [fill=gray,opacity=0.25] (4,4) rectangle (5,5);
\draw [fill=gray,opacity=0.25] (5,5) rectangle (6,6);

\draw (3.5,-1) node[scale=1]{$S_{(-1,1)}(T)$};

\draw (-0.5,2.5) node[scale=0.8]{$3$};
\draw (-0.5,10.5) node[scale=0.8]{$10$};
\draw (7.5,5.5) node[scale=0.8]{$5$};
\draw (7.5,7.5) node[scale=0.8]{$12$};
\end{scope}

\begin{scope}[xshift=10cm,yshift=-31cm]
\draw (0,0) grid[step=1] (7,14);
\draw (0,0) -- (7,7);
\draw (0,14) [dashed] -- (7,7);
\draw [black] (0.5,0.5) node[scale=1]{$\beta_0^e$};
\draw [black] (0.5,10.5) node[scale=1]{$\alpha_2^o$};
\draw [black] (1.5,1.5) node[scale=1]{$\alpha_0^o$};
\draw [black] (1.5,2.5) node[scale=1]{$\beta_1^e$};
\draw [black] (2.5,3.5) node[scale=1]{$\beta_1^e$};
\draw [black] (2.5,9.5) node[scale=1]{$\alpha_1^o$};
\draw [black] (3.5,7.5) node[scale=1]{$\alpha_2^e$};
\draw [black] (3.5,6.5) node[scale=1]{$\alpha_3^o$};
\draw [black] (4.5,4.5) node[scale=1]{$\alpha_0^o$};
\draw [black] (4.5,5.5) node[scale=1]{$\beta_1^e$};
\draw [black] (5.5,11.5) node[scale=1]{$\alpha_1^o$};
\draw [black] (5.5,13.5) node[scale=1]{$\beta_1^e$};
\draw [black] (6.5,8.5) node[scale=1]{$\beta_0^e$};
\draw [black] (6.5,12.5) node[scale=1]{$\alpha_0^o$};

\draw [fill=gray,opacity=0.25] (0,12) rectangle (1,13);
\draw [fill=gray,opacity=0.25] (1,11) rectangle (2,12);
\draw [fill=gray,opacity=0.25] (2,10) rectangle (3,11);
\draw [fill=gray,opacity=0.25] (3,9) rectangle (4,10);
\draw [fill=gray,opacity=0.25] (4,8) rectangle (5,9);
\draw [fill=gray,opacity=0.25] (5,7) rectangle (6,8);

\draw [fill=gray,opacity=0.25] (0,0) rectangle (1,1);
\draw [fill=gray,opacity=0.25] (1,1) rectangle (2,2);
\draw [fill=gray,opacity=0.25] (2,2) rectangle (3,3);
\draw [fill=gray,opacity=0.25] (3,3) rectangle (4,4);
\draw [fill=gray,opacity=0.25] (4,4) rectangle (5,5);
\draw [fill=gray,opacity=0.25] (5,5) rectangle (6,6);

\draw (3.5,-1) node[scale=1]{$S_{(-1,-1)}(T)$};

\draw (-0.5,2.5) node[scale=0.8]{$3$};
\draw (-0.5,10.5) node[scale=0.8]{$10$};
\draw (7.5,5.5) node[scale=0.8]{$5$};
\draw (7.5,7.5) node[scale=0.8]{$12$};
\end{scope}

\end{tikzpicture}

 \caption{The tableaux $S_{\mu}(T)$ for all $\mu \in \left\{-1,1\right\}^2$.}
 \label{fig:les4}

\end{figure}
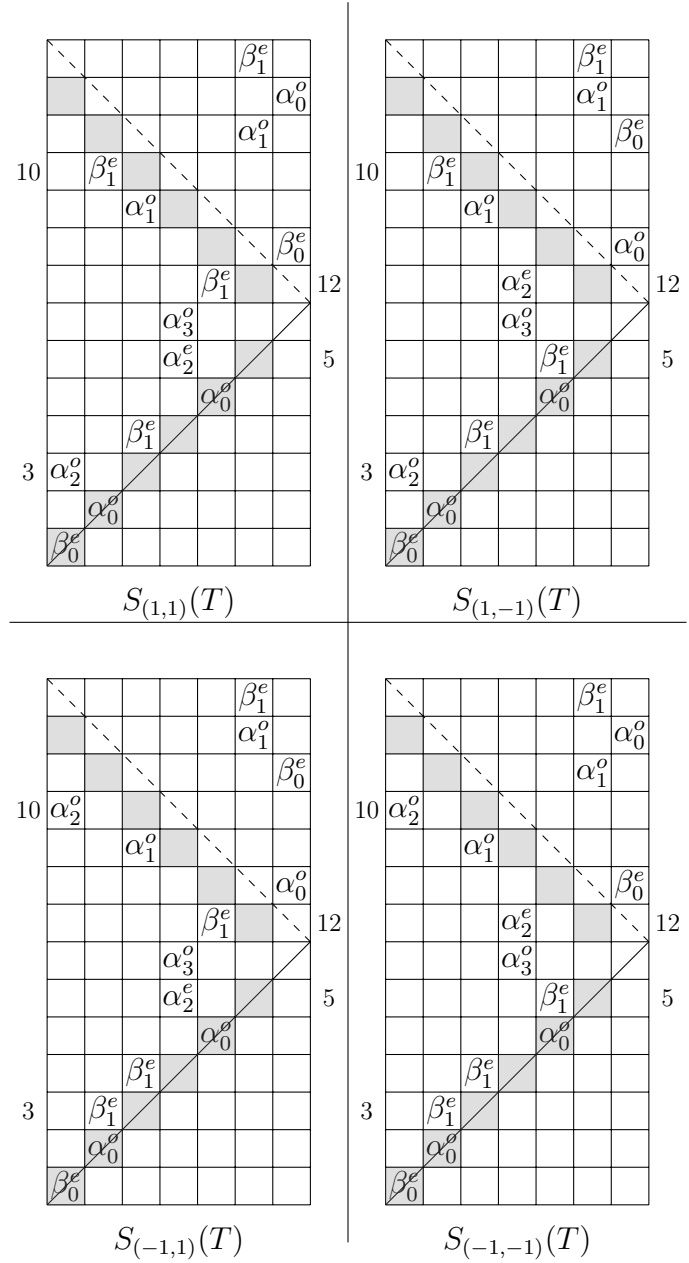

Note that if $\mu_0 = (\mu_1^T,\mu_2^T)$ ($=(-1,1)$), then $S_{\mu_0}(T) = T$ (in the bottom left-hand corner in Figure \ref{fig:les4}). Afterwards, for all $\mu \in \left\{-1,1\right\}^2$ and $j \in [7]$, we have $\epsilon_j^{S_{\mu}(T)} = \epsilon_j^T$, consequently $\varphi(S_{\mu}(T)) = \varphi(T)$ ; also, we have $\mathcal{C}(S_{\mu}(T)) = \mathcal{C}(T) = \left\{3\right\}$, and $t_{S_{\mu}(T)}(3) = \beta = t_T(3)$. All these remarks are generalized in the easy following result.

\begin{prop}
\label{prop:existenceetuniciteduswitch}
Let $T \in \T_n$, $f = \varphi(T) \in \SP_n$ and $\left\{i_1,i_2,\hdots,i_m\right\}_{<} = \mathcal{S}(T)$. For all $\mu = (\mu_1,\mu_2,\hdots,\mu_m) \in \left\{-1,1\right\}^m$, the tableau $S_{\mu}(T)$ is the unique tableau $T' \in \varphi^{-1}(f)$ such that :
\begin{itemize}
\item $\mathcal{S}(T') = \mathcal{S}(T)$ and for all $k \in [m]$, we have $\mu_k^{T'} = \mu_k$;
\item $\mathcal{C}(T') = \mathcal{C}(T)$ and for all $j \in \mathcal{C}(T)$, we have $t_T(j) = t_{t'}(j)$.
\end{itemize}
\end{prop}

\subsection{An operation on $\mathcal{C}(T)$}

\begin{algo}
\label{algo:mute}
Let $T \in \varphi^{-1}(f)$ for some given $f \in \SP_n$, $j_0 \in \mathcal{C}(T)$ and $\gamma \in \left\{\alpha,\beta\right\}$. We define a tableau $M_{j_0,\gamma}(T)$ as follows. First of all, let $\left\{i_1,i_2,\hdots,i_m\right\}_< = \mathcal{S}(T)$, we define $\tilde{T} \in \varphi^{-1}(f)$ as $S_{\mu}(T)$ where $\mu$ is the sequence $(1,1,\hdots,1) \in \left\{-1,1\right\}^m$. Afterward, we define $T^{j_0+1} \in \mathfrak{T}_n^{j_0+1}$ as follows. Let $(c,\bar{c})$ be defined as $(a,b)$ if $\gamma = \alpha$, as $(b,a)$ otherwise.
\begin{itemize}

\item For all $j < j_0$ and $i \in [j,2n]$, if the column $C_j^{\tilde{T}}$ contains the dot $d_i^{\tilde{T}}$ whose type label is $\alpha$ (respectively $\beta$), then the column $C_j^{T^{j_0+1}}$ contains the dot $d_i^{T^{j_0+1}}$ labeled with the letter $a$ (respectively $b$).

\item If $d_i^T$ and $d_{n+i}^T$ are the twin dots of $C_{j_0}^{\tilde{T}}$, then the column $C_{j_0}^{T^{j_0+1}}$ contains the twin dots $d_i^{T^{j_0+1}}$ and $d_{n+i}^{T^{j_0+1}}$ labeled with the letters $c$ and $\bar{c}$ respectively.

\end{itemize}

Afterwards, we define $M_{j_0,\gamma}(T) \in \T_n$ as the tableau produced by the restriction of Algorithm \ref{algo:insertionlabeling} from step $j_0+1$ (using $T^{j_0+1}$) to step $n$.
\end{algo}

\begin{rem}
\label{rem:touslesmusontenvoyesurlememe}
With the notations of Algorithm \ref{algo:mute}, for all $\mu \in \left\{-1,1\right\}^m$, we have the equality $M_{j_0,\gamma}(T) = M_{j_0,\gamma}(S_{\mu}(T))$.
\end{rem}

For example, consider the tableau $T \in \T_7$ of Figure \ref{fig:undes4}, with $\mathcal{S}(T) = \left\{3,5\right\}$ and $\mathcal{C}(T) = \left\{3\right\}$. To compute $M_{3,\alpha}(T)$ and $M_{3,\beta}(T)$, we first need to make $f = \varphi(T)$ explicit, which can be read from the pistol labeling of $T$ in Figure \ref{fig:undes4} (or from any pistol labeling of the tableaux depicted in Figure \ref{fig:les4} for that matter) : $$f = (6,\textbf{2},4,\textbf{6},8,\textbf{8},14,\textbf{12},10,\textbf{12},14,\textbf{14},14,\textbf{14}) \in \SP_7,$$
whose graphical representation is depicted in Figure \ref{fig:f}.

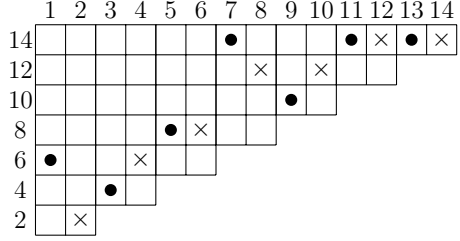
\begin{figure}[!htbp]

\begin{center}

\begin{tikzpicture}[scale=0.4]

\draw (0,0) grid[step=1] (2,7);
\draw (2,1) grid[step=1] (4,7);
\draw (4,2) grid[step=1] (6,7);
\draw (6,3) grid[step=1] (8,7);
\draw (8,4) grid[step=1] (10,7);
\draw (10,5) grid[step=1] (12,7);
\draw (12,6) grid[step=1] (14,7);

\fill (0.5,2.5) circle (0.2);
\draw (1.5,0.5) node[scale=0.8]{$\times$};
\fill (2.5,1.5) circle (0.2);
\draw (3.5,2.5) node[scale=0.8]{$\times$};
\fill (4.5,3.5) circle (0.2);
\fill (5.5,3.5) node[scale=0.8]{$\times$};
\fill (6.5,6.5) circle (0.2);
\draw (7.5,5.5) node[scale=0.8]{$\times$};
\fill (8.5,4.5) circle (0.2);
\draw (9.5,5.5) node[scale=0.8]{$\times$};
\fill (10.5,6.5) circle (0.2);
\draw (11.5,6.5) node[scale=0.8]{$\times$};
\fill (12.5,6.5) circle (0.2);
\draw (13.5,6.5) node[scale=0.8]{$\times$};

\draw (-0.5,0.5) node[scale=0.8]{$2$};
\draw (-0.5,1.5) node[scale=0.8]{$4$};
\draw (-0.5,2.5) node[scale=0.8]{$6$};
\draw (-0.5,3.5) node[scale=0.8]{$8$};
\draw (-0.5,4.5) node[scale=0.8]{$10$};
\draw (-0.5,5.5) node[scale=0.8]{$12$};
\draw (-0.5,6.5) node[scale=0.8]{$14$};

\draw (0.5,7.5) node[scale=0.8]{$1$};
\draw (1.5,7.5) node[scale=0.8]{$2$};
\draw (2.5,7.5) node[scale=0.8]{$3$};
\draw (3.5,7.5) node[scale=0.8]{$4$};
\draw (4.5,7.5) node[scale=0.8]{$5$};
\draw (5.5,7.5) node[scale=0.8]{$6$};
\draw (6.5,7.5) node[scale=0.8]{$7$};
\draw (7.5,7.5) node[scale=0.8]{$8$};
\draw (8.5,7.5) node[scale=0.8]{$9$};
\draw (9.5,7.5) node[scale=0.8]{$10$};
\draw (10.5,7.5) node[scale=0.8]{$11$};
\draw (11.5,7.5) node[scale=0.8]{$12$};
\draw (12.5,7.5) node[scale=0.8]{$13$};
\draw (13.5,7.5) node[scale=0.8]{$14$};

\end{tikzpicture}

\end{center}
\caption{The surjective pistol $f = \varphi(T) \in \SP_7$.}
\label{fig:f}

\end{figure}

 Following the notations of Algorithm \ref{algo:mute}, we have $\tilde{T} = S_{(1,1)}(T) \in \T_7$, which is represented at the top left-hand corner of Figure \ref{fig:les4}. We then use $\tilde{T}$ to compute the insertion labeled versions of $M_{3,\alpha}(T)$ and $M_{3,\beta}(T)$ in Figure \ref{fig:les2}.

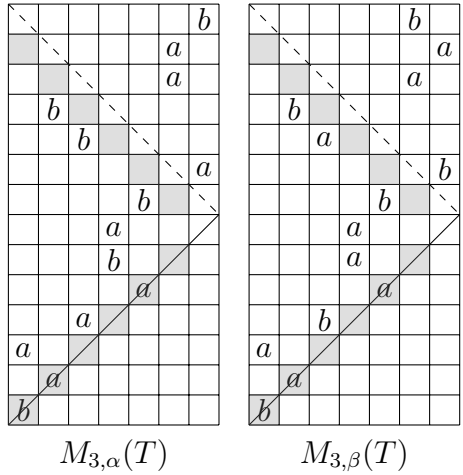
\begin{figure}[!htbp]
    \centering

\begin{tikzpicture}[scale=0.4]

\begin{scope}[xshift=1cm,yshift=-14cm]
\draw (0,0) grid[step=1] (7,14);
\draw (0,0) -- (7,7);
\draw (0,14) [dashed] -- (7,7);
\draw (0.5,0.5) node[scale=1]{$b$};
\draw (0.5,2.5) node[scale=1]{$a$};
\draw (1.5,1.5) node[scale=1]{$a$};
\draw (1.5,10.5) node[scale=1]{$b$};
\draw (2.5,3.5) node[scale=1]{$a$};
\draw (2.5,9.5) node[scale=1]{$b$};
\draw (3.5,5.5) node[scale=1]{$b$};
\draw (3.5,6.5) node[scale=1]{$a$};
\draw (4.5,4.5) node[scale=1]{$a$};
\draw (4.5,7.5) node[scale=1]{$b$};
\draw (5.5,11.5) node[scale=1]{$a$};
\draw (5.5,12.5) node[scale=1]{$a$};
\draw (6.5,8.5) node[scale=1]{$a$};
\draw (6.5,13.5) node[scale=1]{$b$};

\draw [fill=gray,opacity=0.25] (0,12) rectangle (1,13);
\draw [fill=gray,opacity=0.25] (1,11) rectangle (2,12);
\draw [fill=gray,opacity=0.25] (2,10) rectangle (3,11);
\draw [fill=gray,opacity=0.25] (3,9) rectangle (4,10);
\draw [fill=gray,opacity=0.25] (4,8) rectangle (5,9);
\draw [fill=gray,opacity=0.25] (5,7) rectangle (6,8);

\draw [fill=gray,opacity=0.25] (0,0) rectangle (1,1);
\draw [fill=gray,opacity=0.25] (1,1) rectangle (2,2);
\draw [fill=gray,opacity=0.25] (2,2) rectangle (3,3);
\draw [fill=gray,opacity=0.25] (3,3) rectangle (4,4);
\draw [fill=gray,opacity=0.25] (4,4) rectangle (5,5);
\draw [fill=gray,opacity=0.25] (5,5) rectangle (6,6);

\draw (3.5,-1) node[scale=1]{$M_{3,\alpha}(T)$};
\end{scope}

\begin{scope}[xshift=9cm,yshift=-14cm]
\draw (0,0) grid[step=1] (7,14);
\draw (0,0) -- (7,7);
\draw (0,14) [dashed] -- (7,7);
\draw (0.5,0.5) node[scale=1]{$b$};
\draw (0.5,2.5) node[scale=1]{$a$};
\draw (1.5,1.5) node[scale=1]{$a$};
\draw (1.5,10.5) node[scale=1]{$b$};
\draw (2.5,3.5) node[scale=1]{$b$};
\draw (2.5,9.5) node[scale=1]{$a$};
\draw (3.5,5.5) node[scale=1]{$a$};
\draw (3.5,6.5) node[scale=1]{$a$};
\draw (4.5,4.5) node[scale=1]{$a$};
\draw (4.5,7.5) node[scale=1]{$b$};
\draw (5.5,11.5) node[scale=1]{$a$};
\draw (5.5,13.5) node[scale=1]{$b$};
\draw (6.5,8.5) node[scale=1]{$b$};
\draw (6.5,12.5) node[scale=1]{$a$};

\draw [fill=gray,opacity=0.25] (0,12) rectangle (1,13);
\draw [fill=gray,opacity=0.25] (1,11) rectangle (2,12);
\draw [fill=gray,opacity=0.25] (2,10) rectangle (3,11);
\draw [fill=gray,opacity=0.25] (3,9) rectangle (4,10);
\draw [fill=gray,opacity=0.25] (4,8) rectangle (5,9);
\draw [fill=gray,opacity=0.25] (5,7) rectangle (6,8);

\draw [fill=gray,opacity=0.25] (0,0) rectangle (1,1);
\draw [fill=gray,opacity=0.25] (1,1) rectangle (2,2);
\draw [fill=gray,opacity=0.25] (2,2) rectangle (3,3);
\draw [fill=gray,opacity=0.25] (3,3) rectangle (4,4);
\draw [fill=gray,opacity=0.25] (4,4) rectangle (5,5);
\draw [fill=gray,opacity=0.25] (5,5) rectangle (6,6);

\draw (3.5,-1) node[scale=1]{$M_{3,\beta}(T)$};
\end{scope}

\end{tikzpicture}

 \caption{The insertion-labeled versions of the tableaux $M_{3,\alpha}(T)$ and $M_{3,\beta}(T)$.}
 \label{fig:les2}

\end{figure}

\begin{lem}
\label{lem:sij1pluspetitquej2alorsaetbBIS}
With the notations of Algorithm \ref{algo:mute}, let $T' = M_{j_0,\gamma}(T)$ and $i \in [n]$. If $d_{i,\min}^{T'} = d_{n+i}^{T'}$, then $f(2i) = 2i$ and the two dots of $C_i^{T'}$ have different insertion labels.
\end{lem}

\begin{proof}
Let $j \in [n]$ such that $C_j^{T'}$ contains $d_{n+i}^{T'}$. Since $C_{j_0}^{T'}$ contains twin dots, we know that $j \neq j_0$. If $j < j_0$, we also have $d_{n+i}^{\tilde{T}} = d_{i,\min}^{\tilde{T}}$. By Definition~of $\tilde{T}$, this implies that $i \not\in \mathcal{S}(\tilde{T})$. Since $d_{n+i}^{\tilde{T}}$ is not free, it is then necessary that $C_i^{\tilde{T}}$ contains a dot labeled with $\beta_0^e$, hence $f(2i) = 2i$ by Lemma~\ref{lem:pointfixesietseulementsibeta0e}. If $j > j_0$, the proof of $f(2i) = 2i$ is the same as in the proof of Lemma~\ref{lem:sij1pluspetitquej2alorsaetb}.

Afterwards, since $f(2i) = 2i$ and $f(2 j_0) \neq 2j_0$ because $j_0 \in \mathcal{C}(\tilde{T})$, we have $i \neq j_0$. If $i < j_0$, then by Lemma~\ref{lem:pointfixesietseulementsibeta0e} and Remark~\ref{rem:aproposdupistollabeling}(e) the type labels of the dots of $C_i^{\tilde{T}}$ are different, hence the insertion labels of the dots of $T'$ are different by definition. If $i > j_0$, the proof of the insertion labels of the dots of $T'$ being different is the same as in the proof of Lemma~\ref{lem:sij1pluspetitquej2alorsaetb}.
\end{proof}

\begin{lem}
\label{lem:assialpha}
With the notations of Algorithm \ref{algo:switch}, the type label of a dot of the tableau $T'= M_{j_0,\gamma_0}(T)$ is $\alpha$ if and only if its insertion label is $a$.
\end{lem}

\begin{proof}
The proof of the Lemma~for the dots of $C_{j_0+1}^{T'},C_{j_0+2}^{T'},\hdots,C_n^{T'}$ is the same as that of Lemma~\ref{lem:insertionlabelsequalpistollabels} where Lemma~\ref{lem:sij1pluspetitquej2alorsaetbBIS} plays the role of Lemma~\ref{lem:sij1pluspetitquej2alorsaetb}.

Now, let $i \in [j_0,2n]$ such that $C_{j_0}^{T'}$ contains the twin dots $d_i^{T'}$ and $d_{n+i}^{T'}$. Since $j_0 \in \mathcal{C}(\tilde{T})$, we know that $i > j_0$. Let $(c,\bar{c})$ and $\bar{\gamma_0}$ be defined as $(a,b)$ and $\beta$ respectively if $\gamma_0 = \alpha$, as $(b,a)$ and $\alpha$ otherwise. By Definition~the insertion labels of $d_i^{T'}$ and $d_{n+i}^{T'}$ are $c$ and $\bar{c}$ respectively. Afterwards, suppose that $C_i^{T'}$ contains a dot $d_e$ labeled with $\beta_0^e$. By Remark~\ref{rem:aproposdupistollabeling}(e), the other dot $d_o$ of $C_i^{T'}$ has the type label $\alpha$. Since $i > j_0$, we know that the insertion labels of $d_o$ and $d_e$ are $a$ and $b$ respectively, so these insertion labels have been defined following Rule I.2-a) of Algorithm \ref{algo:insertionlabeling}. Since the digital label of $d_e$ is $0 = h_e = f(2i)/2-i$ in this situation, by Lemma~\ref{lem:pointfixesietseulementsibeta0e} the column $C_i^{\tilde{T}}$ contains a dot labeled with $\beta_0^e$, which contradicts $j_0 \in \mathcal{C}(\tilde{T})$. So no dot of $C_i^{T'}$ is labeled with $\beta_0^e$, and the type labels of $d_i^{T'}$ and $d_{n+i}^{T'}$ are defined by Rule II.1-b) of Algorithm \ref{algo:pistollabeling}. Since the insertion labels of the dots of $C_i^{T'}$ are defined by Rule I.2- of Algorithm \ref{algo:insertionlabeling}, and since they are different labels if and only if these two dots of different type labels because $i > j_0$, then by Rule I.1-b) of Algorithm \ref{algo:pistollabeling}, the type labels of $d_i^{T'}$ and $d_{n+i}^{T'}$ are respectively $\alpha$ and $\beta$ if $c_0 = a$ (\textit{i.e.}, if $\gamma_0 = \alpha$), as $\beta$ and $\alpha$ otherwise (if $\gamma_0 = \beta$), in other words their type labels are respectively $\gamma_0$ and $\bar{\gamma_0}$, and the Lemma~is true for these two dots.

Finally, a thorough analysis of the rules of Algorithm \ref{algo:pistollabeling} following which the type labels of the dots of $C_1^{T'},C_2^{T'},\hdots,C_{j_0-1}^{T'}$ are defined show that these type labels are the same as in $\tilde{T}$, hence that the type label of each dot is $\alpha$ if and only if its insertion label is $a$ by definition.
\end{proof}

\begin{prop}
\label{prop:lemuteverifielesconditions}
With the notations of Algorithm \ref{algo:switch}, the tableau $T'= M_{j_0,\gamma}(T)$ is an element of $\T(T,j_0,\gamma)$.
\end{prop}

\begin{proof}
Let $g = \varphi(T')$. The proof of $g_{|[2j_0+1,2n]} = f_{|[2j_0+1,2n]}$ is the same as in the proof of Proposition~\ref{prop:varphicircphi} where Lemma~\ref{lem:assialpha} plays the role of Lemma~\ref{lem:pointfixesietseulementsibeta0e}. Also, in the proof of Lemma~\ref{lem:assialpha}, we show with precision that $\epsilon_j^{T'} = \epsilon_j^{\tilde{T}}$ for all $j \leq j_0$, so $g = f$.

Afterwards, for all $(i,j) \in [j_0]^2$, the twin dots $d_i^{T'}$ and $d_{n+i}^{T'}$ are by definition the two dots of $C_j^{T'}$ if and only if $d_i^{\tilde{T}}$ and $d_{n+i}^{\tilde{T}}$ are  the two dots of $C_j^{\tilde{T}}$. In that case, the integer $j$ belongs to $\mathcal{C}(T')$ if and only if no dot of $C_i^{T'}$ is labeled with $\beta_0^e$, which, in view of Lemma~\ref{lem:pointfixesietseulementsibeta0e}, is equivalent with $f(2i) > 2i$ and no dot of $C_i^{\tilde{T}}$ being labeled with $\beta_0^e$, hence with $j \in \mathcal{C}(\tilde{T})$. So $\mathcal{C}(T') \cap [j_0] = \mathcal{C}(\tilde{T}) \cap [j_0] = \mathcal{C}(T) \cap [j_0]$.

Finally, if $j < j_0$, the type label of $d_i^{T'}$ being $\alpha$ is equivalent with its insertion label being $a$ (by Lemma~\ref{lem:assialpha}), hence with the type label of $d_i^{\tilde{T}}$ being $\alpha$ by definition. By Proposition~\ref{prop:existenceetuniciteduswitch}, this is also equivalent with the type label of $d_i^T$ being $\alpha$. In other words $t_{T'}(j) = t_{T}(j)$, and $T' \in \T(T,j_0)$. With precision, by Part I.2- of Algorithm \ref{algo:mute}, the insertion label of the lower dot of $C_{j_0}^{T'}$ is $c$ defined as $a$ if $\gamma_0 = \alpha$, as $b$ otherwise. So its type label is $\gamma_0$ by Lemma~\ref{lem:assialpha}, and $T' \in \T(T,j_0,\gamma)$.
\end{proof}

\begin{rem}
\label{rem:siTnonvidealors}
Proposition~\ref{prop:lemuteverifielesconditions} implies that for all $T \in \T_n$, $j \in \mathcal{C}(T)$ and $\gamma \in \left\{\alpha,\beta\right\}$, the set $\T(T,j,\gamma)$ is not empty.
\end{rem}

\begin{rem}
\label{rem:transitif}
For all $f \in \SP_n$, we can now construct every element of $\varphi^{-1}(f)$. Indeed, every two elements $T$ and $T'$ of $\varphi^{-1}(f)$ are linked by a finite numbers of applications of the kind $S_{\mu}$ and $M_{j_0,\gamma_0}$. To prove it, it is enough to show that we can obtain $\phi(f) \in \varphi^{-1}(f)$ by applying a finite number of these applications to any element $T \in \varphi^{-1}(f)$. Recall that $\phi(f)$ is the unique element of $\tilde{T}_n$ in $\varphi^{-1}(f)$ because $\varphi_{|\tilde{T}_n}$ is injective by Proposition~\ref{prop:auplusunelementdeTtilde}, so we only need to show that $T$ is mapped to an element of $\tilde{T}_n$ by a finite number of these applications. We do that as follows.

If $\mathcal{C}(T)$ is not empty, let $j_0$ be its minimal element. We define $T_1$ as $M_{j_0,\alpha}(T)$. Afterwards, if $\mathcal{C}(T_1) \cap [j_0+1,n]$ is not empty, we set $j_1$ as its minimal element, and we define $T_2$ as $M_{j_1,\alpha}(T_1)$. Clearly, by induction, we define a finite sequence $(T=T_0,T_1,T_2,\hdots,T_k)$ (for some $k \geq 0$, where the case $k = 0$ corresponds to $\mathcal{C}(T)$ being empty) such that $t_{T_k}(j) = \alpha$ for all $j \in \mathcal{C}(T_k)$ in view of Proposition~\ref{prop:lemuteverifielesconditions}. Finally, let $m = \# \mathcal{S}(T_k) \in [0,n]$. If $m = 0$, then obviously $T_k \in \tilde{T}_n$. Otherwise, let $\mu = (1,1,\hdots,1) \in \left\{-1,1\right\}^m$, then $S_{\mu}(T_k) \in \tilde{T}_n$ in view of Proposition~\ref{prop:existenceetuniciteduswitch}.
\end{rem}

\subsection{Proof of Formula (\ref{eq:lasommesurlesTdunmemeF})}

\begin{lem}
\label{lem:sijalorsminpourtoutlemonde}
Let $f \in \SP_n$, $T_0 \in \varphi^{-1}(f)$ and $k \in \left\{0\right\} \sqcup \mathcal{C}(T_0)$. We consider $T \in \T(T_0,k)$ (where $\T(T_0,0)$ is defined as $\varphi^{-1}(f)$). If $\mathcal{C}(T_0) \cap [k+1,n] = \emptyset$, then $\mathcal{C}(T) \cap [k+1,n] = \emptyset$. Otherwise, we have $$\min \mathcal{C}(T_0) \cap [k+1,n] =  \min \mathcal{C}(T) \cap [k+1,n].$$
\end{lem}

\begin{proof}
Let $j_0$ (respectively $j$) be defined as $n+1$ if $\mathcal{C}(T_0) \cap [k+1,n] = \emptyset$ (respectively $\mathcal{C}(T) \cap [k+1,n] = \emptyset$), as $\min \mathcal{C}(T_0) \cap [k+1,n]$ (respectively $\min \mathcal{C}(T) \cap [k+1,n]$) otherwise. The proof of the Lemma~consists in proving the equality $j_0 = j$. Assume that $j_0 \neq j$. Since $\T(T_0,j) = \T(T,j)$, should $(T_0,T)$ be replaced with $(T,T_0)$, we can suppose that $j_0 > j$ (which implies that $\mathcal{C}(T) \neq \emptyset$ and $j \in [n]$). Since $f = \varphi(T)$, by Definition~\ref{defi:varphi} we know that $f(2j-1) = f(2j) = 2i$ where the dots of $C_{j}^{T}$ are the twin dots $d_i^{T}$ and $d_{n+i}^{T}$. Also, by Part III. of Algorithm \ref{algo:pistollabeling}, since $d_i^{T}$ and $d_{n+i}^{T}$ have the same digital label $i-j$, then they have different type labels, and we obtain $\epsilon_{j}^{T} = \left\{\alpha_{i-j}^o,\beta_{i-j}^e\right\}$. Also, in this situation $d_{i,\min}^{T} = d_i^{T} \in C_{j}^{T}$. But $f$ is also $\varphi(T_0)$, so by Lemma~\ref{lem:uniquefacondechangerdetypelabel} it is necessary that $\epsilon_{j}^{T_0} = \epsilon_{j}^{T} = \left\{\alpha_{i-j}^o,\beta_{i-j}^e\right\}$. Now, since $j \not\in \mathcal{C}(T_0)$ and $C_i^{T_0}$ has no dot labeled with $\beta_0^e$ (otherwise, by Lemma~\ref{lem:pointfixesietseulementsibeta0e} it would imply that $f(2i) = 2i$ and that $C_i^{T}$ also contains a dot labeled with $\beta_0^e$, which would contradict $j \in \mathcal{C}(T)$), this implies that the dots of $C_{j}^{T_0}$ are not twin dots. Still, since $\epsilon_{j}^{T_0} = \left\{\alpha_{i-j}^o,\beta_{i-j}^e\right\}$, by Remark~\ref{rem:aproposdupistollabeling}(f) the dots of $C_{j}^{T_0}$ are $d_{\left(\pi_{j}^{T_0}\right)^{-1}(i)}^{T_0}$ and $d_{\left(\pi_{j}^{T_0}\right)^{-1}(n+i)}^{T_0}$ ; since they are not the twin dots $d_i^{T_0}$ and $d_{n+i}^{T_0}$, this implies that either $d_i^{T_0}$ or $d_{n+i}^{T_0}$ belongs to a column $C_{j'}^{T_0}$ with $j'<j'$, hence $d_{i,\min}^{T_0} \not\in C_{j}^{T_0}$, which is absurd in view of Lemma~\ref{lem:ouestdmin} and the fact that $d_{i,\min}^{T} \in C_{j}^{T}$. So $j_0 = j$.
\end{proof}

\begin{lem}
\label{cor:sifreechezlunetpaschezlautrealorschangementlabel}
Let $f \in \SP_n$ and $(T,T') \in \varphi^{-1}(f)^2$. Let $i \in [n]$ such that $d_{n+i}^T$ is not free and $d_{n+i}^{T'}$ is free. Let also $(j_1,j_2) \in [i]^2$ such that $d_{i,\min}^T \in C_{j_1}^T$ and $d \in C_{j_2}^T$ where $d$ is the twin dot of $d_{i,\min}^T$. Then, there exists $k \in \mathcal{C}(T) \cap \mathcal{C}(T') \in [j_2-1]$ such that $t_T(k) \neq t_{T'}(k)$.
\end{lem}

\begin{proof}
By Lemma~\ref{lem:uniquefacondechangerdetypelabel}, it suffices that show that $\epsilon_{j_2}^T \neq \epsilon_{j_2}^{T'}$. Since $d_{n+i}^{T'}$ is free, and in view of Lemma~\ref{lem:ouestdmin}, we know that $d_{i,\min}^{T'} = d_i^{T'} \in C_{j_1}^{T'}$. Also, since $d_{n+i}^{T'}$ is in particular non grounded, Proposition~\ref{prop:egalitedesstats} and Lemma~\ref{lem:pointfixesietseulementsibeta0e} imply that no dot of $C_i^T$ or $C_i^{T'}$ is labeled with $\beta_0^e$. Now, by Rule I. of Algorithm \ref{algo:pistollabeling}, the digital label of $d$ is $i-j_2$. By Definition~\ref{defi:varphi}, this implies that either $f(2j_2-1) = 2i$ or $f(2j_2) = 2i$, hence at least one of the dots of $C_{j_2}^{T'}$ has the digital label $i-j_2$. In fact, since $d_{n+i}^T$ is free, by Remark~\ref{rem:aproposdupistollabeling}(f) there exists exactly one such dot : the dot $d'= d_{\left(\pi_j^{T'}\right)^{-1}(i)}^{T'}$. Now, the type labels of $d$ and $d'$ are defined by the same rule of Algorithm \ref{algo:pistollabeling}, and this rule is either Rule II.1-b) or Rule~II.2-.

Suppose that the type labels of $d$ and $d'$ are defined by Rule II.1-b). Since $d \in \left\{d_i^T,d_{n+i}^T\right\} \backslash \left\{d_{i,\min}^T\right\}$ but $d'= d_{\left(\pi_j^{T'}\right)^{-1}(i)}^{T'}$ where $d_i^{T'} = d_{i,\min}^{T'}$, the type labels of $d$ and $d'$ are different. Assume now that $\epsilon_{j_2}^T = \epsilon_{j_2}^{T'}$. Then these two sets equal $\left\{\alpha_{i-j_2}^o,\beta_{i-j_2}^e\right\}$, which contradicts $d'$ being the only dot of $C_{j_2}^{T'}$ that has the digital label $i-j_2$. So $\epsilon_{j_2}^T \neq \epsilon_{j_2}^{T'}$.

Suppose finally that the type labels of $d$ and $d'$ are defined by Rule II.2-. In this situation, since $i = j_2$, we know that $d = d_{n+i}^T$. Whether its type label is defined by Rule II.2-a), Rule II.2-b)i. or Rule II.2-b)ii., it equals $\beta$. Afterwards, since the dots of $C_{j_2}^{T'}$ (among which is $d'$) have different digital labels, the type label of $d'$ is defined by Rule II.2-b), and whether it follows Rule II.2-b)i. or Rule II.2-b)ii., the dots of $C_{j_2}^{T'}$ have the same type label. Consequently, if we suppose that $\epsilon_{j_2}^T = \epsilon_{j_2}^{T'}$, then they must have the type label $\beta$ following Rule II.2-b)ii., which is absurd because it implies that $d_i^{T'} \neq d_{i,\min}^{T'}$. So $\epsilon_{j_2}^T \neq \epsilon_{j_2}^{T'}$.
\end{proof}

\begin{prop}
\label{prop:egalitetableauxpistolets}
For all $f \in \SP_n$, we have
\begin{equation}
\label{eq:equationdulemme}
\sum_{T \in \varphi^{-1}(f)} 2^{\fr(T)} = 2^{\ndf(f)}.
\end{equation}
\end{prop}

\begin{proof}
Let $T_0 \in \varphi^{-1}(f)$. If $\mathcal{C}(T_0) = \emptyset$, then $\mathcal{C}(T) = \emptyset$ for all $T \in \varphi^{-1}(f)$ by Lemma~\ref{lem:sijalorsminpourtoutlemonde}, and $\fr(T) = \fr(T_0)$ by Corollary \ref{cor:sifreechezlunetpaschezlautrealorschangementlabel}. Consequently, we obtain
\begin{equation}
\label{eq:cardinalzero}
\sum_{T \in \varphi^{-1}(f)} 2^{\fr(T)} = 2^{\fr(T_0)} \times \# \varphi^{-1}(f). 
\end{equation}
Now, by Proposition~\ref{prop:existenceetuniciteduswitch}, we know that $\#\varphi^{-1}(f) = 2^{\# \mathcal{S}(T_0)}$. Also, by Remark~\ref{rem:sommedefreeplusSplusCegalegrounded}, the integer $\fr(T_0) + \mathcal{S}(T_0)$ equals $\ndf(f)$ because $\mathcal{C}(T_0) =~0$ by hypothesis, hence $2^{\fr(T_0)} \times \# \varphi^{-1}(f) = 2^{\ndf(f)}$, and Formula (\ref{eq:cardinalzero}) becomes Formula (\ref{eq:equationdulemme}).

It remains to prove Formula (\ref{eq:equationdulemme}) if there exists $T \in \varphi^{-1}(f)$ such that $\mathcal{C}(T)$ is not empty, \textit{i.e.}, if there exists $j \in [n]$ such that $\T_f(j) \neq \emptyset$. Under that hypothesis, let $\left\{j_1, j_2, \hdots, j_p\right\}_< = \left\{j \in [n] : \T_f(j) \neq \emptyset\right\}$. Let $T_p$ be any element of $\T_f(j_p)$, and $\gamma \in \left\{\alpha,\beta\right\}$. We consider $\tilde{T}_p \in \T(T_p,j_p,\gamma)$ (which is not empty in view of Remark~\ref{rem:siTnonvidealors}). For all $T \in \T(T_p,j_p,\gamma)$, we have $\fr(T) = \fr(\tilde{T}_p)$ by Corollary \ref{cor:sifreechezlunetpaschezlautrealorschangementlabel}. Consequently, we obtain
\begin{equation}
\label{eq:cardinal}
\sum_{T \in \T(T_p,j_p,\gamma)} 2^{\fr(T)} = 2^{\fr(\tilde{T}_p)} \times \# \T(T_p,j_p,\gamma). 
\end{equation}
Now, in view of Proposition~\ref{prop:existenceetuniciteduswitch}, the cardinality of $\T(T_p,j_p,\gamma)$ equals $2^{\#\mathcal{S}(\tilde{T}_p)}$, so Formula (\ref{eq:cardinal}) becomes
\begin{equation}
\label{eq:cardinal2}
\sum_{T \in \T(T_p,j_p,\gamma)} 2^{\fr(T)} = 2^{\fr(\tilde{T}_p)+ \#\mathcal{S}(\tilde{T}_p)}. 
\end{equation}
By Remark~\ref{rem:sommedefreeplusSplusCegalegrounded}, we know that 
$\fr(\tilde{T}_p)+ \#\mathcal{S}(\tilde{T}_p) = \ndf(f) - \mathcal{C}(\tilde{T}_p)$, and by hypothesis $\tilde{T}_p \in \T(T_p,j_p)$, so $\mathcal{C}(\tilde{T}_p) = \mathcal{C}(T_p)$, and Formula (\ref{eq:cardinal2}) becomes
\begin{equation}
\label{eq:cardinal3}
\sum_{T \in \T(T_p,j_p,\gamma)} 2^{\fr(T)} = 2^{\ndf(f)-\#\mathcal{C}(T_p)}. 
\end{equation}
Since Formula (\ref{eq:cardinal3}) is true for all $\gamma \in \left\{\alpha,\beta\right\}$, and in view of the equality $\T(T_p,j_p) = \T(T_p,j_p,\alpha) \sqcup \T(T_p,j_p,\beta)$, we obtain
\begin{equation}
\label{eq:cardinal4}
\sum_{T \in \T(T_p,j_p)} 2^{\fr(T)} = 2^{\ndf(f)-\#\mathcal{C}(T_p)+1}. 
\end{equation}
Suppose now that for some $q \in [2,p]$, and for all $T_q \in \T_f(j_q)$, we have the Formula~
\begin{equation}
\label{eq:celleci}
\sum_{T \in \Tab(T_q,j_q)} 2^{\fr(T)} = 2^{\ndf(f)-\#(\mathcal{C}(T_q) \cap [j_q-1])}
\end{equation}
(it is true for $q= p$ in view of Formula~\ref{eq:cardinal4}). Let $T_{q-1} \in \T_f(j_{q-1})$, $\gamma \in \left\{\alpha,\beta\right\}$ and $\tilde{T}_{q-1} \in \T(T_{q-1},j_{q-1},\gamma)$ (which is not empty in view of Remark~\ref{rem:siTnonvidealors}). We first intend to prove the following Formula~:
\begin{equation}
\label{eq:egalitepourgamma}
\sum_{T \in \T(T_{q-1},j_{q-1},\gamma)} 2^{\fr(T)} = 2^{\ndf(f)- \# (\mathcal{C}(T_{q-1}) \cap [j_{q-1}])}. 
\end{equation}
\begin{itemize}
\item If $[j_{q-1}+1,n] \cap \mathcal{C}(\tilde{T}_{q-1}) = \emptyset$, by Lemma~\ref{lem:sijalorsminpourtoutlemonde}, it is necessary that $[j_{q-1}+1,n] \cap \mathcal{C}(T) = \emptyset$ for all $T \in \T(T_{q-1},j_{q-1},\gamma)$, and $\fr(\tilde{T}_{q-1}) = \fr(T)$ by Corollary \ref{cor:sifreechezlunetpaschezlautrealorschangementlabel}, hence
\begin{equation}
\label{eq:cardinalbis}
\sum_{T \in \T(T_{q-1},j_{q-1},\gamma)} 2^{\fr(T)} = 2^{\fr(\tilde{T}_{q-1})} \times \# \T(T_{q-1},j_{q-1},\gamma).
\end{equation}
By Proposition~\ref{prop:existenceetuniciteduswitch}, we have $\# \T(T_{q-1},j_{q-1},\gamma) = 2^{\# \mathcal{S}(\tilde{T}_{q-1})}$. By Remark~\ref{rem:sommedefreeplusSplusCegalegrounded}, the integer $\fr(\tilde{T}_{q-1}) + \mathcal{S}(\tilde{T}_{q-1})$ equals the integer $\ndf(f)-\mathcal{C}(\tilde{T}_{q-1}) = \ndf(f) - \# (\mathcal{C}(T_{q-1}) \cap [j_{q-1}])$ because $\mathcal{C}(\tilde{T}_{q-1}) \cap [j_{q-1}+1,n] = \emptyset$ by hypothesis, hence Formula~(\ref{eq:cardinalbis}) becomes Formula~(\ref{eq:egalitepourgamma}).
\item Otherwise, let $j = \min [j_{q-1}+1,n] \cap \mathcal{C}(\tilde{T}_{q-1})$. By Lemma~\ref{lem:sijalorsminpourtoutlemonde}, it is necessary that $j$ is also $\min [j_{q-1}+1,n] \cap \mathcal{C}(T)$ for all $T \in \T(T_{q-1},j_{q-1},\gamma)$. In other words, the set $\T(T_{q-1},j_{q-1},\gamma)$ is in fact $\T(\tilde{T}_{q-1},j)$.  Let $q' > q-1$ such that $j = j_{q'}$. By hypothesis, we know that 
\begin{equation}
\label{eq:cardinal4deux}
\sum_{T \in \T(\tilde{T}_{q-1},j_{q'})} 2^{\fr(T)} =  2^{\ndf(f) - \# (\mathcal{C}(\tilde{T}_{q-1}) \cap [j_{q'}-1])}.
\end{equation}
Since $j_{q'} = \min \mathcal{C}(\tilde{T}_{q-1}) \cap [j_{q-1}+1,n]$ and $\tilde{T}_{q-1} \in \T(T_{q-1},j_{q-1})$, we have $\# (\mathcal{C}(\tilde{T}_{q-1}) \cap [j_{q'}-1]) = \# (\mathcal{C}(T_q) \cap [j_{q-1}])$, hence Formula~(\ref{eq:cardinal4deux}) becomes Formula~(\ref{eq:egalitepourgamma}) in view of $\T(T_{q-1},j_{q-1},\gamma) = \T(T_{q-1},j_{q-1},\gamma)$.
\end{itemize}
So Formula~(\ref{eq:egalitepourgamma}) is true for all $\gamma \in \left\{\alpha,\beta\right\}$, and in view of $$\T(T_{q-1},j_{q-1}) = \T(T_{q-1},j_{q-1},\alpha) \sqcup \T(T_{q-1},j_{q-1},\beta),$$we obtain
\begin{align*}
\sum_{T \in \T(T_{q-1},j_{q-1})} 2^{\fr(T)} &=  2^{\ndf(f) - \# (\mathcal{C}(T_{q-1}) \cap [j_{q-1}])+1}\\ 
&= 2^{\ndf(f) - \# (\mathcal{C}(T_{q-1}) \cap [j_{q-1}-1])}.
\end{align*}
So Formula~(\ref{eq:celleci}) is true for all $q \in [p]$ by induction. In particular, for $q = 1$, we obtain Formula~(\ref{eq:equationdulemme}).
\end{proof}

This ends the proof of Theorem~\ref{theo:bigeni}.

\section*{Acknowledgements}
Ange Bigeni is affiliated to the National Research University Higher School of Economics (HSE) and may be reached at: \href{mailto:abigeni@hse.ru}{abigeni@hse.ru}.

\begin{appendices}
\section{Pistol labeling of the tableau $T_1 \in \Tab_7$}
\label{annex:pistollabelingT1}

We give in Figure \ref{fig:pistollabelingT1annex} the details of the pistol labeling of the tableau $T_1 \in \Tab_7$ depicted in Figure \ref{fig:T0}. From $j$ from $7$ down to $1$, we show how the two dots of the column $C_j^{T_1}$ receive their pistol labels. In the following, we specify which rule of Part II. of Algorithm \ref{algo:pistollabeling} is applied, for $j$ from $7$ down to $1$.
\begin{itemize}

\item $j = 7$ : Rule II.2-a) for both dots $d_{8}^{T_1}$ and $d_{12}^{T_1}$.
\item $j = 6$ : Rule II.1-a) for $d_{14}^{T_1}$, then Rule II.2-b)ii. for $d_9^{T_1}$.
\item $j = 5$ : Rule II.1-a) for $d_7^{T_1}$ and Rule II.1-b) for $d_{13}^{T_1}$.
\item $j = 4$ : Rule II.1-b) for $d_5^{T_1}$, then Rule II.2-b)ii. for $d_{10}^{T_1}$.
\item $j = 3$ : Rule II.1-b) for $d_6^{T_1}$, then Rule II.2-b)i. for $d_3^{T_1}$.
\item $j = 2$ : Rule II.1-b) for $d_4^{T_1}$, then Rule II.2-b)ii. for $d_2^{T_1}$.
\item $j = 1$ : Rule II.1-a) for $d_{11}^{T_1}$, then Rule II.2-b)i. for $d_1^{T_1}$.

\end{itemize}

\begin{figure}[!htbp]
    \centering

\begin{tikzpicture}[scale=0.4]

\draw (0,2) -- (0,-32);
\draw (0,2) -- (0,-32);
\draw (9,2) -- (9,-32);
\draw (18,2) -- (18,-32);
\draw (0,2) rectangle (27,-32);
\draw (0,-15) rectangle (27,-16);
\draw (0,2) rectangle (27,1);

\draw (4.5,1.5) node[scale=0.8]{$j = 7$};
\draw (13.5,1.5) node[scale=0.8]{$j = 6$};
\draw (22.5,1.5) node[scale=0.8]{$j = 5$};
\draw (4.5,-15.5) node[scale=0.8]{$j = 4$};
\draw (13.5,-15.5) node[scale=0.8]{$j = 3$};
\draw (22.5,-15.5) node[scale=0.8]{$j = 2$};

\begin{scope}[xshift=1cm,yshift=-14cm]
\draw (0,0) grid[step=1] (7,14);
\draw (0,0) -- (7,7);
\draw (0,14) [dashed] -- (7,7);
\fill (0.5,0.5) circle (0.2);
\fill (0.5,10.5) circle (0.2);
\fill (1.5,1.5) circle (0.2);
\fill (1.5,3.5) circle (0.2);
\fill (2.5,2.5) circle (0.2);
\fill (2.5,5.5) circle (0.2);
\fill (3.5,4.5) circle (0.2);
\fill (3.5,9.5) circle (0.2);
\fill (4.5,6.5) circle (0.2);
\draw (4.5,12.5) node[scale=0.6]{$\bigstar$};
\draw (5.5,8.5) node[scale=0.6]{$\bigstar$};
\draw (5.5,13.5) node[scale=0.6]{$\bigstar$};
\draw [red] (6.5,7.5) node[scale=0.8]{$\beta_0^e$};
\draw [blue] (6.5,11.5) node[scale=0.8]{$\alpha_0^o$};

\draw (-0.5,0.5) node[scale=0.7]{$1$};
\draw (-0.5,1.5) node[scale=0.7]{$2$};
\draw (-0.5,2.5) node[scale=0.7]{$3$};
\draw (-0.5,3.5) node[scale=0.7]{$4$};
\draw (-0.5,4.5) node[scale=0.7]{$5$};
\draw (-0.5,5.5) node[scale=0.7]{$6$};
\draw (-0.5,6.5) node[scale=0.7]{$7$};
\draw (-0.5,12.5) node[scale=0.7]{$8$};
\draw (-0.5,11.5) node[scale=0.7]{$9$};
\draw (-0.5,10.5) node[scale=0.7]{$10$};
\draw (-0.5,9.5) node[scale=0.7]{$11$};
\draw (-0.5,8.5) node[scale=0.7]{$12$};
\draw (-0.5,7.5) node[scale=0.7]{$13$};
\draw (-0.5,13.5) node[scale=0.7]{$14$};

\draw [fill=gray,opacity=0.25] (0,12) rectangle (1,13);
\draw [fill=gray,opacity=0.25] (1,11) rectangle (2,12);
\draw [fill=gray,opacity=0.25] (2,10) rectangle (3,11);
\draw [fill=gray,opacity=0.25] (3,9) rectangle (4,10);
\draw [fill=gray,opacity=0.25] (4,8) rectangle (5,9);
\draw [fill=gray,opacity=0.25] (5,7) rectangle (6,8);

\draw [fill=gray,opacity=0.25] (0,0) rectangle (1,1);
\draw [fill=gray,opacity=0.25] (1,1) rectangle (2,2);
\draw [fill=gray,opacity=0.25] (2,2) rectangle (3,3);
\draw [fill=gray,opacity=0.25] (3,3) rectangle (4,4);
\draw [fill=gray,opacity=0.25] (4,4) rectangle (5,5);
\draw [fill=gray,opacity=0.25] (5,5) rectangle (6,6);

\draw [fill=blue,opacity=0.25] (6,6) rectangle (7,7);
\draw [fill=red,opacity=0.25] (6,13) rectangle (7,14);

\draw [red, very thick] (6,7.5) [densely dashed] -- (5.5,7.5);
\draw [red, very thick] (5.5,7.5) [densely dashed] -- (5.5,13.5);
\draw [red, very thick] (5.5,13.5) [densely dashed] -- (6.5,13.5);
\draw [red] (6.5,13.5) circle (0.7);

\draw [blue, very thick] (6,11.5) [densely dashed] -- (1.5,11.5);
\draw [blue, very thick] (1.5,11.5) [densely dashed] -- (1.5,3.5);
\draw [blue, very thick] (1.5,3.5) [densely dashed] -- (3.5,3.5);
\draw [blue, very thick] (3.5,3.5) [densely dashed] -- (3.5,4.5);
\draw [blue, very thick] (3.5,4.5) [densely dashed] -- (4.5,4.5);
\draw [blue, very thick] (4.5,4.5) [densely dashed] -- (4.5,6.5);
\draw [blue, very thick] (4.5,6.5) [densely dashed] -- (6.5,6.5);
\draw [blue] (6.5,6.5) circle (0.7);
\end{scope}

\begin{scope}[xshift=10cm,yshift=-14cm]
\draw (0,0) grid[step=1] (7,14);
\draw (0,0) -- (7,7);
\draw (0,14) [dashed] -- (7,7);
\fill (0.5,0.5) circle (0.2);
\fill (0.5,10.5) circle (0.2);
\fill (1.5,1.5) circle (0.2);
\fill (1.5,3.5) circle (0.2);
\fill (2.5,2.5) circle (0.2);
\fill (2.5,5.5) circle (0.2);
\fill (3.5,4.5) circle (0.2);
\fill (3.5,9.5) circle (0.2);
\fill (4.5,6.5) circle (0.2);
\draw (4.5,12.5) node[scale=0.6]{$\bigstar$};
\draw [blue] (5.5,8.5) node[scale=0.8]{$\alpha_0^o$};
\draw [red] (5.5,13.5) node[scale=0.8]{$\beta_1^e$};
\draw [black] (6.5,7.5) node[scale=0.8]{$\beta_0^e$};
\draw [black] (6.5,11.5) node[scale=0.8]{$\alpha_0^o$};

\draw (-0.5,0.5) node[scale=0.7]{$1$};
\draw (-0.5,1.5) node[scale=0.7]{$2$};
\draw (-0.5,2.5) node[scale=0.7]{$3$};
\draw (-0.5,3.5) node[scale=0.7]{$4$};
\draw (-0.5,4.5) node[scale=0.7]{$5$};
\draw (-0.5,5.5) node[scale=0.7]{$6$};
\draw (-0.5,6.5) node[scale=0.7]{$7$};
\draw (-0.5,12.5) node[scale=0.7]{$8$};
\draw (-0.5,11.5) node[scale=0.7]{$9$};
\draw (-0.5,10.5) node[scale=0.7]{$10$};
\draw (-0.5,9.5) node[scale=0.7]{$11$};
\draw (-0.5,8.5) node[scale=0.7]{$12$};
\draw (-0.5,7.5) node[scale=0.7]{$13$};
\draw (-0.5,13.5) node[scale=0.7]{$14$};

\draw [fill=gray,opacity=0.25] (0,12) rectangle (1,13);
\draw [fill=gray,opacity=0.25] (1,11) rectangle (2,12);
\draw [fill=gray,opacity=0.25] (2,10) rectangle (3,11);
\draw [fill=gray,opacity=0.25] (3,9) rectangle (4,10);
\draw [fill=gray,opacity=0.25] (4,8) rectangle (5,9);
\draw [fill=gray,opacity=0.25] (5,7) rectangle (6,8);

\draw [fill=gray,opacity=0.25] (0,0) rectangle (1,1);
\draw [fill=gray,opacity=0.25] (1,1) rectangle (2,2);
\draw [fill=gray,opacity=0.25] (2,2) rectangle (3,3);
\draw [fill=gray,opacity=0.25] (3,3) rectangle (4,4);
\draw [fill=gray,opacity=0.25] (4,4) rectangle (5,5);
\draw [fill=gray,opacity=0.25] (5,5) rectangle (6,6);

\draw [fill=blue,opacity=0.25] (5,5) rectangle (6,7);
\draw [fill=red,opacity=0.25] (5,13) rectangle (6,14);
\draw [fill=red,opacity=0.25] (5,7) rectangle (6,8);

\draw [red] (5.5,13.5) circle (0.7);

\draw [blue, very thick] (5,8.5) [densely dashed] -- (4.5,8.5);
\draw [blue, very thick] (4.5,8.5) [densely dashed] -- (4.5,12.5);
\draw [blue, very thick] (4.5,12.5) [densely dashed] -- (0.5,12.5);
\draw [blue, very thick] (0.5,12.5) [densely dashed] -- (0.5,10.5);
\draw [blue, very thick] (0.5,10.5) [densely dashed] -- (2.5,10.5);
\draw [blue, very thick] (2.5,10.5) [densely dashed] -- (2.5,5.5);
\draw [blue, very thick] (2.5,5.5) [densely dashed] -- (5.5,5.5);
\draw [blue] (5.5,5.5) circle (0.7);
\end{scope}

\begin{scope}[xshift=19cm,yshift=-14cm]
\draw (0,0) grid[step=1] (7,14);
\draw (0,0) -- (7,7);
\draw (0,14) [dashed] -- (7,7);
\fill (0.5,0.5) circle (0.2);
\fill (0.5,10.5) circle (0.2);
\fill (1.5,1.5) circle (0.2);
\fill (1.5,3.5) circle (0.2);
\fill (2.5,2.5) circle (0.2);
\fill (2.5,5.5) circle (0.2);
\fill (3.5,4.5) circle (0.2);
\fill (3.5,9.5) circle (0.2);
\draw [blue] (4.5,6.5) node[scale=0.8]{$\alpha_2^o$};
\draw [blue] (4.5,12.5) node[scale=0.8]{$\alpha_1^e$};
\draw [black] (5.5,8.5) node[scale=0.8]{$\alpha_0^o$};
\draw [black] (5.5,13.5) node[scale=0.8]{$\beta_1^e$};
\draw [red] (6.5,7.5) node[scale=0.8]{\underline{$\beta_0^e$}};
\draw [blue] (6.5,11.5) node[scale=0.8]{\underline{$\alpha_0^o$}};

\draw (-0.5,0.5) node[scale=0.7]{$1$};
\draw (-0.5,1.5) node[scale=0.7]{$2$};
\draw (-0.5,2.5) node[scale=0.7]{$3$};
\draw (-0.5,3.5) node[scale=0.7]{$4$};
\draw (-0.5,4.5) node[scale=0.7]{$5$};
\draw (-0.5,5.5) node[scale=0.7]{$6$};
\draw (-0.5,6.5) node[scale=0.7]{$7$};
\draw (-0.5,12.5) node[scale=0.7]{$8$};
\draw (-0.5,11.5) node[scale=0.7]{$9$};
\draw (-0.5,10.5) node[scale=0.7]{$10$};
\draw (-0.5,9.5) node[scale=0.7]{$11$};
\draw (-0.5,8.5) node[scale=0.7]{$12$};
\draw (-0.5,7.5) node[scale=0.7]{$13$};
\draw (-0.5,13.5) node[scale=0.7]{$14$};

\draw [fill=gray,opacity=0.25] (0,12) rectangle (1,13);
\draw [fill=gray,opacity=0.25] (1,11) rectangle (2,12);
\draw [fill=gray,opacity=0.25] (2,10) rectangle (3,11);
\draw [fill=gray,opacity=0.25] (3,9) rectangle (4,10);
\draw [fill=gray,opacity=0.25] (4,8) rectangle (5,9);
\draw [fill=gray,opacity=0.25] (5,7) rectangle (6,8);

\draw [fill=gray,opacity=0.25] (0,0) rectangle (1,1);
\draw [fill=gray,opacity=0.25] (1,1) rectangle (2,2);
\draw [fill=gray,opacity=0.25] (2,2) rectangle (3,3);
\draw [fill=gray,opacity=0.25] (3,3) rectangle (4,4);
\draw [fill=gray,opacity=0.25] (4,4) rectangle (5,5);
\draw [fill=gray,opacity=0.25] (5,5) rectangle (6,6);

\draw [fill=blue,opacity=0.25] (4,4) rectangle (5,7);
\draw [fill=red,opacity=0.25] (4,13) rectangle (5,14);
\draw [fill=red,opacity=0.25] (4,7) rectangle (5,9);

\draw [blue, very thick] (4,12.5) [densely dashed] -- (0.5,12.5);
\draw [blue, very thick] (0.5,12.5) [densely dashed] -- (0.5,10.5);
\draw [blue, very thick] (0.5,10.5) [densely dashed] -- (2.5,10.5);
\draw [blue, very thick] (2.5,10.5) [densely dashed] -- (2.5,5.5);
\draw [blue, very thick] (2.5,5.5) [densely dashed] -- (4.5,5.5);
\draw [blue] (4.5,5.5) circle (0.7);

\draw [blue] (4.5,6.5) circle (0.7);

\end{scope}

\begin{scope}[xshift=1cm,yshift=-31cm]
\draw (0,0) grid[step=1] (7,14);
\draw (0,0) -- (7,7);
\draw (0,14) [dashed] -- (7,7);
\fill (0.5,0.5) circle (0.2);
\fill (0.5,10.5) circle (0.2);
\fill (1.5,1.5) circle (0.2);
\fill (1.5,3.5) circle (0.2);
\fill (2.5,2.5) circle (0.2);
\fill (2.5,5.5) circle (0.2);
\draw [red] (3.5,4.5) node[scale=0.8]{$\beta_1^e$};
\draw [red] (3.5,9.5) node[scale=0.8]{$\beta_0^o$};
\draw [blue] (4.5,6.5) node[scale=0.8]{\underline{$\alpha_2^o$}};
\draw [blue] (4.5,12.5) node[scale=0.8]{\underline{$\alpha_1^e$}};
\draw [black] (5.5,8.5) node[scale=0.8]{$\alpha_0^o$};
\draw [black] (5.5,13.5) node[scale=0.8]{$\beta_1^e$};
\draw [black] (6.5,7.5) node[scale=0.8]{$\beta_0^e$};
\draw [black] (6.5,11.5) node[scale=0.8]{$\alpha_0^o$};

\draw (-0.5,0.5) node[scale=0.7]{$1$};
\draw (-0.5,1.5) node[scale=0.7]{$2$};
\draw (-0.5,2.5) node[scale=0.7]{$3$};
\draw (-0.5,3.5) node[scale=0.7]{$4$};
\draw (-0.5,4.5) node[scale=0.7]{$5$};
\draw (-0.5,5.5) node[scale=0.7]{$6$};
\draw (-0.5,6.5) node[scale=0.7]{$7$};
\draw (-0.5,12.5) node[scale=0.7]{$8$};
\draw (-0.5,11.5) node[scale=0.7]{$9$};
\draw (-0.5,10.5) node[scale=0.7]{$10$};
\draw (-0.5,9.5) node[scale=0.7]{$11$};
\draw (-0.5,8.5) node[scale=0.7]{$12$};
\draw (-0.5,7.5) node[scale=0.7]{$13$};
\draw (-0.5,13.5) node[scale=0.7]{$14$};

\draw [fill=gray,opacity=0.25] (0,12) rectangle (1,13);
\draw [fill=gray,opacity=0.25] (1,11) rectangle (2,12);
\draw [fill=gray,opacity=0.25] (2,10) rectangle (3,11);
\draw [fill=gray,opacity=0.25] (3,9) rectangle (4,10);
\draw [fill=gray,opacity=0.25] (4,8) rectangle (5,9);
\draw [fill=gray,opacity=0.25] (5,7) rectangle (6,8);

\draw [fill=gray,opacity=0.25] (0,0) rectangle (1,1);
\draw [fill=gray,opacity=0.25] (1,1) rectangle (2,2);
\draw [fill=gray,opacity=0.25] (2,2) rectangle (3,3);
\draw [fill=gray,opacity=0.25] (3,3) rectangle (4,4);
\draw [fill=gray,opacity=0.25] (4,4) rectangle (5,5);
\draw [fill=gray,opacity=0.25] (5,5) rectangle (6,6);

\draw [fill=blue,opacity=0.25] (3,3) rectangle (4,7);
\draw [fill=red,opacity=0.25] (3,13) rectangle (4,14);
\draw [fill=red,opacity=0.25] (3,7) rectangle (4,10);

\draw [red] (3.5,9.5) circle (0.7);
\draw [red] (3.5,4.5) circle (0.7);

\end{scope}

\begin{scope}[xshift=10cm,yshift=-31cm]
\draw (0,0) grid[step=1] (7,14);
\draw (0,0) -- (7,7);
\draw (0,14) [dashed] -- (7,7);
\fill (0.5,0.5) circle (0.2);
\fill (0.5,10.5) circle (0.2);
\fill (1.5,1.5) circle (0.2);
\fill (1.5,3.5) circle (0.2);
\draw [red] (2.5,2.5) node[scale=0.8]{$\beta_0^e$};
\draw [blue] (2.5,5.5) node[scale=0.8]{$\alpha_3^o$};
\draw [black] (3.5,4.5) node[scale=0.8]{$\beta_1^e$};
\draw [black] (3.5,9.5) node[scale=0.8]{$\beta_0^o$};
\draw [black] (4.5,6.5) node[scale=0.8]{$\alpha_2^o$};
\draw [black] (4.5,12.5) node[scale=0.8]{$\alpha_1^e$};
\draw [blue] (5.5,8.5) node[scale=0.8]{\underline{$\alpha_0^o$}};
\draw [red] (5.5,13.5) node[scale=0.8]{\underline{$\beta_1^e$}};
\draw [black] (6.5,7.5) node[scale=0.8]{$\beta_0^e$};
\draw [black] (6.5,11.5) node[scale=0.8]{$\alpha_0^o$};

\draw (-0.5,0.5) node[scale=0.7]{$1$};
\draw (-0.5,1.5) node[scale=0.7]{$2$};
\draw (-0.5,2.5) node[scale=0.7]{$3$};
\draw (-0.5,3.5) node[scale=0.7]{$4$};
\draw (-0.5,4.5) node[scale=0.7]{$5$};
\draw (-0.5,5.5) node[scale=0.7]{$6$};
\draw (-0.5,6.5) node[scale=0.7]{$7$};
\draw (-0.5,12.5) node[scale=0.7]{$8$};
\draw (-0.5,11.5) node[scale=0.7]{$9$};
\draw (-0.5,10.5) node[scale=0.7]{$10$};
\draw (-0.5,9.5) node[scale=0.7]{$11$};
\draw (-0.5,8.5) node[scale=0.7]{$12$};
\draw (-0.5,7.5) node[scale=0.7]{$13$};
\draw (-0.5,13.5) node[scale=0.7]{$14$};

\draw [fill=gray,opacity=0.25] (0,12) rectangle (1,13);
\draw [fill=gray,opacity=0.25] (1,11) rectangle (2,12);
\draw [fill=gray,opacity=0.25] (2,10) rectangle (3,11);
\draw [fill=gray,opacity=0.25] (3,9) rectangle (4,10);
\draw [fill=gray,opacity=0.25] (4,8) rectangle (5,9);
\draw [fill=gray,opacity=0.25] (5,7) rectangle (6,8);

\draw [fill=gray,opacity=0.25] (0,0) rectangle (1,1);
\draw [fill=gray,opacity=0.25] (1,1) rectangle (2,2);
\draw [fill=gray,opacity=0.25] (2,2) rectangle (3,3);
\draw [fill=gray,opacity=0.25] (3,3) rectangle (4,4);
\draw [fill=gray,opacity=0.25] (4,4) rectangle (5,5);
\draw [fill=gray,opacity=0.25] (5,5) rectangle (6,6);

\draw [fill=blue,opacity=0.25] (2,2) rectangle (3,7);
\draw [fill=red,opacity=0.25] (2,13) rectangle (3,14);
\draw [fill=red,opacity=0.25] (2,7) rectangle (3,11);

\draw [blue] (2.5,5.5) circle (0.7);
\draw [red] (2.5,2.5) circle (0.7);

\end{scope}

\begin{scope}[xshift=19cm,yshift=-31cm]
\draw (0,0) grid[step=1] (7,14);
\draw (0,0) -- (7,7);
\draw (0,14) [dashed] -- (7,7);
\fill (0.5,0.5) circle (0.2);
\fill (0.5,10.5) circle (0.2);
\draw [blue] (1.5,1.5) node[scale=0.8]{$\alpha_0^o$};
\draw [red] (1.5,3.5) node[scale=0.8]{$\beta_2^e$};
\draw [black] (2.5,2.5) node[scale=0.8]{$\beta_0^e$};
\draw [black] (2.5,5.5) node[scale=0.8]{$\alpha_3^o$};
\draw [red] (3.5,4.5) node[scale=0.8]{\underline{$\beta_1^e$}};
\draw [red] (3.5,9.5) node[scale=0.8]{\underline{$\beta_0^o$}};
\draw [black] (4.5,6.5) node[scale=0.8]{$\alpha_2^o$};
\draw [black] (4.5,12.5) node[scale=0.8]{$\alpha_1^e$};
\draw [black] (5.5,8.5) node[scale=0.8]{$\alpha_0^o$};
\draw [black] (5.5,13.5) node[scale=0.8]{$\beta_1^e$};
\draw [black] (6.5,7.5) node[scale=0.8]{$\beta_0^e$};
\draw [black] (6.5,11.5) node[scale=0.8]{$\alpha_0^o$};

\draw [fill=gray,opacity=0.25] (0,12) rectangle (1,13);
\draw [fill=gray,opacity=0.25] (1,11) rectangle (2,12);
\draw [fill=gray,opacity=0.25] (2,10) rectangle (3,11);
\draw [fill=gray,opacity=0.25] (3,9) rectangle (4,10);
\draw [fill=gray,opacity=0.25] (4,8) rectangle (5,9);
\draw [fill=gray,opacity=0.25] (5,7) rectangle (6,8);

\draw [fill=gray,opacity=0.25] (0,0) rectangle (1,1);
\draw [fill=gray,opacity=0.25] (1,1) rectangle (2,2);
\draw [fill=gray,opacity=0.25] (2,2) rectangle (3,3);
\draw [fill=gray,opacity=0.25] (3,3) rectangle (4,4);
\draw [fill=gray,opacity=0.25] (4,4) rectangle (5,5);
\draw [fill=gray,opacity=0.25] (5,5) rectangle (6,6);

\draw (-0.5,0.5) node[scale=0.7]{$1$};
\draw (-0.5,1.5) node[scale=0.7]{$2$};
\draw (-0.5,2.5) node[scale=0.7]{$3$};
\draw (-0.5,3.5) node[scale=0.7]{$4$};
\draw (-0.5,4.5) node[scale=0.7]{$5$};
\draw (-0.5,5.5) node[scale=0.7]{$6$};
\draw (-0.5,6.5) node[scale=0.7]{$7$};
\draw (-0.5,12.5) node[scale=0.7]{$8$};
\draw (-0.5,11.5) node[scale=0.7]{$9$};
\draw (-0.5,10.5) node[scale=0.7]{$10$};
\draw (-0.5,9.5) node[scale=0.7]{$11$};
\draw (-0.5,8.5) node[scale=0.7]{$12$};
\draw (-0.5,7.5) node[scale=0.7]{$13$};
\draw (-0.5,13.5) node[scale=0.7]{$14$};

\draw [fill=blue,opacity=0.25] (1,1) rectangle (2,7);
\draw [fill=red,opacity=0.25] (1,13) rectangle (2,14);
\draw [fill=red,opacity=0.25] (1,7) rectangle (2,12);

\draw [red] (1.5,3.5) circle (0.7);
\draw [blue] (1.5,1.5) circle (0.7);

\end{scope}

\begin{scope}[xshift=9cm,yshift=-32cm]

\draw (0,0) rectangle (9,-17);
\draw (0,0) rectangle (9,-1);

\draw (4.5,-0.5) node[scale=0.8]{$j = 1$};

\begin{scope}[xshift=1cm,yshift=-16cm]
\draw (0,0) grid[step=1] (7,14);
\draw (0,0) -- (7,7);
\draw (0,14) [dashed] -- (7,7);
\draw [blue] (0.5,0.5) node[scale=0.8]{$\alpha_0^o$};
\draw [red] (0.5,10.5) node[scale=0.8]{$\beta_2^e$};
\draw [black] (1.5,1.5) node[scale=0.8]{$\alpha_0^o$};
\draw [black] (1.5,3.5) node[scale=0.8]{$\beta_2^e$};
\draw [black] (2.5,2.5) node[scale=0.8]{\underline{$\beta_0^e$}};
\draw [black] (2.5,5.5) node[scale=0.8]{$\alpha_3^o$};
\draw [black] (3.5,4.5) node[scale=0.8]{$\beta_1^e$};
\draw [black] (3.5,9.5) node[scale=0.8]{$\beta_0^o$};
\draw [black] (4.5,6.5) node[scale=0.8]{$\alpha_2^o$};
\draw [black] (4.5,12.5) node[scale=0.8]{$\alpha_1^e$};
\draw [black] (5.5,8.5) node[scale=0.8]{$\alpha_0^o$};
\draw [black] (5.5,13.5) node[scale=0.8]{$\beta_1^e$};
\draw [black] (6.5,7.5) node[scale=0.8]{$\beta_0^e$};
\draw [black] (6.5,11.5) node[scale=0.8]{$\alpha_0^o$};

\draw [fill=gray,opacity=0.25] (0,12) rectangle (1,13);
\draw [fill=gray,opacity=0.25] (1,11) rectangle (2,12);
\draw [fill=gray,opacity=0.25] (2,10) rectangle (3,11);
\draw [fill=gray,opacity=0.25] (3,9) rectangle (4,10);
\draw [fill=gray,opacity=0.25] (4,8) rectangle (5,9);
\draw [fill=gray,opacity=0.25] (5,7) rectangle (6,8);

\draw [fill=gray,opacity=0.25] (0,0) rectangle (1,1);
\draw [fill=gray,opacity=0.25] (1,1) rectangle (2,2);
\draw [fill=gray,opacity=0.25] (2,2) rectangle (3,3);
\draw [fill=gray,opacity=0.25] (3,3) rectangle (4,4);
\draw [fill=gray,opacity=0.25] (4,4) rectangle (5,5);
\draw [fill=gray,opacity=0.25] (5,5) rectangle (6,6);

\draw (-0.5,0.5) node[scale=0.7]{$1$};
\draw (-0.5,1.5) node[scale=0.7]{$2$};
\draw (-0.5,2.5) node[scale=0.7]{$3$};
\draw (-0.5,3.5) node[scale=0.7]{$4$};
\draw (-0.5,4.5) node[scale=0.7]{$5$};
\draw (-0.5,5.5) node[scale=0.7]{$6$};
\draw (-0.5,6.5) node[scale=0.7]{$7$};
\draw (-0.5,12.5) node[scale=0.7]{$8$};
\draw (-0.5,11.5) node[scale=0.7]{$9$};
\draw (-0.5,10.5) node[scale=0.7]{$10$};
\draw (-0.5,9.5) node[scale=0.7]{$11$};
\draw (-0.5,8.5) node[scale=0.7]{$12$};
\draw (-0.5,7.5) node[scale=0.7]{$13$};
\draw (-0.5,13.5) node[scale=0.7]{$14$};

\draw [fill=blue,opacity=0.25] (0,0) rectangle (1,7);
\draw [fill=red,opacity=0.25] (0,7) rectangle (1,14);

\draw [red] (0.5,10.5) circle (0.7);
\draw [blue] (0.5,0.5) circle (0.7);

\end{scope}

\end{scope}

\end{tikzpicture}

 \caption{Pistol labeling of $T_1 \in \Tab_7$.}
\label{fig:pistollabelingT1annex}

\end{figure}

\section{Computation of $\Phi(f_1)$}
\label{annex:insertionlabelingf1}

We give in Figure \ref{fig:computationPhif1annex} the details of the computation of $\Phi(f_1) \in \Tab_7$ where $f_1 \in \SP_7$ is the surjective pistol depicted in Figure \ref{fig:f0}. From $j$ from $1$ to $7$, we show how the two labeled dots of $C_j^{\Phi(f_1)}$ are inserted. At each step $j$, on the left of every suitable row, we specify the integer $\delta \in [0,7-j]$ it corresponds with (in blue, for dots labeled with $a$, and in red for dots labeled with $b$). In the following table, we make explicit every rule of Algorithm \ref{algo:insertionlabeling} and Definition~\ref{defi:deltainsertions} that leads to the plotting of the dots of $C_j^{\Phi(f_1)}$.

\begin{center}
\begin{center}
  \begin{tabular}{ c | c | c}
    $j$ & Rule of Algorithm \ref{algo:insertionlabeling} & Rules of Definition~\ref{defi:deltainsertions}  \\ \hline
    1 & I.1- & 1.(a) and 2.(a)i. \\ \hline
        2 & I.1- & 1.(a) and 2.(a)ii. \\ \hline
            3 & I.1- & 1.(a) and 2.(a)ii. \\ \hline
                4 & I.2-b)i. & 1.(b) and 2.(a)ii. \\ \hline
                    5 & I.2-b)ii. & 2.(b) and 2.(a)ii. \\ \hline
                        6 & I.2-a) & 2.(b) and 2.(a)i. \\ \hline
                            7 & I.1- & 1.(b) and 1.(b)
  \end{tabular}
\end{center}
\end{center}

\begin{figure}[!htbp]
    \centering

\begin{tikzpicture}[scale=0.4]

\draw (0,2) -- (0,-32);
\draw (0,2) -- (0,-32);
\draw (9,2) -- (9,-32);
\draw (18,2) -- (18,-32);
\draw (0,2) rectangle (27,-32);
\draw (0,-15) rectangle (27,-16);
\draw (0,2) rectangle (27,1);

\draw (4.5,1.5) node[scale=0.8]{$j=1$};
\draw (13.5,1.5) node[scale=0.8]{$j=2$};
\draw (22.5,1.5) node[scale=0.8]{$j=3$};
\draw (4.5,-15.5) node[scale=0.8]{$j=4$};
\draw (13.5,-15.5) node[scale=0.8]{$j=5$};
\draw (22.5,-15.5) node[scale=0.8]{$j=6$};

\begin{scope}[xshift=1cm,yshift=-14cm]
\draw (0,0) grid[step=1] (7,14);
\draw (0,0) -- (7,7);
\draw (0,14) [dashed] -- (7,7);
\draw [color = blue] (0.5,0.5) node[scale=1]{$a$};
\draw [color = red] (0.5,10.5) node[scale=1]{$b$};

\draw [color = blue] (0.5,0.5) circle (0.7);
\draw [color = red] (0.5,10.5) circle (0.7);

\draw [color=blue] (-0.5,0.5) node[scale=0.7]{$0$};
\draw (-0.5,1.5) node[scale=0.7]{$1$};
\draw (-0.5,2.5) node[scale=0.7]{$2$};
\draw (-0.5,3.5) node[scale=0.7]{$3$};
\draw (-0.5,4.5) node[scale=0.7]{$4$};
\draw (-0.5,5.5) node[scale=0.7]{$5$};
\draw (-0.5,6.5) node[scale=0.7]{$6$};
\draw (-0.5,12.5) node[scale=0.7]{$0$};
\draw (-0.5,11.5) node[scale=0.7]{$1$};
\draw [color=red] (-0.5,10.5) node[scale=0.7]{$2$};
\draw (-0.5,9.5) node[scale=0.7]{$3$};
\draw (-0.5,8.5) node[scale=0.7]{$4$};
\draw (-0.5,7.5) node[scale=0.7]{$5$};
\draw (-0.5,13.5) node[scale=0.7]{$6$};

\draw [fill=gray,opacity=0.25] (0,12) rectangle (1,13);
\draw [fill=gray,opacity=0.25] (1,11) rectangle (2,12);
\draw [fill=gray,opacity=0.25] (2,10) rectangle (3,11);
\draw [fill=gray,opacity=0.25] (3,9) rectangle (4,10);
\draw [fill=gray,opacity=0.25] (4,8) rectangle (5,9);
\draw [fill=gray,opacity=0.25] (5,7) rectangle (6,8);

\draw [fill=gray,opacity=0.25] (0,0) rectangle (1,1);
\draw [fill=gray,opacity=0.25] (1,1) rectangle (2,2);
\draw [fill=gray,opacity=0.25] (2,2) rectangle (3,3);
\draw [fill=gray,opacity=0.25] (3,3) rectangle (4,4);
\draw [fill=gray,opacity=0.25] (4,4) rectangle (5,5);
\draw [fill=gray,opacity=0.25] (5,5) rectangle (6,6);

\draw [fill=blue,opacity=0.25] (0,0) rectangle (1,7);
\draw [fill=red,opacity=0.25] (0,7) rectangle (1,14);
\end{scope}

\begin{scope}[xshift=10cm,yshift=-14cm]
\draw (0,0) grid[step=1] (7,14);
\draw (0,0) -- (7,7);
\draw (0,14) [dashed] -- (7,7);
\draw (0.5,0.5) node[scale=1]{$a$};
\draw (0.5,10.5) node[scale=1]{$b$};
\draw [color = blue] (1.5,1.5) node[scale=1]{$a$};
\draw [color = red] (1.5,3.5) node[scale=1]{$b$};

\draw [color = red] (1.5,3.5) circle (0.7);
\draw [color = blue] (1.5,1.5) circle (0.7);

\draw [fill=blue,opacity=0.25] (1,1) rectangle (2,7);
\draw [fill=red,opacity=0.25] (1,13) rectangle (2,14);
\draw [fill=red,opacity=0.25] (1,7) rectangle (2,12);

\draw [fill=gray,opacity=0.25] (0,12) rectangle (1,13);
\draw [fill=gray,opacity=0.25] (1,11) rectangle (2,12);
\draw [fill=gray,opacity=0.25] (2,10) rectangle (3,11);
\draw [fill=gray,opacity=0.25] (3,9) rectangle (4,10);
\draw [fill=gray,opacity=0.25] (4,8) rectangle (5,9);
\draw [fill=gray,opacity=0.25] (5,7) rectangle (6,8);

\draw [fill=gray,opacity=0.25] (0,0) rectangle (1,1);
\draw [fill=gray,opacity=0.25] (1,1) rectangle (2,2);
\draw [fill=gray,opacity=0.25] (2,2) rectangle (3,3);
\draw [fill=gray,opacity=0.25] (3,3) rectangle (4,4);
\draw [fill=gray,opacity=0.25] (4,4) rectangle (5,5);
\draw [fill=gray,opacity=0.25] (5,5) rectangle (6,6);

\draw [color = blue] (-0.5,1.5) node[scale=0.7]{$0$};
\draw (-0.5,2.5) node[scale=0.7]{$1$};
\draw [color = red] (-0.5,3.5) node[scale=0.7]{$2$};
\draw (-0.5,4.5) node[scale=0.7]{$3$};
\draw (-0.5,5.5) node[scale=0.7]{$4$};
\draw (-0.5,6.5) node[scale=0.7]{$5$};
\draw (-0.5,11.5) node[scale=0.7]{$0$};
\draw (-0.5,10.5) node[scale=0.7]{$1$};
\draw (-0.5,9.5) node[scale=0.7]{$2$};
\draw (-0.5,8.5) node[scale=0.7]{$3$};
\draw (-0.5,7.5) node[scale=0.7]{$4$};
\draw (-0.5,13.5) node[scale=0.7]{$5$};
\end{scope}

\begin{scope}[xshift=19cm,yshift=-14cm]
\draw (0,0) grid[step=1] (7,14);
\draw (0,0) -- (7,7);
\draw (0,14) [dashed] -- (7,7);
\draw (0.5,0.5) node[scale=1]{$a$};
\draw (0.5,10.5) node[scale=1]{$b$};
\draw (1.5,1.5) node[scale=1]{$a$};
\draw (1.5,3.5) node[scale=1]{$b$};
\draw [color = red] (2.5,2.5) node[scale=1]{$b$};
\draw [color = blue] (2.5,5.5) node[scale=1]{$a$};

\draw [color = red] (2.5,2.5) circle (0.7);
\draw [color = blue] (2.5,5.5) circle (0.7);

\draw [color=red] (-0.5,2.5) node[scale=0.7]{$0$};
\draw (-0.5,3.5) node[scale=0.7]{$1$};
\draw (-0.5,4.5) node[scale=0.7]{$2$};
\draw [color = blue] (-0.5,5.5) node[scale=0.7]{$3$};
\draw (-0.5,6.5) node[scale=0.7]{$4$};
\draw (-0.5,10.5) node[scale=0.7]{$0$};
\draw (-0.5,9.5) node[scale=0.7]{$1$};
\draw (-0.5,8.5) node[scale=0.7]{$2$};
\draw (-0.5,7.5) node[scale=0.7]{$3$};
\draw (-0.5,13.5) node[scale=0.7]{$4$};

\draw [fill=gray,opacity=0.25] (0,12) rectangle (1,13);
\draw [fill=gray,opacity=0.25] (1,11) rectangle (2,12);
\draw [fill=gray,opacity=0.25] (2,10) rectangle (3,11);
\draw [fill=gray,opacity=0.25] (3,9) rectangle (4,10);
\draw [fill=gray,opacity=0.25] (4,8) rectangle (5,9);
\draw [fill=gray,opacity=0.25] (5,7) rectangle (6,8);

\draw [fill=gray,opacity=0.25] (0,0) rectangle (1,1);
\draw [fill=gray,opacity=0.25] (1,1) rectangle (2,2);
\draw [fill=gray,opacity=0.25] (2,2) rectangle (3,3);
\draw [fill=gray,opacity=0.25] (3,3) rectangle (4,4);
\draw [fill=gray,opacity=0.25] (4,4) rectangle (5,5);
\draw [fill=gray,opacity=0.25] (5,5) rectangle (6,6);

\draw [fill=blue,opacity=0.25] (2,2) rectangle (3,7);
\draw [fill=red,opacity=0.25] (2,13) rectangle (3,14);
\draw [fill=red,opacity=0.25] (2,7) rectangle (3,11);
\end{scope}

\begin{scope}[xshift=1cm,yshift=-31cm]
\draw (0,0) grid[step=1] (7,14);
\draw (0,0) -- (7,7);
\draw (0,14) [dashed] -- (7,7);
\draw (0.5,0.5) node[scale=1]{$a$};
\draw (0.5,10.5) node[scale=1]{$b$};
\draw (1.5,1.5) node[scale=1]{$a$};
\draw (1.5,3.5) node[scale=1]{$b$};
\draw (2.5,2.5) node[scale=1]{$b$};
\draw (2.5,5.5) node[scale=1]{$a$};
\draw [color = red] (3.5,4.5) node[scale=1]{$b$};
\draw [color = red] (3.5,9.5) node[scale=1]{$b$};

\draw [color = red] (3.5,4.5) circle (0.7);
\draw [color = red] (3.5,9.5) circle (0.7);

\draw (-0.5,3.5) node[scale=0.7]{$0$};
\draw [color = red] (-0.5,4.5) node[scale=0.7]{$1$};
\draw (-0.5,5.5) node[scale=0.7]{$2$};
\draw (-0.5,6.5) node[scale=0.7]{$3$};
\draw [color = red] (-0.5,9.5) node[scale=0.7]{$0$};
\draw (-0.5,8.5) node[scale=0.7]{$1$};
\draw (-0.5,7.5) node[scale=0.7]{$2$};
\draw (-0.5,13.5) node[scale=0.7]{$3$};

\draw [fill=gray,opacity=0.25] (0,12) rectangle (1,13);
\draw [fill=gray,opacity=0.25] (1,11) rectangle (2,12);
\draw [fill=gray,opacity=0.25] (2,10) rectangle (3,11);
\draw [fill=gray,opacity=0.25] (3,9) rectangle (4,10);
\draw [fill=gray,opacity=0.25] (4,8) rectangle (5,9);
\draw [fill=gray,opacity=0.25] (5,7) rectangle (6,8);

\draw [fill=gray,opacity=0.25] (0,0) rectangle (1,1);
\draw [fill=gray,opacity=0.25] (1,1) rectangle (2,2);
\draw [fill=gray,opacity=0.25] (2,2) rectangle (3,3);
\draw [fill=gray,opacity=0.25] (3,3) rectangle (4,4);
\draw [fill=gray,opacity=0.25] (4,4) rectangle (5,5);
\draw [fill=gray,opacity=0.25] (5,5) rectangle (6,6);

\draw [fill=blue,opacity=0.25] (3,3) rectangle (4,7);
\draw [fill=red,opacity=0.25] (3,13) rectangle (4,14);
\draw [fill=red,opacity=0.25] (3,7) rectangle (4,10);
\end{scope}

\begin{scope}[xshift=10cm,yshift=-31cm]
\draw (0,0) grid[step=1] (7,14);
\draw (0,0) -- (7,7);
\draw (0,14) [dashed] -- (7,7);
\draw (0.5,0.5) node[scale=1]{$a$};
\draw (0.5,10.5) node[scale=1]{$b$};
\draw (1.5,1.5) node[scale=1]{$a$};
\draw (1.5,3.5) node[scale=1]{$b$};
\draw (2.5,2.5) node[scale=1]{$b$};
\draw (2.5,5.5) node[scale=1]{$a$};
\draw (3.5,4.5) node[scale=1]{$b$};
\draw (3.5,9.5) node[scale=1]{$b$};
\draw [color = blue] (4.5,6.5) node[scale=1]{$a$};
\draw [color = blue] (4.5,12.5) node[scale=1]{$a$};

\draw [color = blue] (4.5,6.5) circle (0.7);

\draw [color = blue] (4.5,5.5) circle (0.7);
\draw [blue, very thick] (4.5,5.5) [densely dashed] -- (2.5,5.5);
\draw [blue, very thick] (2.5,5.5) [densely dashed] -- (2.5,10.5);
\draw [blue, very thick] (2.5,10.5) [densely dashed] -- (0.5,10.5);
\draw [blue, very thick] (0.5,10.5) [densely dashed] -- (0.5,12.5);
\draw [blue, very thick] (0.5,12.5) [densely dashed] -- (4.3,12.5);

\draw [fill=gray,opacity=0.25] (0,12) rectangle (1,13);
\draw [fill=gray,opacity=0.25] (1,11) rectangle (2,12);
\draw [fill=gray,opacity=0.25] (2,10) rectangle (3,11);
\draw [fill=gray,opacity=0.25] (3,9) rectangle (4,10);
\draw [fill=gray,opacity=0.25] (4,8) rectangle (5,9);
\draw [fill=gray,opacity=0.25] (5,7) rectangle (6,8);

\draw [fill=gray,opacity=0.25] (0,0) rectangle (1,1);
\draw [fill=gray,opacity=0.25] (1,1) rectangle (2,2);
\draw [fill=gray,opacity=0.25] (2,2) rectangle (3,3);
\draw [fill=gray,opacity=0.25] (3,3) rectangle (4,4);
\draw [fill=gray,opacity=0.25] (4,4) rectangle (5,5);
\draw [fill=gray,opacity=0.25] (5,5) rectangle (6,6);

\draw (-0.5,4.5) node[scale=0.7]{$0$};
\draw [color=blue] (-0.5,5.5) node[scale=0.7]{$1$};
\draw [color=blue] (-0.5,6.5) node[scale=0.7]{$2$};
\draw (-0.5,8.5) node[scale=0.7]{$0$};
\draw (-0.5,7.5) node[scale=0.7]{$1$};
\draw (-0.5,13.5) node[scale=0.7]{$2$};

\draw [fill=blue,opacity=0.25] (4,4) rectangle (5,7);
\draw [fill=red,opacity=0.25] (4,13) rectangle (5,14);
\draw [fill=red,opacity=0.25] (4,7) rectangle (5,9);
\end{scope}

\begin{scope}[xshift=19cm,yshift=-31cm]
\draw (0,0) grid[step=1] (7,14);
\draw (0,0) -- (7,7);
\draw (0,14) [dashed] -- (7,7);
\draw (0.5,0.5) node[scale=1]{$a$};
\draw (0.5,10.5) node[scale=1]{$b$};
\draw (1.5,1.5) node[scale=1]{$a$};
\draw (1.5,3.5) node[scale=1]{$b$};
\draw (2.5,2.5) node[scale=1]{$b$};
\draw (2.5,5.5) node[scale=1]{$a$};
\draw (3.5,4.5) node[scale=1]{$b$};
\draw (3.5,9.5) node[scale=1]{$b$};
\draw (4.5,6.5) node[scale=1]{$a$};
\draw (4.5,12.5) node[scale=1]{$a$};
\draw [color = blue] (5.5,8.5) node[scale=1]{$a$};
\draw [color = red] (5.5,13.5) node[scale=1]{$b$};

\draw [color = blue] (5.5,5.5) circle (0.7);
\draw [blue, very thick] (5.5,5.5) [densely dashed] -- (2.5,5.5);
\draw [blue, very thick] (2.5,5.5) [densely dashed] -- (2.5,10.5);
\draw [blue, very thick] (2.5,10.5) [densely dashed] -- (0.5,10.5);
\draw [blue, very thick] (0.5,10.5) [densely dashed] -- (0.5,12.5);
\draw [blue, very thick] (0.5,12.5) [densely dashed] -- (4.5,12.5);
\draw [blue, very thick] (4.5,12.5) [densely dashed] -- (4.5,8.5);
\draw [blue, very thick] (4.5,8.5) [densely dashed] -- (5.3,8.5);

\draw [fill=gray,opacity=0.25] (0,12) rectangle (1,13);
\draw [fill=gray,opacity=0.25] (1,11) rectangle (2,12);
\draw [fill=gray,opacity=0.25] (2,10) rectangle (3,11);
\draw [fill=gray,opacity=0.25] (3,9) rectangle (4,10);
\draw [fill=gray,opacity=0.25] (4,8) rectangle (5,9);
\draw [fill=gray,opacity=0.25] (5,7) rectangle (6,8);

\draw [fill=gray,opacity=0.25] (0,0) rectangle (1,1);
\draw [fill=gray,opacity=0.25] (1,1) rectangle (2,2);
\draw [fill=gray,opacity=0.25] (2,2) rectangle (3,3);
\draw [fill=gray,opacity=0.25] (3,3) rectangle (4,4);
\draw [fill=gray,opacity=0.25] (4,4) rectangle (5,5);
\draw [fill=gray,opacity=0.25] (5,5) rectangle (6,6);

\draw [color = red] (5.5,13.5) circle (0.7);

\draw [color=blue] (-0.5,5.5) node[scale=0.7]{$0$};
\draw (-0.5,6.5) node[scale=0.7]{$1$};
\draw (-0.5,7.5) node[scale=0.7]{$0$};
\draw [color = red] (-0.5,13.5) node[scale=0.7]{$1$};

\draw [fill=blue,opacity=0.25] (5,5) rectangle (6,7);
\draw [fill=red,opacity=0.25] (5,13) rectangle (6,14);
\draw [fill=red,opacity=0.25] (5,7) rectangle (6,8);
\end{scope}

\begin{scope}[xshift=9cm,yshift=-32cm]

\draw (0,0) rectangle (9,-17);
\draw (0,0) rectangle (9,-1);

\draw (4.5,-0.5) node[scale=0.8]{$j = 7$};

\begin{scope}[xshift=1cm,yshift=-16cm]
\draw (0,0) grid[step=1] (7,14);
\draw (0,0) -- (7,7);
\draw (0,14) [dashed] -- (7,7);
\draw (0.5,0.5) node[scale=1]{$a$};
\draw (0.5,10.5) node[scale=1]{$b$};
\draw (1.5,1.5) node[scale=1]{$a$};
\draw (1.5,3.5) node[scale=1]{$b$};
\draw (2.5,2.5) node[scale=1]{$b$};
\draw (2.5,5.5) node[scale=1]{$a$};
\draw (3.5,4.5) node[scale=1]{$b$};
\draw (3.5,9.5) node[scale=1]{$b$};
\draw (4.5,6.5) node[scale=1]{$a$};
\draw (4.5,12.5) node[scale=1]{$a$};
\draw (5.5,8.5) node[scale=1]{$a$};
\draw (5.5,13.5) node[scale=1]{$b$};
\draw [color = blue] (6.5,11.5) node[scale=1]{$a$};
\draw [color = red] (6.5,7.5) node[scale=1]{$b$};

\draw [fill=gray,opacity=0.25] (0,12) rectangle (1,13);
\draw [fill=gray,opacity=0.25] (1,11) rectangle (2,12);
\draw [fill=gray,opacity=0.25] (2,10) rectangle (3,11);
\draw [fill=gray,opacity=0.25] (3,9) rectangle (4,10);
\draw [fill=gray,opacity=0.25] (4,8) rectangle (5,9);
\draw [fill=gray,opacity=0.25] (5,7) rectangle (6,8);

\draw [fill=gray,opacity=0.25] (0,0) rectangle (1,1);
\draw [fill=gray,opacity=0.25] (1,1) rectangle (2,2);
\draw [fill=gray,opacity=0.25] (2,2) rectangle (3,3);
\draw [fill=gray,opacity=0.25] (3,3) rectangle (4,4);
\draw [fill=gray,opacity=0.25] (4,4) rectangle (5,5);
\draw [fill=gray,opacity=0.25] (5,5) rectangle (6,6);

\draw [color = blue] (6.5,6.5) circle (0.7);
\draw [blue, very thick] (6.5,6.5) [densely dashed] -- (4.5,6.5);
\draw [blue, very thick] (4.5,6.5) [densely dashed] -- (4.5,4.5);
\draw [blue, very thick] (4.5,4.5) [densely dashed] -- (3.5,4.5);
\draw [blue, very thick] (3.5,4.5) [densely dashed] -- (3.5,3.5);
\draw [blue, very thick] (3.5,3.5) [densely dashed] -- (1.5,3.5);
\draw [blue, very thick] (1.5,3.5) [densely dashed] -- (1.5,11.5);
\draw [blue, very thick] (1.5,11.5) [densely dashed] -- (6.3,11.5);

\draw [color = red] (6.5,13.5) circle (0.7);
\draw [red, very thick] (6.5,13.5) [densely dashed] -- (5.5,13.5);
\draw [red, very thick] (5.5,13.5) [densely dashed] -- (5.5,7.5);
\draw [red, very thick] (5.5,7.5) [densely dashed] -- (6.3,7.5);

\draw [color = blue] (-0.5,6.5) node[scale=0.7]{$0$};
\draw [color = red] (-0.5,13.5) node[scale=0.7]{$0$};

\draw [fill=blue,opacity=0.25] (6,6) rectangle (7,7);
\draw [fill=red,opacity=0.25] (6,13) rectangle (7,14);
\end{scope}

\end{scope}

\end{tikzpicture}

 \caption{Computation of $\Phi(f_1) \in \Tab_7$.}
\label{fig:computationPhif1annex}

\end{figure}

\end{appendices} 

\FloatBarrier

\end{document}